\documentclass[11pt]{amsart}

\title{Complete Calabi--Yau metrics on noncompact K3-fibered threefolds}

\author[R. Liang]{Ruiming Liang}
\address[R. Liang]{Peking University}
\email{2201110017@stu.pku.edu.cn}
\date{}

\usepackage{array}
\usepackage{color}
\usepackage{amsmath}
\usepackage{mathrsfs}
\usepackage{mathtools}
\usepackage{overpic}
\usepackage{graphicx}
\usepackage{float}
\usepackage[all]{xy}
\usepackage{pifont}
\usepackage{fancyhdr}

\usepackage{amscd}
\usepackage{amsfonts}
\usepackage{amssymb}
\usepackage{amsthm}
\usepackage{bm}
\usepackage{geometry}
\usepackage{setspace}

\usepackage{booktabs}
\usepackage[labelfont=bf]{caption}
\usepackage{caption}
\usepackage{enumitem}
\usepackage{subfigure}

\usepackage{esint}
\usepackage{tikz-cd}

\usepackage{hyperref}
\hypersetup{
  pdfborder={0 0 0},
  colorlinks=true,
  linkcolor=red,              
  citecolor=green,              
  urlcolor=blue,                
  filecolor=blue,
  linktoc=all,
  pdflinkmargin=0pt,
}

\newtheorem{theorem}{Theorem}[section]   
\newtheorem{lemma}[theorem]{Lemma}       
\newtheorem{corollary}[theorem]{Corollary}
\newtheorem{proposition}[theorem]{Proposition}

\newtheorem{remark}[theorem]{Remark}

	\numberwithin{equation}{section}

\allowdisplaybreaks[4]

\begin{document}

\begin{abstract}

In this article, we construct complete Calabi--Yau metrics on Lefschetz K3 fibrations over $\mathbb{C}$. Our approach involves a gluing construction for a collapsing approximate metric, followed by a perturbation argument. These Calabi--Yau metrics have linear volume growth, unbounded sectional curvature, and are Gromov--Hausdorff asymptotic to a product $\mathbb{R}\times Z$, where $Z$ is a compact singular metric space.

\end{abstract}
\maketitle
\setcounter{tocdepth}{1}
\tableofcontents

\section{Introduction}

Since Yau's seminal work on Ricci curvature \cite{yau1978ricci}, the study of Calabi–Yau metrics has become a central topic in K{\"a}hler geometry, with a major challenge being the construction of complete Calabi–Yau metrics on noncompact K{\"a}hler manifolds. Early on, sporadic concrete examples appeared, such as complete hyperk{\"a}hler metrics constructed via the Gibbons–Hawking ansatz \cite{gibbons1978gravitational}. However, these examples rely on explicit metric expressions and are therefore difficult to generalize to higher-dimensional manifolds. The first systematic investigation of the existence of complete Calabi--Yau metrics on noncompact K\"ahler manifolds was initiated by Tian and Yau in their work \cite{tian1990complete}. Their paper focuses primarily on constructing such metrics on a Fano manifold minus a smooth anti-canonical divisor. It also addresses the case of a non-Fano Calabi--Yau fibration from which a smooth fiber has been removed.

Following their work, Hein considered the construction of such metrics on rational elliptic surfaces (which are not Fano) with certain singular fibers removed in his thesis \cite{hein2010gravitational}. In these cases, the fibration structure is crucial for formulating an explicit metric ansatz. Moreover, the asymptotic behavior of the metric at infinity is determined by the type of the removed singular fiber.

It is natural to ask whether such constructions can be generalized to higher-dimensional fibrations. One attempt in this direction is the work of the author and his collaborator in \cite{liang2025complete}, where such metrics are constructed on abelian fibrations. Wang \cite{wang2022ricci} also considered constructions on isotrivial Calabi–Yau fibrations. A more challenging problem, however, is to construct complete Calabi--Yau metrics on general Calabi--Yau fibrations with varying complex structures, where the geometry at infinity is determined by the removed singular fiber. Unlike the elliptic fibration case, the construction of a metric ansatz in this setting is more implicit, and its properties are far less favorable than those of the semi-flat metric in the elliptic case. 

In this paper, we construct such metrics in the setting of Lefschetz K3 fibrations. A distinctive feature of the Calabi–Yau metric in this setting is that its sectional curvature is expected to be unbounded, a phenomenon absent from the previously cited constructions.

\subsection{Main theorem}

\begin{theorem}[Main]\label{main theorem}

Let $\pi : X \to Y = \mathbb{P}^1$ be a K3-fibered threefold satisfying the following conditions:
\begin{itemize}
    \item $(X,\omega_X)$ is a compact K\"ahler manifold;
    \item near each singular fiber, the fibration is modeled on a standard Lefschetz fibration and admits at worst one nodal singularity per singular fiber;
    \item $K_X = -[F]$, where $F$ denotes a generic fiber.
\end{itemize}

Fix a singular fiber $X_\infty$ and a holomorphic volume form $\Omega$ on $M \coloneqq X \setminus X_\infty$ with a simple pole along $X_\infty$. Then there exists a sufficiently small $\varepsilon_0 > 0$ such that, for every $0 < \varepsilon \le \varepsilon_0$, the manifold $M$ admits a complete Calabi--Yau metric $\omega_{\operatorname{CY},\varepsilon}$ in the K\"ahler class $\left[ \omega_X|_M\right]$, satisfying $\omega_{\operatorname{CY},\varepsilon}^3 = \frac{3\sqrt{-1}}{\varepsilon}\Omega \wedge \overline{\Omega}$ and the following properties:
\begin{itemize}
    \item the sectional curvature is unbounded at infinity; 
    \item the volume growth is linear;
    \item the tangent cone at infinity is the flat half-line $[0,\infty)$;
    \item $\omega_{\operatorname{CY},\varepsilon}$ is Gromov--Hausdorff asymptotic to $\mathbb{R}\times Z$, where $Z$ is a compact singular metric space naturally associated to the removed singular fiber.
\end{itemize}
\end{theorem}

\begin{remark}

The existence of such a Lefschetz K3 fibration over $\mathbb{P}^1$ is guaranteed by Proposition 2.14 in \cite{kovalev2005coassociative}. Concretely, begin with a generic smooth Fano threefold $V$ and consider two generic divisors $D_1, D_2 \in |-K_V|$, i.e., two K3 surfaces. Their intersection $C = D_1 \cap D_2$ is a curve in $V$. Blowing up $C$ yields the aforementioned Lefschetz K3 fibration model. Since the proof in \cite{kovalev2005coassociative} is a bit sketchy, we would like to give a more detailed construction in Appendix A.

\end{remark}

\begin{remark}

The above theorem answers an open question raised by Haskins--Hein--Nordstr\"om in the concluding remarks of \cite{haskins2015asymptotically}: namely, there exist complete Calabi--Yau manifolds of linear volume growth that are Gromov--Hausdorff asymptotic to $\mathbb{R} \times Z$, with $Z$ a singular space. Furthermore, it demonstrates that the local curvature bound assumed in Theorem~A of \cite{yan2025uniqueness} is indeed necessary.

\end{remark}

\begin{remark}

A notable feature of the metric $\omega_{\operatorname{CY},\varepsilon}$ is that its sectional curvature is unbounded. To the best of the author's knowledge, a complete noncompact Calabi--Yau metric with unbounded sectional curvature has not been specified in the existing literature. Of course, using the method in \cite{anderson1989complete}, one could consider the Taub--NUT spaces with infinitely many marked points $\{p_1, p_2, \dots\}$ and choose their positions appropriately — for instance, letting $p_i$ and $p_{i+1}$ be mutually closer while $p_i$ tends to infinity — which should also yield a complete noncompact Calabi--Yau metric with unbounded sectional curvature. However, the resulting manifold has infinite topological type; in particular, it cannot be realized as a compact K\"ahler manifold minus a divisor.

\end{remark}

\subsection{Obstructions and strategy}

In our situation, direct construction of an explicit Calabi--Yau metric is hard. We adopt the current mainstream approach, namely, first constructing an approximate metric and then solving the complex Monge--Amp\`ere equation. However, the unboundedness of the sectional curvature poses a serious difficulty, as the Tian--Yau--Hein package (first shown in \cite{tian1990complete} and later refined in \cite{hein2010gravitational}) can no longer be applied. The essential reason is that most existing methods for solving the noncompact complex Monge--Amp\`ere equation aim to follow Yau's approach for the compact case, which requires a uniform bound on the $L^\infty$ norm of the sectional curvature when deriving $C^2$ estimates. Therefore, we need to find a different path.

Rather than deriving global nonlinear a priori estimates, we treat the complex Monge--Amp\`ere equation as a small perturbation of its linearization. More precisely, we aim to solve the equation by a contraction-mapping argument on suitable weighted Banach spaces, with the Poisson equation providing the essential linear input. For this strategy to work, two quantitative estimates are required: the Ricci potential of the approximate metric must be sufficiently small, and the linearized operator must admit a sufficiently precise right inverse. Perturbative constructions of this type have appeared, for instance, in the Kummer constructions of \cite{donaldson2010calabi,Biquard2011Kummer}, in the collapsing fibration construction of \cite{li2019gluing}, and in the related work \cite{szekelyhidi2019degenerations}.

To obtain the required smallness, we exploit the fibration structure and introduce a collapsing parameter $\varepsilon$. Collapsing suppresses the contribution from the fibers and makes the Ricci potential small even in the finite region. It also creates a new difficulty, however: one must control the inverse of the Laplacian uniformly as $\varepsilon\to 0$. This issue is less severe in the classical collapse of elliptically fibered K3 surfaces studied in \cite{gross2000large}. There, the Ricci potential of the approximate metric is exponentially small in the collapsing parameter, whereas the norm of the inverse Laplacian grows at most polynomially, of order $O(\varepsilon^{-1})$. Thus a relatively coarse estimate for the inverse is already sufficient to close the perturbation argument. In the case of K3 fibrations, and more generally Calabi--Yau fibrations with higher-dimensional fibers, the Ricci potential associated with the semi-Ricci-flat approximation is only of order $O(\varepsilon)$. Crude estimates for the inverse are therefore no longer adequate.

We address this problem by carrying out a global gluing construction on the total space of the fibration. The manifold is divided into several regions according to their geometric properties, and suitable weighted H\"older spaces are introduced on each of them. The local inverses of the Laplacian are then constructed separately and patched together to produce a global parametrix. The resulting estimate is uniform in $\varepsilon$ at the complex-Hessian level: formally, we obtain a uniform bound for $2\sqrt{-1}\partial\bar{\partial}
\Delta_{\varepsilon}^{-1}$, rather than for the potential-level inverse $\Delta_{\varepsilon}^{-1}$ itself. This is precisely the estimate required by the nonlinear perturbation argument. An important reason why this construction succeeds is that the collapsing makes every cutoff error produced by the gluing procedure $\varepsilon$-small.

The main novelty of the paper lies in the analysis of the noncompact end. Near the removed nodal fiber $X_\infty$, the generalized K\"ahler--Einstein metric pulled back from the base contains an additional $|y|^{-2}$ factor. Consequently, the base geometry becomes asymptotically cylindrical while the Ricci-flat K3 fibers simultaneously degenerate towards the singular orbifold metric on $X_\infty$. This behavior is fundamentally different from the finite Lefschetz fibration region considered by Yang Li \cite{li2019gluing}, and requires several new analytic ingredients:

\begin{itemize}

\item First, the most notable difference is that the semi-Ricci-flat approximation at infinity has favorable properties. Although the sectional curvature is unbounded, we can still obtain a quite decent decay of the Ricci potential, as provided by Proposition~\ref{Ricci potential decay 1}. This is in sharp contrast with the compact collapsing situation studied in \cite{li2019gluing}, where the semi-Ricci-flat approximation breaks down at the quantum scale and has to be replaced by the exotic Calabi--Yau metric on $\mathbb C^3$.

\item Second, the asymptotically cylindrical geometry alone is not sufficient to apply the standard elliptic theory, because the cross-sections themselves become increasingly singular as one approaches the removed fiber. We therefore introduce weighted H\"older spaces adapted simultaneously to the cylindrical direction and to the distance from the vanishing cycle. The resulting uniform mapping theorem for the Laplacian is established in Proposition~\ref{Laplace for C times K3 weighed version}.

\item Third, the local Green operators constructed in different geometric regions must be patched together into a single global inverse. Since one has to pass continuously from regions uniformly away from the singular fiber to regions where the fibers degenerate, two different kinds of harmonic analysis have to be coupled. The key compatibility estimate is provided by Lemma~\ref{approximate invert 1}, which allows the construction of a global parametrix.

\item Finally, besides the fiberwise analysis, one must also solve the Poisson equation on the base $\mathbb C$, whose geometry is asymptotically cylindrical. Here we apply the Fredholm theory of Lockhart--McOwen \cite{lockhart1985elliptic}. Combining the Green operator on the base with the local inverses constructed above yields the global right inverse in Proposition~\ref{inverse of Laplacian}, from which the perturbation argument follows.

\end{itemize}

\subsection{Organization of the paper}The paper is organized as follows.

In Section \ref{Construction of the gluing ansatz}, we begin by giving a rough description of the ansatz. We then review several model metrics on specific geometric models to clarify the gluing construction. We conclude this section by establishing the well-definedness of the ansatz.

In Section \ref{Harmonic analysis on model spaces}, we further investigate the geometric properties of the model metrics introduced in Section \ref{Construction of the gluing ansatz}, with a focus on their Laplacian operators. This analysis motivates the definition of a global weighted norm on $M$, which is carried out in Section \ref{Global weighted norms and deviation estimates}. It also aids in constructing a global approximate inverse for the Laplacian of the gluing ansatz, presented in Section \ref{Decomposition, patching and parametrix}.

In Section \ref{Global weighted norms and deviation estimates}, we define a global weighted norm on $M$. We show that, under this norm, the Ricci potential of the gluing ansatz is small. Additionally, we compare the gluing ansatz with the model metrics in their respective regions and control their differences.

In Section \ref{Decomposition, patching and parametrix}, we invert the Laplacian associated to the ansatz by constructing a parametrix for the Green operator via decomposition and patching techniques.

In Section \ref{Perturbation to Calabi--Yau metric}, we use the Banach fixed point theorem to perturb the gluing ansatz into a genuine complete Calabi–Yau metric. A contradiction argument then demonstrates that the resulting complete Calabi--Yau metric still possesses unbounded sectional curvature at infinity.

In Section \ref{Further discussion}, we present further discussion.

In Appendix \ref{Abundance of examples}, we provide a concrete construction of our geometric model and argue that it is generic in the algebro-geometric sense.

\subsection{Acknowledgments} The author would like to express his sincere gratitude to Professor Gang Tian for his encouragement and guidance. The author is deeply grateful to Professor Hans-Joachim Hein for enlightening discussions and comments on this paper. The author acknowledges funding from the PhD Student's Short-term Overseas Research Program of the Graduate School of Peking University. The author is also funded by the Deutsche Forschungsgemeinschaft (DFG, German Research Foundation) under Germany's Excellence Strategy EXC 2044/2–390685587 – ``Mathematics Münster: Dynamics–Geometry–Structure".

\section{Construction of the gluing ansatz}\label{Construction of the gluing ansatz}

\subsection{Rough description of the ansatz}\label{Rough description of the ansatz}For the sake of completeness and rigor, we first review the setting and fix the notation that will be used frequently. Some of the following content parallels that in \cite{li2019gluing} and \cite{spotti2014deformations}, but we will make adjustments to fit our setting.

Let $(X,\omega_X)$ be a compact K{\"a}hler threefold with a Lefschetz K3 fibration structure $\pi:X\to Y=\mathbb{P}^1$ such that $K_X=-[F]$, where $F$ is a generic fiber. We assume there is only one nodal singularity on each singular fiber. Without loss of generality, we can impose the normalization condition $\int_{X_y}\omega_X^2=1$. 

For algebro-geometric reasons, a fibration over $\mathbb{P}^1$ cannot have only one singular fiber. Consequently, we must inevitably discuss the metric behavior near the singular fibers in the finite region. Since the local constructions near different finite singular fibers are completely identical, we shall, purely for notational simplicity, pretend that there are only two singular fibers: one finite singular fiber $X_0$, and one singular fiber $X_\infty$ that is removed. The general case is obtained by repeating the same construction near each finite singular fiber and summing the resulting local contributions. Let $p_0,p_\infty\in \mathbb{P}^1$ be the base points corresponding to $X_0$ and $X_\infty$, and set $S = \{p_0, p_\infty\}$.

Let $\Omega\in H^0(X,K_X\otimes\mathcal{O}_X(X_\infty))$ be a nontrivial holomorphic section with $\operatorname{div}(\Omega)=-X_\infty$. Its restriction to $M$ is a holomorphic volume form.

The generalized K{\"a}hler-Einstein metric $\omega_B$ on $B=Y\backslash\{p_\infty\}\cong \mathbb{C}$ is defined as the pushforward of $\sqrt{-1}\Omega\wedge\overline{\Omega}$ to the base $B$. By definition $\pi_*(\sqrt{-1}\Omega\wedge\overline{\Omega})$ is the unique K{\"a}hler current on $B$ that satisfies
\begin{align*}
\int_{B}\psi\cdot\pi_*(\sqrt{-1}\Omega\wedge\overline{\Omega})=\int_{M}\pi^* \psi\cdot\sqrt{-1}\Omega\wedge\overline{\Omega} 
\end{align*}for all compactly supported smooth functions $\psi$ on $B$. This is globally well-defined because of duality theory in topology.

Near the singular fiber $X_0$, we choose a local coordinate $(W_0, t)$ on $\mathbb{P}^1$ with ${t = 0}$ corresponding to $X_0$; near $X_\infty$, we similarly choose a local coordinate $(W_\infty, y)$ on $\mathbb{P}^1$ with ${y = 0}$ corresponding to $X_\infty$.

In the coordinate $(W_0, t)$, we can write $\Omega$ as $\Omega=dt\wedge \Omega_t$, where $\Omega_{t}$ is a holomorphic volume form on the fiber $X_{t}$. Set $A_t\coloneqq \int_{X_t}\Omega_t\wedge \overline{\Omega}_t$. Then in this coordinate, $\omega_B=\sqrt{-1}A_tdt\wedge d\bar{t}$. By Lemma 2.1 in \cite{li2019gluing}, we know $A_t=\text{smooth terms}-2|g_1(t)|^2|t|$, where $g_1(t)$ is a holomorphic function in $t$. Hence $A_t$ is a Lipschitz function.

In the coordinate $(W_\infty,y)$, we can write $\Omega$ as $\Omega=\frac{dy}{y}\wedge \Omega_{y}$, where $\Omega_{y}$ is a holomorphic volume form on the fiber $X_{y}$. Set $A_y\coloneqq \int_{X_y}\Omega_y\wedge \overline{\Omega}_y$ and $\widetilde{A}_y\coloneqq\frac{1}{|y|^2}A_y$. Then in this coordinate,  $\omega_B=\sqrt{-1}\widetilde{A}_ydy\wedge d\bar{y}$. By Lemma 2.1 in \cite{li2019gluing}, we know $A_y=\text{smooth terms}-2|g_2(y)|^2|y|$, where $g_2(y)$ is a holomorphic function in $y$. Thus, $A_y$ is a Lipschitz function. Although $A_y$ is not an invariant quantity, its value at $y=0$ is independent of the choice of coordinates: suppose $y=f(w)$, where $w$ is another local coordinate near $p_\infty$ such that ${w=0}$ corresponds to the singular fiber $X_{\infty}$. Since $\omega_B$ is a geometric quantity, we have
\begin{align*}
\sqrt{-1}\frac{A_w}{|w|^2}dw\wedge d\bar{w}=\sqrt{-1}\frac{A_y}{|y|^2}dy\wedge d\bar{y}=\sqrt{-1}\frac{A_y}{|f(w)|^2}|f'(w)|^2dw\wedge d\bar{w},
\end{align*}and hence $A_w=A_y\frac{|w|^2}{|f(w)|^2}|f'(w)|^2$. By L'H\^{o}pital's rule, we know $A_{y}|_{y=0}=A_{w}|_{w=0}$. We denote this number by $A_\infty$. By multiplying $\Omega$ by a nonzero constant, we may assume $A_\infty=\frac12$.

We can solve the complex Monge--Amp\`ere equation fiberwise on $X$:
\begin{equation}
\left\{
	\begin{aligned}
&\left(\omega_X+\sqrt{-1}\partial\bar{\partial}\psi_y\right)\big{|}_{X_y}^2=A_y^{-1}\Omega_y\wedge \overline{\Omega}_y \\
   &  \int_{X_y}\psi_y\omega_X^2=0
    \end{aligned}
	\right.
\end{equation}
Note that $A_y^{-1}\Omega_y\wedge \overline{\Omega}_y$ is invariant under change of coordinates, hence is globally defined. Thus the above equation makes sense for all $p\in Y=\mathbb{P}^1$. For those $y\in Y\backslash S$, the equation is just the usual complex Monge--Amp\`ere equation on a smooth compact K{\"a}hler manifold, which can be solved by Yau's theorem \cite{yau1978ricci}. For those $y\in S$, we can still solve the equation in the orbifold sense, see \cite{eyssidieux2009singular}.

Write 
\begin{equation}\label{fiber SRF}
\omega_{\operatorname{SRF}}|_{X_y}=\omega_X|_{X_y}+\sqrt{-1}\partial\bar{\partial}\psi_y.
\end{equation}
We define the global semi-Ricci-flat form $\omega_{\operatorname{SRF}}^{\varepsilon}$ on $M$ as
\begin{equation}\label{SRF}
\omega_{\operatorname{SRF}}^{\varepsilon}=\frac{1}{\varepsilon}\pi^*\omega_B+\omega_X+\sqrt{-1}\partial\bar{\partial}\psi.
\end{equation}

To prevent confusion, let $\psi_t^{(0)}$ be the potential on $X_t$ for $t$ near $p_0$, and $\psi_y^{(\infty)}$ the potential on $X_y$ for $y$ near $p_\infty$. In particular, we denote the orbifold potential on $X_0$ by $\psi_0^{(0)}$, and that on $X_\infty$ by $\psi^{(\infty)}_0$. In most cases, the coordinate we are working with is clear from the context, so for brevity we will simply write $\psi_t$ and $\psi_y$.

By a standard application of the implicit function theorem, the family $\{\psi_y\}$ varies smoothly in the $y$ variable when $y\in Y\backslash S$ and hence gives rise to a global smooth function $\psi$ on $\pi^{-1}(Y\backslash S)$ such that $\psi|_{X_y}=\psi_y$.

Near each nodal point of each singular fiber, the fibration is locally  (up to a local biholomorphism) modeled on $\{(z_1,z_2,z_3,y)\in \mathbb{C}^4|y=z_1^2+z_2^2+z_3^2\}$, which is the standard Lefschetz fibration. To simplify notation, we write $(U_0, (z_1, z_2, z_3, t))$ for the coordinate system around $q_0$ and $(U_\infty, (z_1, z_2, z_3, y))$ for that around $q_\infty$. 

To align with the conventions in Yang Li's works \cite{li2019gluing,li2019new}, we adopt the following notation in the local coordinates of each singular fiber: let $H = r^4 = |z_1|^2 + |z_2|^2 + |z_3|^2$, and $\eta = |t|^2$ in the $t$-chart (or $\eta=|y|^2$ in the $y$-chart). For convenience, we extend the function $r$, originally defined on $U_0$ (or $U_\infty$), smoothly to the entire neighborhood $\pi^{-1}(\pi(U_0))$ (or $\pi^{-1}(\pi(U_\infty))$) so that $r$ is of order $1$ outside the coordinate neighborhood $U_0$ (or $U_\infty$).

By adjusting the coordinates, we can further assume that $\psi_0^{(0)}=c_0+r^2+O(r^4)$, $\psi^{(\infty)}_0=c_\infty+r^2+O(r^4)$. Here the constants $c_0,c_\infty$ are chosen to satisfy the normalization condition.

By a direct computation, in the coordinate neighborhood $(U_0,(t,z_1,z_2,z_3))$ we obtain
\begin{align*}
\sqrt{-1}\Omega\wedge \overline{\Omega}\Big{|}_{z=0}=A_0\prod\limits_{i=1}^3\sqrt{-1}dz_i\wedge d\bar{z}_i.
\end{align*}
Hence, up to a $U(1)$-valued constant, the holomorphic volume form admits a local expression in $U_0$ of the form $\Omega=\sqrt{A_0}(1+O(z))\,dz_1\wedge dz_2\wedge dz_3$. Analogously, in the coordinate neighborhood $(U_\infty,(y,z_1,z_2,z_3))$ we have
\begin{align*}
\sqrt{-1}|y|^2\Omega\wedge \overline{\Omega}\Big{|}_{z=0}=A_\infty\prod\limits_{i=1}^3\sqrt{-1}dz_i\wedge d\bar{z}_i.
\end{align*}
Consequently, up to a $U(1)$-valued constant, the holomorphic volume form is locally expressed in $U_\infty$ as $\Omega=\frac{\sqrt{A_\infty}}{y}(1+O(z))\,dz_1\wedge dz_2\wedge dz_3$.

Suppose that in a local coordinate chart $U_i$ around the node $q_i$, the K\"ahler form $\omega_X$ admits a potential $\Theta_i$, for $i = 0,\infty$. Moreover, we may assume that within this coordinate neighborhood, the potential satisfies $\Theta_i = O(r^4)$.

We refer to these assumptions as Setting (A). All subsequent discussions of Theorem \ref{main theorem} will be conducted under this setting.

Define a partition of unity $\{\chi_0,\chi_1,...,\chi_N,\chi_\infty\}$ on the base $\mathbb{P}^1$ such that:
\begin{itemize}
\item $\chi_0=1$ on $\left\{|t|<\varepsilon^{\frac{6}{14+\tau}}\right\}$ and the support of $\chi_0$ is contained in $\left\{|t|<2\varepsilon^{\frac{6}{14+\tau}}\right\}$. Here $-2<\tau<0$ is a number to be chosen;
\item $\chi_\infty=1$ on $\left\{|y|<\varepsilon^{\frac{12}{2+\mu}}\right\}$ and the support of $\chi_\infty$ is contained in $\left\{|y|<2\varepsilon^{\frac{12}{2+\mu}}\right\}$. Here $-2<\mu<0$ is a number to be chosen;
\item For $1 \le i \le N$, the supports of $\chi_i$ are contained in the complement of $\left\{|t| \le \varepsilon^{\frac{6}{14+\tau}}\right\}\cup \left\{|y|\le \varepsilon^{\frac{12}{2+\mu}}\right\}$, each having length scale $O(\varepsilon^{\frac12})$ with respect to the $\omega_B$ metric, containing a point $y_i$ as the center of that support. We can demand that $0 \le \chi_i \le 1$, and all these $\chi_i$ have uniform $C^k$ bounds with respect to the metric $\frac{1}{\varepsilon}\omega_B$ for any $k\in\mathbb{N}$. Moreover, at each point in $Y$ the number of non-vanishing $\chi_i$ is uniformly bounded independent of $\varepsilon$ and the position of that point, even though $N = O\left(-\frac{12}{2+\mu}\varepsilon^{-1}\log \varepsilon\right)$.
\end{itemize}

Define two smooth cutoff functions on $\mathbb{R}$ as follows. Choose a smooth function $\gamma_1(s)$ such that $\gamma_1(s) = 0$ on $(-\infty, 1]$, $\gamma_1(s) = 1$ on $[2, \infty)$, and $0 \le \gamma_1(s) \le 1$ for all $s \in \mathbb{R}$. Then set $\gamma_2(s) = 1 - \gamma_1(s)$.

We consider the following gluing ansatz:
\begin{equation}\label{Ansatz 1}
\begin{aligned}
\omega_\varepsilon&=\frac{1}{\varepsilon}\pi^*\omega_B+ \pi^*\beta+ \omega_X \\
&\quad + \sqrt{-1}\partial\bar{\partial}\left(\chi_0\left\{c_0+\gamma_1\left(\frac{r}{\varepsilon^{\frac{1}{10}}+\varepsilon^{\frac{1}{12}}\rho'^{\frac16}}\right)G_0^*\left(\psi_0^{(0)}-c_0\right)+\gamma_2\left(\frac{r}{\varepsilon^{\frac{1}{10}}+\varepsilon^{\frac{1}{12}}\rho'^{\frac16}}\right)\left(\frac{\varepsilon}{2A_0}\right)^{\frac13}\phi''_{\mathbb{C}^3}\right\}\right) \\
&\quad + \sqrt{-1}\partial\bar{\partial}\left(\chi_\infty\left\{c_\infty+\gamma_1\left(\frac{H}{|y|^{\frac56}}\right)G_\infty^*\left(\psi_0^{(\infty)}-c_\infty\right)+\gamma_2\left(\frac{H}{|y|^{\frac56}}\right)\left(\sqrt{H+|y|}-\Theta_\infty\right)\right\}\right)\\
&\quad + \sqrt{-1}\partial\bar{\partial}\left(\sum\limits_{i=1}^N\chi_iG_{y_i}^*\psi_{y_i}\right),
\end{aligned}
\end{equation}
where $\Theta_{\infty}$ is the K{\"a}hler potential of $\omega_X$ in the coordinate neighborhood $U_\infty$ ; $G_0$ and $G_\infty$ are local diffeomorphisms around the singular fibers; $G_{y_i}$ is a local diffeomorphism around the smooth fiber $X_{y_i}$. In the subsequent subsections, these maps will be carefully constructed so that every pullback is defined on the support of the corresponding cutoff; $\rho'=\sqrt{H^{\frac12}+\frac{2A_0|t|}{\varepsilon}}$ is an appropriate weight function near the node $q_0\in X_0$; $\phi_{\mathbb{C}^3}''$ is a potential related to Yang Li's exotic metric $\omega_{\mathbb{C}^3}$ ; $\beta$ is a bump $(1,1)$-form on $B$ to guarantee that $\omega_\varepsilon$ satisfies an integration condition.

For simplicity, we write
\begin{equation}\label{ansatz potential}
\begin{aligned}
    \phi\coloneqq &\chi_0\left\{c_0+\gamma_1\left(\frac{r}{\varepsilon^{\frac{1}{10}}+\varepsilon^{\frac{1}{12}}\rho'^{\frac16}}\right)G_0^*\left(\psi_0^{(0)}-c_0\right)+\gamma_2\left(\frac{r}{\varepsilon^{\frac{1}{10}}+\varepsilon^{\frac{1}{12}}\rho'^{\frac16}}\right)\left(\frac{\varepsilon}{2A_0}\right)^{\frac13}\phi''_{\mathbb{C}^3}\right\}\\
&+\chi_\infty\left\{c_\infty+\gamma_1\left(\frac{H}{|y|^{\frac56}}\right)G_\infty^*\left(\psi_0^{(\infty)}-c_\infty\right)+\gamma_2\left(\frac{H}{|y|^{\frac56}}\right)\left(\sqrt{H+|y|}-\Theta_\infty\right)\right\}\\
&+\sum\limits_{i=1}^N\chi_iG_{y_i}^*\psi_{y_i}.
\end{aligned}
\end{equation}
Then 
\begin{align*}
\omega_\varepsilon=\frac{1}{\varepsilon}\pi^*\omega_B+ \pi^*\beta+ \omega_X+\sqrt{-1}\partial\bar{\partial}\phi.
\end{align*}

We can naturally extend $\gamma_1\left(\frac{H}{|y|^{\frac56}}\right)=1$ and $\gamma_2\left(\frac{H}{|y|^{\frac56}}\right)=0$ on the central fiber $y=0$. Then $\phi$ is smooth on $X\backslash \{q_\infty\}$.

To rigorously explain every term in the ansatz and obtain its geometric properties in the corresponding regions, we have to recall some model metrics, including the Eguchi–Hanson metric, the exotic metric $\omega_{\mathbb{C}^3}$ on $\mathbb{C}^3$ constructed in \cite{li2019new} and the semi-Ricci-flat metric on a K3 fibration. A rigorous introduction to these geometric models will be given in the following subsections. The global well-definedness and positivity of the resulting form will be proved in Proposition \ref{Well-definedness of the metric ansatz}.

Nevertheless, we first describe heuristically the geometry of the ansatz in different regions:
\begin{itemize}
    \item In the region $\{|t| \ge 2\varepsilon^{\frac{6}{14+\tau}}\}\cap \{|y|\ge 2\varepsilon^{\frac{12}{2+\mu}}\}$, we are far from the singular fibers, and the metric is simply 
\begin{align*}
\omega_\varepsilon = \frac{1}{\varepsilon}\pi^*\omega_B+\omega_X  + \sqrt{-1}\partial\bar{\partial}\left(\sum\limits_{i=1}^N\chi_iG_{y_i}^*\psi_{y_i}\right).
\end{align*}This is approximately the semi-Ricci-flat metric.

\item In the region $\left\{\varepsilon^{\frac{6}{14+\tau}} \le |t| \le 2\varepsilon^{\frac{6}{14+\tau}}\right\}$, the metric $\omega_\varepsilon$ starts to receive contribution from the nodal fiber $X_0$, but the fluctuation effect of $\omega_{\mathbb{C}^3}$ is not yet significant. 

\item In the region $\left\{|t| < \varepsilon^{\frac{6}{14+\tau}},r > \varepsilon^{\frac{1}{10}} + \varepsilon^{\frac{1}{12}}\rho'^{\frac{1}{6}}\right\}$, the metric $\omega_\varepsilon$ is essentially
\begin{align*}
     \omega_\varepsilon \sim \frac{1}{\varepsilon}\pi^*\omega_B+\omega_X  + \sqrt{-1}\partial\bar{\partial} G_0^* \psi_0,
\end{align*}
    which is approximately the product metric on $ \mathbb{C}\times X_0$.

\item In the region $\left\{|t| < \varepsilon^{\frac{6}{14+\tau}},r< \varepsilon^{\frac{1}{10}} + \varepsilon^{\frac{1}{12}}\rho'^{\frac{1}{6}}\right\}$, the metric $\omega_\varepsilon$ is 
\begin{align*}
     \omega_\varepsilon \sim \frac{1}{\varepsilon}\pi^*\omega_B+\omega_X  + \left(\frac{\varepsilon}{2A_0}\right)^{\frac13}\sqrt{-1}\partial\bar{\partial}\phi''_{\mathbb{C}^3}.
\end{align*}
In this region, the semi-Ricci-flat approximation breaks down completely, so we have to glue a scaled exotic metric $\omega_{\mathbb{C}^3}$.

\item In the region $\left\{ |y| \le 2\varepsilon^{\frac{12}{2+\mu}}\right\}$, the metric is essentially
\begin{align*}
     \omega_\varepsilon =& \frac{\sqrt{-1}A_y}{\varepsilon|y|^2}dy\wedge d\bar{y}+\omega_X\\
     &+  \sqrt{-1}\partial\bar{\partial}\left( c_\infty+\gamma_1\left(\frac{H}{|y|^{\frac56}}\right)G_\infty^*\left(\psi_0^{(\infty)}-c_\infty\right)+\gamma_2\left(\frac{H}{|y|^{\frac56}}\right)\left(\sqrt{H+|y|}-\Theta_\infty\right) \right).
\end{align*}This is essentially the semi-Ricci-flat metric. While its sectional curvature is unbounded, it admits a good estimate on the Ricci potential.
    
\end{itemize}

\subsection{Eguchi--Hanson metric}\label{Eguchi Hanson metric}In this subsection, we review the Eguchi--Hanson metric in a separate setting. 

Let $\pi:\mathbb{C}^3\to \mathbb{C},\; (z_1,z_2,z_3)\mapsto y=z_1^2+z_2^2+z_3^2$ be the standard Lefschetz fibration. Let $H=\sum\limits_{i=1}^3|z_i|^2$, $\eta=|y|^2$ and $\operatorname{Vol}_E=-\sqrt{-1}dz_1\wedge d\bar{z}_1\wedge dz_2\wedge d\bar{z}_2\wedge dz_3\wedge d\bar{z}_3$ be the Euclidean volume form on $\mathbb{C}^3$. Write $V=\{(z_1,z_2,z_3,y)\in \mathbb{C}^4|y=z_1^2+z_2^2+z_3^2\}\cong \mathbb{C}^3$, and $V_y=\pi^{-1}(y)\subset \mathbb{C}^3=V$.

The Eguchi--Hanson metric is given by the potential $\phi=F(H,\eta)=\sqrt{H+\eta^{\frac12}}$. Write $\omega_{\operatorname{EH}}\coloneqq \sqrt{-1}\partial\bar{\partial}\phi$ and $\operatorname{EH}_y=\omega_{\operatorname{EH}_y}\coloneqq \left(\sqrt{-1}\partial\bar{\partial}\phi\right)\Big{|}_{V_y}$.

\begin{lemma}\label{K{\"a}hler of ST}

$\phi$ is a strictly plurisubharmonic function on $\mathbb{C}^3\backslash\{y=0\}$. In particular, $\omega_{\operatorname{EH}}$ is indeed a K{\"a}hler metric on $\mathbb{C}^3\backslash\{y=0\}$. Furthermore, $\operatorname{EH}_y$ is a complete Calabi--Yau metric of ALE type.

\end{lemma}

\begin{proof}

By direct computation we have
\begin{align*}
&\frac{\left(\sqrt{-1}\partial\bar{\partial}\sqrt{H+|y|}\right)^3}{\operatorname{Vol}_E}=\frac{3}{8 |y| \sqrt{H+|y|}}>0;\\
&\frac{\sqrt{-1}\partial\bar{\partial}H\wedge\left(\sqrt{-1}\partial\bar{\partial}\sqrt{H+|y|}\right)^2}{\operatorname{Vol}_E}=\frac{3}{4|y|}>0;\\
&\frac{(\sqrt{-1}\partial\bar{\partial}H)^2\wedge\sqrt{-1}\partial\bar{\partial}\sqrt{H+|y|}}{\operatorname{Vol}_E}=\frac{H}{|y| \sqrt{H+|y|}}+\frac{2}{\sqrt{H+|y|}}>0.
\end{align*}Hence $\sqrt{-1}\partial\bar{\partial}\sqrt{H+|y|}$ is a K{\"a}hler metric on $\{y\neq 0\}$.

For the fiberwise Calabi--Yau property, see Section 2.2 in \cite{li2019new}.

\end{proof}

We are primarily interested in the behavior of the Eguchi--Hanson metric in the neighborhood $U = \{H < 1\}$ of the origin.

To better understand the metric, we calculate the metric components of $\omega_{\operatorname{EH}}$ in the local coordinate $y,u=z_1,v=z_2$. From now on, we restrict ourselves to the domain 
\begin{align*}
D&=\left\{(z_1,z_2,z_3,y)\in V\left||y|<1,\;H<1,\;  |z_3|>\frac12\max\{|z_1|,|z_2|\}\right.\right\}\\
&\cong \left\{(y,u,v)\in\mathbb{C}^3\left||y|<1,\;H<1,\;|y-u^2-v^2|^{\frac12}>\frac12\max\{|u|,|v|\}\right.\right\}.
\end{align*}Set $A=y-u^2-v^2=z_3^2$, $L=\sqrt{H+|y|}=\sqrt{|y|+|u|^2+|v|^2+|A|}$. In this region, $|A|= O(H)$ and $L= O(H^{\frac12})$. 

Later, we will estimate the magnitudes of various geometric quantities. Since these quantities are geometric in nature, i.e., independent of the choice of coordinates, the estimates obtained in $D$ actually hold for all other regions as well. The choice of the region $D$ and the corresponding coordinate system above is purely for the purpose of performing crude bounds when estimating the norms.

\begin{lemma}\label{first and second order derivatives}

(1) The first-order derivative is computed as follows:
\begin{align*}
&\frac{\partial A}{\partial y}=1,\; \frac{\partial A}{\partial \bar{y}}=0,\; \frac{\partial |A|}{\partial y}=\frac{\bar{A}}{2|A|},\; \frac{\partial |A|}{\partial \bar{y}}=\frac{A}{2|A|};\\
&\frac{\partial A}{\partial u}=-2u,\; \frac{\partial A}{\partial \bar{u}}=0,\; \frac{\partial |A|}{\partial u}=-\frac{\bar{A}u}{|A|},\; \frac{\partial |A|}{\partial \bar{u}}=-\frac{A\bar{u}}{|A|};\\
&\frac{\partial L}{\partial y}=\frac{1}{2L}\left(\frac{\bar{y}}{2|y|}+\frac{\bar{A}}{2|A|}\right), \; \frac{\partial L}{\partial \bar{y}}=\frac{1}{2L}\left(\frac{y}{2|y|}+\frac{A}{2|A|}\right);\\
&\frac{\partial L}{\partial u}=\frac{1}{2L}\left(\bar{u}-\frac{\bar{A}u}{|A|}\right), \; \frac{\partial L}{\partial \bar{u}}=\frac{1}{2L}\left(u-\frac{A\bar{u}}{|A|}\right).
\end{align*}

(2) The second-order derivatives can be calculated by the following rule:
\begin{align*}
\partial_i\partial_{\bar{j}}L=\frac{\partial_i\partial_{\bar{j}}(H+|y|)}{2L}-\frac{\partial_i(H+|y|)\partial_{\bar{j}}(H+|y|)}{4L^3}.
\end{align*}Hence
\begin{align*}
&\partial_y\partial_{\bar{y}}L=\frac{1}{8L}\left(\frac{1}{|y|+\frac{1}{|A|}}\right)-\frac{1}{4L^3}\left(\frac{\bar{y}}{2|y|}+\frac{\bar{A}}{2|A|}\right)\left(\frac{y}{2|y|}+\frac{A}{2|A|}\right);\\
&\partial_y\partial_{\bar{u}}L=-\frac{\bar{u}}{4|A|L}-\frac{1}{4L^3}\left(\frac{\bar{y}}{2|y|}+\frac{\bar{A}}{2|A|}\right)\left(u-\frac{A\bar{u}}{|A|}\right);\\
&\partial_u\partial_{\bar{u}}L=\frac{1+\frac{|u|^2}{|A|}}{2L}-\frac{1}{4L^3}\left(\bar{u}-\frac{\bar{A}u}{|A|}\right)\left(u-\frac{A\bar{u}}{|A|}\right);\\
&\partial_u\partial_{\bar{v}}L=\frac{u\bar{v}}{2L|A|}-\frac{1}{4L^3}\left(\bar{u}-\frac{\bar{A}u}{|A|}\right)\left(v-\frac{A\bar{v}}{|A|}\right).
\end{align*}

\end{lemma}

\begin{proof}

By direct computation.

\end{proof}

We now consider the K\"ahler metric $\omega_1 \coloneqq \omega_{\mathbb D^*} + \omega_{\operatorname{EH}}$ on $\pi^{-1}(\mathbb D^*) \cap B(0,1) \subset \mathbb C^3$, where $\mathbb D^* = \{y \in \mathbb C : 0 < |y| < 1\}$ is the punctured unit disc in $\mathbb{\mathbb{C}}$, $\omega_{\mathbb{D}^*}=\frac{\sqrt{-1}}{|y|^2}dy\wedge d\bar{y}$ is the cylindrical metric on $\mathbb{D}^*$ and $B(0,1)$ is the unit ball in $\mathbb{C}^3$.

In the region $D$, the coordinates $(y, u, v)$ provide a local holomorphic trivialization of the Lefschetz fibration. Using this trivialization, extend the fiber metric $g_{\operatorname{EH},y}$ constantly in the base direction and define the local product metric $g_{\operatorname{P}} \coloneqq  g_{\mathbb D^*} + g_{\operatorname{EH},y}$. The corresponding product connection will be denoted by $\nabla_{g_{\operatorname{P}}}$. In this local coordinate, the covariant derivatives $\nabla_{\omega_{\mathbb{D}^*}}$ and $\nabla_{\operatorname{EH}_y}$ make sense.

The following lemma makes precise the fact that $\omega_1$ is asymptotically close to the product metric.

\begin{lemma}\label{almost product 1}

In the coordinate region $D$, the following statements hold:

(1) For every smooth function $u$, and for $|y|$ sufficiently small,
\begin{align*}
\frac{5}{6}\left(\left|\nabla_{\operatorname{EH}_y}u\right|_{\operatorname{EH}_y}^2+\left|\nabla_{\omega_{\mathbb{D}^*}}u\right|_{\omega_{\mathbb{D}^*}}^2\right)\le \left|\nabla_{\omega_{1}}u\right|^2_{\omega_1}\le \frac65\left(\left|\nabla_{\operatorname{EH}_y}u\right|_{\operatorname{EH}_y}^2+\left|\nabla_{\omega_{\mathbb{D}^*}}u\right|_{\omega_{\mathbb{D}^*}}^2\right)
\end{align*}

(2) For every integer $k \ge 1$, there are constants $C(k) > 0$ such that
\begin{align*}
\left|\nabla^k_{\omega_{1}}u\right|_{\omega_1}\le C(k)\sum\limits_{l=1}^kH^{-\frac{k-l}{4}}\left|\nabla^l_{g_{\operatorname{P}}}u\right|_{g_{\operatorname{P}}}
\end{align*}where $u$ is a smooth function.

\end{lemma}

\begin{proof}

(1) Let $\textbf{A}=\operatorname{diag}(a,0,0)$, $\textbf{B}=(b_{ij})_{3\times 3}$, $\textbf{C}=(b_{ij})_{2\leq i,j\leq 3}$ where   $a=\frac{1}{|y|^2}$, $b_{ij}=\frac{\partial^2L}
{\partial z_{i-1}\partial\overline{z}_{j-1}}$, and where $r=\begin{pmatrix}
b_{21}\\
b_{31}
\end{pmatrix}$, $r^*=\overline r^{t}$.

Thus $\textbf{A}+\textbf{B}=\begin{pmatrix}
a+b_{11}&r^*\\
r&\textbf{C}
\end{pmatrix}$, while the Hermitian matrix of the product metric $g_{\operatorname{P}}$ at the
point under consideration is $\operatorname{diag}(a,\textbf{C})$.

We first estimate the metric coefficients. By Lemma \ref{first and second order derivatives}, we have
\begin{align*}
b_{11}=\frac{1}{8L}\left(\frac{1}{|y|}+\frac{1}{|y-u^2-v^2|}\right)-\frac{1}{4L^3}\left|\frac{\overline y}{2|y|}+\frac{\overline{y-u^2-v^2}}{2|y-u^2-v^2|}\right|^2.
\end{align*}
Since $B$ is positive definite, $b_{11}>0$. Dropping the
nonpositive second term and using the estimates above, we obtain
\begin{align*}
0<b_{11}&\leq\frac{1}{8L}\left(\frac{1}{|y|}+\frac{1}{|y-u^2-v^2|}\right)\leq C\left(\frac{1}{|y|H^{\frac{1}{2}}}+\frac{1}{H^{\frac{3}{2}}}\right)\leq\frac{C}{|y|H^{\frac{1}{2}}},
\end{align*}
where in the last inequality we used $|y|\leq H$. Since
$a=|y|^{-2}$, it follows that 
\begin{align*}
0<a^{-1}b_{11}\leq C\frac{|y|}{H^{\frac{1}{2}}}\leq C|y|^{\frac{1}{2}}.
\end{align*}

Since $\textbf{B}$ is positive definite, its principal block $\textbf{C}$ is
positive definite, and the Schur complement of $\textbf{C}$ in $\textbf{B}$ is
positive. Thus $b_{11}-r^*\textbf{C}^{-1}r>0$, and therefore $0\leq r^*\textbf{C}^{-1}r<b_{11}$. Consequently,
\begin{align*}
a^{-1}r^*\textbf{C}^{-1}r
\leq a^{-1}b_{11}
\leq C|y|^{\frac{1}{2}}.
\end{align*}

We now normalize the matrix $\textbf{A}+\textbf{B}$ by the product metric:
\begin{align*}
&\operatorname{diag}(a,\textbf{C})^{-\frac{1}{2}}
(\textbf{A}+\textbf{B})
\operatorname{diag}(a,\textbf{C})^{-\frac{1}{2}} =
I+
\begin{pmatrix}
a^{-1}b_{11}
&
a^{-\frac{1}{2}}r^*\textbf{C}^{-\frac{1}{2}}\\
a^{-\frac{1}{2}}\textbf{C}^{-\frac{1}{2}}r
&
0
\end{pmatrix}.
\end{align*}
The operator norm of the second matrix is bounded by
\begin{align*}
\left\|
\begin{pmatrix}
a^{-1}b_{11}
&
a^{-\frac{1}{2}}r^*\textbf{C}^{-\frac{1}{2}}\\
a^{-\frac{1}{2}}\textbf{C}^{-\frac{1}{2}}r
&
0
\end{pmatrix}
\right\|_{\mathrm{op}}&\leq
a^{-1}b_{11}
+
a^{-\frac{1}{2}}|\textbf{C}^{-\frac{1}{2}}r|=
a^{-1}b_{11}
+
\left(a^{-1}r^*\textbf{C}^{-1}r\right)^{\frac{1}{2}}\\
&\leq
a^{-1}b_{11}
+
\left(a^{-1}b_{11}\right)^{\frac{1}{2}}\leq C|y|^{\frac{1}{4}}.
\end{align*}
We may therefore assume that
\begin{align*}
\left\|\operatorname{diag}(a,\textbf{C})^{-\frac{1}{2}}(\textbf{A}+\textbf{B})\operatorname{diag}(a,\textbf{C})^{-\frac{1}{2}}
-\operatorname{Id}\right\|_{\mathrm{op}}\leq \frac16
\end{align*}
whenever $|y|$ is sufficiently small. Since the matrix in question is Hermitian,
this gives
\begin{align*}
\frac56 \operatorname{Id}
\leq
\operatorname{diag}(a,\textbf{C})^{-\frac{1}{2}}
(\textbf{A}+\textbf{B})
\operatorname{diag}(a,\textbf{C})^{-\frac{1}{2}}
\leq
\frac76 \operatorname{Id}.
\end{align*}
Equivalently,
\begin{align*}
\frac56
\operatorname{diag}(a^{-1},\textbf{C}^{-1})\leq\frac67
\operatorname{diag}(a^{-1},\textbf{C}^{-1})
\leq
(\textbf{A}+\textbf{B})^{-1}
\leq
\frac65
\operatorname{diag}(a^{-1},\textbf{C}^{-1}).
\end{align*}

Let $v=(v_1,v_2,v_3)\in\mathbb C^3$, $w=(v_2,v_3)\in\mathbb C^2$. Applying the preceding Hermitian matrix inequality to $v$, we find
\begin{align*}
\frac{5}{6}\left(a^{-1}|v_1|^2+w\textbf{C}^{-1}w^*\right)&\leq v(\textbf{A}+\textbf{B})^{-1}v^* \leq \frac65\left(a^{-1}|v_1|^2+w\textbf{C}^{-1}w^*\right).
\end{align*}
Taking $v$ to be the coordinate vector of $du$, this is exactly
\begin{align*}
\frac{5}{6}
\left(
|\nabla_{\omega_{\mathbb{D}^*}}u|_{\omega_{\mathbb{D}^*}}^2
+
|\nabla_{\operatorname{EH}_y}u|_{\operatorname{EH}_y}^2
\right)
\leq
|\nabla_{\omega_1}u|_{\omega_1}^2\leq
\frac65
\left(
|\nabla_{\omega_{\mathbb{D}^*}}u|_{\omega_{\mathbb{D}^*}}^2
+
|\nabla_{\operatorname{EH}_y}u|_{\operatorname{EH}_y}^2
\right).
\end{align*}

(2) Let $\mathcal A \coloneqq \nabla_{\omega_1} - \nabla_{g_{\operatorname{P}}}$. Direct differentiation of the explicit metric coefficients gives the scale-invariant bounds
\begin{align*}
\left|\nabla_{g_{\operatorname{P}}}^m \mathcal A\right|_{g_{\operatorname{P}}} \le C_m H^{-\frac{m+1}{4}}, \quad m \ge 0.
\end{align*}
Here we use a deliberately non-sharp estimate: in fact, the zeroth-order norm of $\mathcal A$ is uniformly bounded, but the weaker bound above is sufficient for the present application.

For a covariant tensor $T$, 
\begin{align*}
\nabla_{\omega_1} T = \nabla_{g_{\operatorname{P}}} T + \mathcal A * T.
\end{align*}
Iterating this identity shows that $\nabla_{\omega_1}^k u$ is a finite sum of terms of the form
\begin{align*}
\nabla_{g_{\operatorname{P}}}^{q_1}\mathcal A * \cdots * \nabla_{g_{\operatorname{P}}}^{q_s}\mathcal A * \nabla_{g_{\operatorname{P}}}^{l} u,
\end{align*}
where $q_1 + \cdots + q_s + s + l = k$. The coefficient of such a term is bounded by
\begin{align*}
C H^{-\frac{q_1+1}{4}} \cdots H^{-\frac{q_s+1}{4}} = C H^{-\frac{k-l}{4}}.
\end{align*}
The equivalence of $g_{\omega_1}$ and $g_{\operatorname{P}}$, proved in part (1), therefore yields
\begin{align*}
\left|\nabla_{\omega_1}^k u\right|_{\omega_1} \le C(k) \sum_{l=1}^k H^{-\frac{k-l}{4}} \left|\nabla_{g_{\operatorname{P}}}^{l} u\right|_{g_{\operatorname{P}}}.
\end{align*}

\end{proof}

\begin{proposition}\label{prop of omega 1}

In the region $U$, the metric $\omega_1$ has the following properties:

(1) The sectional curvature of $\omega_1$ is unbounded as $|y|\to 0$;

(2) The Ricci potential $h_1\coloneqq \log\left(\frac{\omega_1^3}{\frac{3}{\eta}\operatorname{Vol}_E}\right)$ of $\omega_1$ satisfies $\left|\nabla^k_{\omega_1}h_1(x)\right|_{\omega_1}\le C\left(\frac{|y|}{H^{\frac12+\frac{k}{4}}}\right)$, for $x=(y,z_1,z_2,z_3)\in U$, where $C=C(k)$ is a uniform constant.

(3) The Ricci form of $\omega_1$ satisfies: $\left|\nabla^k_{\omega_1}\operatorname{Ric}_{\omega_1}(x)\right|_{\omega_1}\le C\left(\frac{|y|}{H^{1+\frac{k}{4}}}\right)$ for $x=(y,z_1,z_2,z_3)\in U$, where $C=C(k)$ is a uniform constant.

In particular, the Ricci curvature of $\omega_1$ is uniformly bounded when $|y|\to 0$;

\end{proposition}

\begin{proof}

(1) By direct calculation, we find that the potentially largest normalized curvature components are $\frac{R_{u\bar{u}u\bar{u}}}{g_{u\bar{u}}^2}$, $\frac{R_{u\bar{u}v\bar{v}}}{g_{u\bar{u}}g_{v\bar{v}}}$, $\frac{R_{v\bar{v}v\bar{v}}}{g_{v\bar{v}}^2}$. They are all of order $ \frac1L\sim H^{-\frac12}\to \infty$ as $(y,u,v)\to (0,0,0)$. When $H\sim|y|$, i.e., near the vanishing cycle region, we have the worst blow-up rate $|y|^{-\frac12}$ of the sectional curvature of the ansatz.

(2) Let $\phi=F(H,\eta)=\sqrt{H+\eta^{\frac12}}$. We then compute
\begin{align*}
\omega_1&=\sqrt{-1}\frac{dy\wedge d\bar{y}}{|y|^2}+\sqrt{-1}\partial\bar{\partial}F(H,\eta)\\
&= F_H \sqrt{-1}\partial\bar{\partial}H + F_{HH} \sqrt{-1}\partial H \wedge \bar{\partial}H \\
&\quad + \sqrt{-1}F_{H\eta}(y d\bar{y} \wedge \bar{\partial}H + y \partial H \wedge d\bar{y}) \\
&\quad + \left(\eta F_{\eta\eta} + F_\eta+\frac{1}{\eta}\right) \sqrt{-1} dy \wedge d\bar{y}.
\end{align*}
Hence
\begin{align*}
\omega_1^3 &= 6 \operatorname{Vol}_E\big\{ (F_H^3 + F_H^{2} F_{HH} H) \\
&\quad + 4\left(\eta F_{\eta\eta} + F_\eta+\frac{1}{\eta}\right)\{H F_H^{2} + (H^2 - \eta) F_H F_{HH}\} \\
&\quad + 4 F_{H\eta} \eta F_H^{2} - 4 F_{H\eta}^2 \eta F_H(H^2 - \eta) \big\}. 
\end{align*}

We compute the following basic quantities:
\begin{align*}
&F_H = \frac{1}{2} \frac{1}{\sqrt{H + \eta^\frac12}}, \quad
F_{HH} = -\frac{1}{4} \frac{1}{(H + \eta^\frac12)^{\frac32}},\\
&F_\eta = \frac{1}{4\eta^\frac12\sqrt{H + \eta^\frac12}},\quad  F_{H\eta} = -\frac{1}{8} \frac{1}{\eta^\frac12(H + \eta^\frac12)^{\frac32}},\\
&H F_H^2 + (H^2 - \eta) F_H F_{HH} = \frac{1}{8}, \quad F_H^3 + F_H^2 H F_{HH} = \frac{H + 2\eta^\frac12}{16(H + \eta^\frac12)^{\frac52}} , \\
&F_\eta + \eta F_{\eta\eta}+\frac{1}{\eta} = \frac{1}{8\eta^\frac12\sqrt{H+\eta^\frac12}}-\frac{1}{16(H+\eta^\frac12)^{\frac 32}}+\frac{1}{\eta}.
\end{align*}

Then
\begin{align*}
\omega_1^3=6\operatorname{Vol}_E\left\{\frac{1}{2\eta}+\frac{1}{16\eta^{\frac12}(H+\eta^{\frac12})^{\frac12}}\right\}.
\end{align*}Therefore
\begin{align*}
\frac{\omega_1^3}{\frac{3}{\eta}\operatorname{Vol}_E}=1+\frac{\eta^{\frac12}}{8(H+\eta^{\frac12})^{\frac12}}.
\end{align*}
Hence $|h_1(x)|\le C\frac{|y|}{H^{\frac12}}$.

Let $s=\frac{\eta^{\frac12}}{8(H+\eta^{\frac12})^{\frac12}}$. Direct differentiation in the product coordinates gives, for every $l \ge 0$, $$\left|\nabla_{g_{\operatorname{P}}}^{l} s\right|_{g_{\operatorname{P}}} \le C_l \frac{|y|}{H^{\frac12 + \frac{l}{4}}}.$$ Indeed, a unit horizontal derivative is represented by a bounded combination of $y\partial_y$ and $\bar y\partial_{\bar y}$, whereas a unit fiberwise derivative has coordinate size $O(H^{\frac{1}{4}})$; each fiberwise derivative therefore costs at most one factor $H^{-\frac14}$. The explicit formulas for the derivatives of $H$ and $|y|$ give the displayed estimate by induction.

Because $s$ is uniformly small, the chain rule applied to $h_1 = \log(1 + s)$ gives $$\left|\nabla_{g_{\operatorname{P}}}^{l} h_1\right|_{g_{\operatorname{P}}} \le C_l \frac{|y|}{H^{\frac12 + \frac{l}{4}}}.$$Applying Lemma~\ref{almost product 1}(2), we obtain
\begin{align*}
\left|\nabla_{\omega_1}^k h_1\right|_{\omega_1}
&\le C_k \sum_{l=1}^k H^{-\frac{k-l}{4}} \left|\nabla_{g_{\operatorname{P}}}^{l} h_1\right|_{g_{\operatorname{P}}} \le C_k \sum_{l=1}^k H^{-\frac{k-l}{4}} \frac{|y|}{H^{\frac12 + \frac{l}{4}}} \le C_k \frac{|y|}{H^{\frac12 + \frac{k}{4}}}.
\end{align*}

(3) By definition $\operatorname{Ric}(\omega_1)=-\sqrt{-1}\partial\bar{\partial}h_1$. Since the K\"ahler connection preserves the complex structure,
\begin{align*}
\left|\nabla_{\omega_1}^k \operatorname{Ric}(\omega_1)\right|_{\omega_1}
\le C_k \left|\nabla_{\omega_1}^{k+2} h_1\right|_{\omega_1}\le C_k \frac{|y|}{H^{\frac12 + \frac{k+2}{4}}} = C_k \frac{|y|}{H^{1 + \frac{k}{4}}}.
\end{align*}

\end{proof}

\begin{remark}

There is a more heuristic way to see the blow-up rate of sectional curvature of $\omega_1$. By symmetry, we may assume $y\in \mathbb{R}^+$. Consider the map
\begin{align*}
\Psi_y:V_y\to V_1,\; x\mapsto \frac{x}{|y|^{\frac12}}.
\end{align*}
Then 
\begin{align*}
\left.\left(\sqrt{-1}\partial\bar{\partial}\sqrt{H+|y|}\right)\right|_{V_y}=|y|^{\frac12}\Psi_y^*\omega_{\operatorname{EH}_1},
\end{align*}where $\omega_{\operatorname{EH}_1}=\sqrt{-1}\partial\bar{\partial}\sqrt{H+1}$ is the standard Stenzel metric on the fiber $V_{1}$. So we have
\begin{align*}
\sup \left|\operatorname{Rm}\left(|y|^{\frac12}\Psi_y^*\omega_{\operatorname{EH}_1}\right)\right|_{V_y}\sim |y|^{-\frac12}\to \infty\quad \text{as}\; |y|\to 0.
\end{align*}From the above observation, it is natural to expect that the worst blow-up rate of the sectional curvature is $|y|^{-\frac12}$. However, since the fiber $V_y$ with induced metric is not totally geodesic, a more rigorous calculation is needed.

\end{remark}

\begin{remark}

The estimate above shows that $\operatorname{Ric}(\omega_1)$ is uniformly bounded, but it does not imply any decay as $|y|\to 0$. We now explain that, in fact, no such decay can hold.

A direct computer-algebra calculation gives the following asymptotic expansions as $H\to 0$:
\begin{align*}
&\frac{\left(K\omega_1+\operatorname{Ric}(\omega_1)\right)^3}
{\eta^{-1}\operatorname{Vol}_E}=a_1K^3+a_2K^2-a_3K+\eta^{\frac14}\left(-a_4-a_5K-a_6K^2+a_7K^3\right)
+O\!\left(\eta^{\frac{1}{2}}\right),\\
&\frac{\sqrt{-1}\,\partial\bar{\partial}H
\wedge\left(K\omega_1+\operatorname{Ric}(\omega_1)\right)^2}
{\eta^{-1}\operatorname{Vol}_E}=\eta^{\frac{1}{4}}\left(a_8K^2+a_9K+O\!\left(\eta^{\frac{1}{4}}\right)
\right),\\
&\frac{(\sqrt{-1}\,\partial\bar{\partial}H)^2
\wedge\left(K\omega_1+\operatorname{Ric}(\omega_1)\right)}
{\eta^{-1}\operatorname{Vol}_E}=\eta^{\frac{1}{2}}\left(
a_{10}K
+O\!\left(\eta^{\frac{1}{2}}\right)
\right),
\end{align*}
where the coefficients $a_i$ are strictly positive and remain uniformly bounded away from zero in the asymptotic region under consideration.

It follows that, for $K>0$ sufficiently large and independent of $y$,
\begin{align*}
K\omega_1+\operatorname{Ric}(\omega_1)>0.
\end{align*}
Applying the same calculation to $K\omega_1-\operatorname{Ric}(\omega_1)$ gives
\begin{align*}
K\omega_1-\operatorname{Ric}(\omega_1)>0
\end{align*}
for another uniform constant $K$. Hence
\begin{align*}
-K\omega_1
\leq
\operatorname{Ric}(\omega_1)
\leq
K\omega_1,
\end{align*}
which recovers the uniform boundedness of the Ricci tensor.

Moreover, the coefficients $-a_3$, $-a_4$ in the first expansion show that the constant $K$ cannot be chosen to tend to zero as $H\to 0$.

\end{remark}

\subsection{\texorpdfstring{The exotic metric $\omega_{\mathbb{C}^3}$}{The exotic metric omega C3}}We give a quick review of the model Calabi--Yau metric $\omega_{\mathbb{C}^3}$ on $\mathbb{C}^3$, based on \cite{li2019new}. See also \cite{Conlon2021Newexamples}, \cite{szekelyhidi2019degenerations}. 

Let $\mathbb{C}^3$ be equipped with the standard coordinates $\mathfrak{z}_1, \mathfrak{z}_2, \mathfrak{z}_3$ and the standard Hermitian inner product $|\cdot|$. Define the functions
$$
\begin{cases}
R = (|\mathfrak{z}_1|^2 + |\mathfrak{z}_2|^2 + |\mathfrak{z}_3|^2)^{\frac14}, \\
\tilde{t} = \mathfrak{z}_1^2 + \mathfrak{z}_2^2 + \mathfrak{z}_3^2, \\
\rho = \sqrt{|\tilde{t}|^2 + \sqrt{R^4 + 1}}.
\end{cases}
$$
Here $\tilde{t}$ defines the standard Lefschetz fibration on $\mathbb{C}^3$ over $\mathbb{C}_{\tilde{t}}$. Then there exists a Calabi--Yau metric $\omega_{\mathbb{C}^3} = \sqrt{-1} \partial \bar{\partial} \phi_{\mathbb{C}^3}$ on $\mathbb{C}^3$, with volume normalization
\begin{align*}
\omega_{\mathbb{C}^3}^3 = \frac{3}{2} \prod_{i=1}^3 \sqrt{-1} d\mathfrak{z}_i \wedge d\bar{\mathfrak{z}}_i;
\end{align*}
its leading-order asymptotic expression at infinity is given by
\begin{align*}
\phi_{\infty} = \frac{1}{2} |\tilde{t}|^2 + \sqrt{R^4 + \rho}, \quad \phi_{\mathbb{C}^3} = \phi_{\infty} + \phi_{\mathbb{C}^3}'. 
\end{align*}
Let
\begin{equation}
\phi_{\mathbb{C}^3}''\coloneqq \phi_{\mathbb{C}^3}'+\sqrt{R^4+\rho}.
\end{equation}This is the potential appearing in the defining equation (\ref{Ansatz 1}) of the gluing ansatz $\omega_\varepsilon$.

The geometry of $\omega_{\mathbb{C}^3}$ is carefully discussed in \cite{li2019new}, so we omit it. 

We now follow \cite{Conlon2021Newexamples} and \cite{szekelyhidi2019degenerations} to define the double weighted H\"{o}lder space $C_{\delta, \tau}^{k, \alpha}(\mathbb{C}^3, \omega_{\mathbb{C}^3})$. Let $\kappa$ be a fixed small positive number, and $K$ be a fixed large positive number. We define a weight function $w$ by
\begin{align*}
w = \begin{cases}
1 & \text{if } R \ge \kappa \rho, \\
\frac{R}{\kappa \rho} & \text{if } R \in \left(\kappa^{-1} \rho^{\frac14}, \kappa \rho\right), \\
\kappa^{-2} \rho^{-\frac34} & \text{if } R \le \kappa^{-1} \rho^{\frac14}.
\end{cases}
\end{align*}
The weighted norm of a function $f$ is then defined by
\begin{equation}\label{Weighted norm on C3}
    \|f\|_{C_{\delta, \tau}^{k, \alpha}(\mathbb{C}^3,\omega_{\mathbb{C}^3})} = \|f\|_{C^{k, \alpha}(\rho < 2K)} + \sum_{i=0}^k \sup_{\rho(z) > K} \rho^{-\delta + i} w^{-\tau + i} |\nabla^i f| + [\rho^{-\delta + k} w^{-\tau + k} \nabla^k f]_{C^{0, \alpha}_{\omega_{\mathbb{C}^3}}},
\end{equation}where the H\"{o}lder seminorm of a tensor $T$ is given by
\begin{equation}\label{Semi norm on C3}
[T]_{C^{0, \alpha}_{\omega_{\mathbb{C}^3}}} = \sup_{\rho(z) > K} \rho(z)^{\alpha} w(z)^{\alpha} \sup_{z \ne z', z' \in B(z, cR(z))} \frac{|T(z) - T(z')|}{d_{\omega_{\mathbb{C}^3}}(z, z')^{\alpha}}.
\end{equation}
Here $c > 0$ is such that the metric balls $B(z, cR(z))$ have bounded geometry and are geodesically convex, so we can compare $T(z)$ with $T(z')$ using parallel transport along a geodesic.

The deviation $\phi_{\mathbb{C}^3}'$ of $\phi_{\mathbb{C}^3}$ from its asymptotic expression $\phi_{\infty}$ is 
\begin{align*}
\|\phi_{\mathbb{C}^3}'\|_{C_{\delta, 0}^{k, \alpha}} \le C(k, \alpha, \delta), \quad \forall \delta > -1.
\end{align*}

As pointed out in \cite{li2019gluing}, we can define an embedding map
\begin{equation}
\begin{aligned}
\Upsilon_\varepsilon : \Upsilon_\varepsilon^{-1}(U_0 \cap \{|t| < \epsilon_1\}) \subset \mathbb{C}^3 &\to U_0 \cap \{|t| < \epsilon_1\} \subset X, \\
                   (\mathfrak{z}_1, \mathfrak{z}_2, \mathfrak{z}_3) &\mapsto (z_1, z_2, z_3) = \left( \left(\frac{\varepsilon}{2A_0}\right)^{\frac{1}{3}} \mathfrak{z}_1, \left(\frac{\varepsilon}{2A_0}\right)^{\frac{1}{3}} \mathfrak{z}_2, \left(\frac{\varepsilon}{2A_0}\right)^{\frac{1}{3}} \mathfrak{z}_3 \right)
\end{aligned}
\end{equation}Then
\begin{equation}
r = \left(\frac{\varepsilon}{2A_0}\right)^{\frac{1}{6}} R, \quad t = \left(\frac{\varepsilon}{2A_0}\right)^{\frac{2}{3}} \tilde{t},
\end{equation}
so that
\begin{equation*}
\sqrt{r^4 + |t|} + \frac{1}{\varepsilon} A_0 |t|^2 = \left(\frac{\varepsilon}{2A_0}\right)^{\frac{1}{3}} \left\{ \sqrt{R^4 + |\tilde{t}|} + \frac{1}{2} |\tilde{t}|^2 \right\} \sim \left(\frac{\varepsilon}{2A_0}\right)^{\frac{1}{3}} \phi_{\infty},
\end{equation*}
where in the final step we interpret $\phi_{\infty}$ as a regularized version of the non-smooth expression $\sqrt{R^4 + |\tilde{t}|} + \frac{1}{2} |\tilde{t}|^2$. Consequently, the scaled model metric $(\frac{\varepsilon}{2A_0})^{\frac{1}{3}} (\Upsilon_\varepsilon^{-1})^*\omega_{\mathbb{C}^3}$ is expected to approximate the Calabi--Yau metric $\omega_\varepsilon$ on $U_0 \cap \{|t| < \epsilon_1\}$ up to a small error.

\subsection{Approximate semi-Ricci-flat metric}From this subsection to the end of Section 2, all discussions are carried out under Setting (A).

In this subsection, we focus on the geometry of the semi-Ricci-flat metric. Although most of the material is already known in the literature, we include it here for the sake of completeness.

First, we discuss the construction of $G_0$ and $G_{\infty}$ carefully. Since the two constructions are identical, we will only discuss $G_0$.

As noted before, we can adjust the coordinate system $(U_0,(t,z_1,z_2,z_3))$ of $X$ around the node $q_0$ such that $\psi_0=\left(c_0+H^{\frac12}+O(H)\right)\Big{|}_{X_0}$. Here $c_0$ appears because of the normalization condition $\int_{X_0}\psi_0\omega_X^2=0$. 

Following \cite{spotti2014deformations}, we can construct a family of diffeomorphisms
\begin{align*}
F_{0,t}:X_0\left\backslash\left\{|z|^2\le \frac{|t|}{2}\right.\right\}\to X_t\backslash\{|w|^2=|t|\}
\end{align*}that depends smoothly on $t$. More precisely, we can first work near the node. Let $V_0=\{(z_1,z_2,z_3)\in \mathbb{C}^3|z_1^2+z_2^2+z_3^2=0\}$, $V_t=\{(w_1,w_2,w_3)\in \mathbb{C}^3|w_1^2+w_2^2+w_3^2=t\}$ as before. Consider the diffeomorphism
\begin{equation}
F_{0,t}:V_0\left\backslash\left\{|z|^2\le \frac{|t|}{2}\right.\right\}\to V_t\backslash\{|w|^2=|t|\},\; z_i\mapsto w_i(z)= z_i+\frac{t}{2|z|^2}\bar{z}_i.
\end{equation}Then we can extend this $F_{0,t}$ to $\tilde{F}_{0,t}:X_0\left\backslash\left\{|z|^2\le \frac{|t|}{2}\right.\right\}\to X_t\backslash\{|w|^2=|t|\}$ by flowing along the vector fields orthogonal to the fibers under the $\omega_X$ metric. We can assume that $\tilde{F}_{0,t}=F_{0,t}$ when restricted to $X_0\cap\{|z|\le 1\}$. From now on, we will use $F_{0,t}$ for the extended diffeomorphism. These $\{F_{0,t}\}$ fit into a fiber-preserving diffeomorphism 
\begin{equation}
F_0: \left\{|t|<\epsilon_1\right\}\times X_0\left\backslash \left\{|z|^2\le \frac12\epsilon_1\right.\right\} \to \{x\in X:|t|<\epsilon_1\}\backslash \{|w|^2=|t|\},
\end{equation}which is the identity map on the central fiber $\{0\}\times X_0$. From this point until the end of Section~3, we maintain the convention that $\epsilon_1$ is never larger than any previously introduced $\epsilon_1$ and may be decreased as needed to satisfy later propositions. To avoid ambiguity: when a property is asserted ``for $|y| < \epsilon_1$,'' or ``where $|y| < \epsilon_1$,'' the constant $\epsilon_1$ is understood to be the final, already-fixed one. In contrast, when we say ``there exists an $\epsilon_1$'' (e.g., in a proposition), this signifies that a possibly smaller value is chosen to satisfy the required condition.

Set $G_{0,t}=F_{0,t}^{-1}$. These $\{G_{0,t}\}$ fit into a fiber-preserving diffeomorphism
\begin{equation}\label{fiber pre map 1}
G_0:\left\{x\in X:|t|<\epsilon_1\right\}\left\backslash\left\{r\le \Lambda_1 |t|^{\frac14}\right.\right\}\to U'\subset \{|t|<\epsilon_1\}\times X_0,
\end{equation}
where $\epsilon_1>0$ is a small number and $\Lambda_1$ is a large fixed number.

Note that when restricted to $V_0$ and $V_t$, we have
\begin{align*}
F_{0,t}^*(|w|^2)=|z|^2+\frac{|t|^2}{4|z|^2},\quad G_{0,t}^*(|z|^2)=\frac{|w|^2+\sqrt{|w|^4-|t|^2}}{2}.
\end{align*}

First, we examine the behavior of $\omega_\varepsilon|_{X_t}$. Adopting the construction of Yang Li \cite{li2019gluing}, we equip each fiber $X_t$ with a weighted H\"older space, with emphasis on fibers lying over the small local base $\{ |t| < \epsilon_1 \}$, where $\epsilon_1 \ll 1$. On $X_t \cap U_0$, the rescaled Eguchi--Hanson metric $\operatorname{EH}_t$ is given by $\left(\sqrt{-1} \partial \bar{\partial} \sqrt{H + |t|}\right)\Big{|}_{V_t}$. The weighted H{\"o}lder norm of a function $f$ on $X_t$ can be defined by
\begin{equation}\label{Weighted norm on Xy}
\begin{aligned}
\| f \|_{C^{k,\alpha}_{\beta}(X_t)}\coloneqq & \| f \|_{C^{k,\alpha}(X_t \setminus \{r < c\}, \omega_X)} + \sum_{i =0}^k \sup_{X_t \cap U_0} r^{-\beta + i} | \nabla_{\operatorname{EH}_t}^i f |\\
&+\sup_{\substack{x,x'\in  X_t\cap U_0\\d_{\operatorname{EH}_t}(x,x')\ll r(x)}}r(x)^{-\beta+\alpha+k}
\frac{|\nabla^k_{\operatorname{EH}_t}f(x)-\nabla^k_{\operatorname{EH}_t}f(x')|}{d_{\operatorname{EH}_t}(x,x')^\alpha}.
\end{aligned}
\end{equation}
Here the values of a tensor at two nearby points are compared by parallel transport along the unique minimizing geodesic joining them. The constant $c$ is meant to be small enough to make $\{ r < c \}$ contained in the coordinate neighborhood $U_0$.

We record some useful fiberwise estimates from \cite{li2019gluing}.

\begin{lemma}\label{fiber deviation 0}

On fiber $X_t$, where $|t|<\epsilon_1$, we have the following estimates:

\begin{enumerate}
\item[(a)] $\| G_{0,t}^* r^2 - r^2 \|_{C^k_{-2}\left((U_0 \cap X_t) \setminus \{r\le \Lambda_1 |t|^{\frac14}\}\right)} \le C|t|$; 
\item[(b)] $\| G_{0,t}^* \Omega_0 - \Omega_t \|_{C^k_{-4}\left((U_0 \cap X_t) \setminus \{r\le \Lambda_1 |t|^{\frac14}\}\right)} \le C|t|$;
\item[(c)] $\| G_{0,t}^* (\omega_X|_{X_0}) - \omega_X|_{X_t} \|_{C^k_{-2}\left((U_0 \cap X_t) \setminus \{r\le \Lambda_1 |t|^{\frac14}\}\right)} \le C|t|$;
\item[(d)] $\| G_{0,t}^* (\partial\bar{\partial}\psi_0) - \partial\bar{\partial}G^*_{0,t}\psi_0|_{X_t} \|_{C^k_{-4}\left((U_0 \cap X_t) \setminus \{r\le \Lambda_1 |t|^{\frac14}\}\right)} \le C|t|$;
\item[(e)] $\left\| \omega_X|_{X_t} + \sqrt{-1} \partial \bar{\partial} G_{0,t}^* \psi_0 - G_{0,t}^* \left( \omega_X|_{X_0} + \sqrt{-1} \partial \bar{\partial} \psi_0 \right) \right\|_{C^k_{-4}\left(X_t \cap \{ H > 2|t|^{\frac56} \}\right)} \le C|t|$,
\end{enumerate}
where $C=C(k)$ is a uniform constant.

\end{lemma}

\begin{remark}

All of the above conclusions are also valid for $X_y \cap U_\infty$ where $|y| < \epsilon_1$, after replacing $t$ by $y$ and $0$ by $\infty$, respectively. To avoid notational confusion, the corresponding diffeomorphisms are denoted by $G_\infty$ and $G_{\infty,y}$.

\end{remark}

Now we focus on $\omega_{\varepsilon}|_{X_y}$, where $X_y$ is a smooth fiber near the removed singular fiber $X_{\infty}$.

Consider the $(1,1)$-form
\begin{equation}\label{fiberwise approximate SRF 1}
    \omega_y=\omega_{X}|_{X_y}+\sqrt{-1}\partial\bar{\partial}\widetilde{\phi}_y,
\end{equation} on $X_y$, where
\begin{equation}
\widetilde{\phi}_y=  c_\infty+ \gamma_1\left(\frac{H}{|y|^{\frac56}}\right)G_{\infty,y}^*\left(\psi_0^{(\infty)}-c_\infty\right)+\gamma_2\left(\frac{H}{|y|^{\frac56}}\right)\left(\sqrt{H+|y|}-\Theta_\infty\right).
\end{equation}

We first compare the Eguchi--Hanson metric $\omega_{\operatorname{EH}}=\sqrt{-1}\partial\bar{\partial}\sqrt{H+|y|}$ with the approximate fiberwise Ricci-flat metric $\omega_{2}=\omega_X+\sqrt{-1}\partial\bar{\partial}G_{\infty,y}^*\left(\psi^{(\infty)}_0-c_\infty\right)$ in a small neighborhood of the node $q_\infty$. This is essentially Lemma 2.2 in \cite{spotti2014deformations}.

\begin{lemma}\label{fiber deviation 1}

On the annulus $\left\{|y|^{\frac56}\le H\le 2|y|^{\frac56}\right\}\cap \{|y|\le \epsilon_1\}$, we have the following fiberwise estimate
\begin{align*}
\left|\left.\nabla^k_{\operatorname{EH}_y}\left(\Theta_\infty+G^*_{\infty,y}\left(\psi_0^{(\infty)}-c_\infty\right)-\sqrt{H+|y|}\right)\right|_{X_y}\right|_{\operatorname{EH}_y}=O\!\left(|y|^{\frac{14-5k}{24}}\right).
\end{align*}In particular, for $0<|y|\le \epsilon_1\ll1$, the $(1,1)$-form $\omega_y$ is positive definite, i.e., a K{\"a}hler metric.

\end{lemma}

\begin{proof}

Let $\phi_1=\sqrt{H+|y|}$, $\phi_2=\Theta_\infty+G_{\infty,y}^*(\psi_0^{(\infty)}-c_\infty)$. 

On the annular region $\omega_{\operatorname{EH}} \approx i\partial\bar{\partial}(H^{\frac12})|_{V_{y}}$. This implies that $\left|\nabla_{\operatorname{EH}_y}^k(|z|^j)\right|_{\operatorname{EH}_y} \le C|z|^{j-\frac{k}{2}}$.

By the definition of $\phi_1$, $\phi_2$ and $G_{\infty,y}$, it follows that
\begin{itemize}

\item $\left|\nabla_{\operatorname{EH}_y}^k(\phi_1 - |z|)\right|_{\operatorname{EH}_y} \le C\frac{|y|}{|z|^{1+\frac{k}{2}}}$;

\item $\left|\nabla_{\operatorname{EH}_y}^k\left(\phi_2 - \sqrt{\frac{|z|^2 + \sqrt{|z|^4 - |y|^2}}{2}}\right)\right|_{\operatorname{EH}_y} \le C|z|^{2-\frac{k}{2}}$.

\end{itemize}

Thus,
\begin{align*}
\left|\nabla_{\operatorname{EH}_y}^k(\phi_1 - \phi_2)\right|_{\operatorname{EH}_y} \le& \left|\nabla_{\operatorname{EH}_y}^k(\phi_1 - |z|)\right|_{\operatorname{EH}_y}+\left|\nabla_{\operatorname{EH}_y}^k\left(\phi_2 - \sqrt{\frac{|z|^2 + \sqrt{|z|^4 -|y|^2}}{2}}\right)\right|_{\operatorname{EH}_y}\\
&+  \left|\nabla_{\operatorname{EH}_y}^k\left(|z| - \sqrt{\frac{|z|^2 + \sqrt{|z|^4 -|y|^2}}{2}}\right)\right|_{\operatorname{EH}_y}\\
\le& C\left( \frac{|y|}{|z|^{1+\frac{k}{2}}}+|z|^{2-\frac{k}{2}} + \frac{|y|^2}{|z|^{3+\frac{k}{2}}}\right)\\
\le& C |y|^{\frac{14-5k}{24}}.
\end{align*}

\end{proof}

\begin{corollary}\label{total deviation 1}

On the annulus $\left\{|y|^{\frac56}\le H\le 2|y|^{\frac56}\right\}\cap \{|y|\le \epsilon_1\}$, we have the following estimate
\begin{align*}
\left|\nabla^k_{\widetilde{\omega}_1}\left(\Theta_\infty+G^*_{\infty,y}\left(\psi_0^{(\infty)}-c_\infty\right)-\sqrt{H+|y|}\right)\right|_{\widetilde{\omega}_1}= O\!\left(|y|^{\frac{14-5k}{24}}\right),
\end{align*}where $\widetilde{\omega}_1=\pi^*\omega_B+\omega_{\operatorname{EH}}$.

\end{corollary}

\begin{proof}

By Lemma \ref{almost product 1}, it suffices to estimate the vertical and horizontal derivatives separately. The fiber part has already been done in Lemma \ref{fiber deviation 1}. Also we have zeroth-order estimate 
\begin{align*}
\left|\Theta_\infty+G^*_{\infty,y}\left(\psi_0^{(\infty)}-c_\infty\right)-\sqrt{H+|y|}\right| = O\!\left(|y|^{\frac{7}{12}}\right).
\end{align*}
Note that $\omega_B=\sqrt{-1}\frac{A_y}{|y|^2}dy\wedge d\bar{y}\sim\frac{\sqrt{-1}}{2|y|^2}dy\wedge d\bar{y}$, hence 
\begin{align*}
\left|\nabla^k_{\omega_B}\left(\Theta_\infty+G^*_{\infty,y}\left(\psi_0^{(\infty)}-c_\infty\right)-\sqrt{H+|y|}\right)\right|_{\widetilde{\omega}_1}\sim |y|^{\frac{7}{12}}.
\end{align*}

Combining the above two estimates, the desired estimate follows.

\end{proof}

Following the proof of Lemma 2.2 in \cite{li2019gluing} with the gluing region changed, we obtain

\begin{proposition}[\text{cf.} \cite{li2019gluing} Lemma 2.2]\label{est on phi y}

For $0<|y|\le \epsilon_1\ll1$, the K{\"a}hler form $\omega_y$ satisfies:
\begin{align*}
\omega_y^2=\frac{1}{A_y}e^{h_y}\Omega_y\wedge \overline{\Omega}_y,\quad \|h_y\|_{C^{k,\alpha}_{\beta-2}}\le C(k,\alpha)|y|^{\frac{7}{12}-\frac{5\beta}{24}},
\end{align*}where $-2\le \beta<0$.
    
\end{proposition}

The following mapping property of the Laplacian on the weighted H{\"o}lder spaces is standard.

\begin{lemma}[\text{cf.} \cite{li2019gluing} Lemma 2.3 and \cite{spotti2014deformations} Proposition 3.2]\label{fiber iso}

If $-2<\beta<0$, then there exists $\epsilon_1>0$ such that for every fiber $X_y$ with $0<|y|<\epsilon_1$, the Laplacian
\begin{align*}
\Delta_{\omega_y}: C^{k+2,\alpha}_{\beta}(X_y) \to C^{k,\alpha}_{\beta-2}(X_y),
\end{align*}
when restricted to the subspace of functions satisfying $\int_{X_y} f \, \omega_y^2 = 0$, is an isomorphism.

 Moreover, its inverse satisfies a uniform estimate in $y$:
\begin{equation}
\left\| \Delta_{\omega_y}^{-1} f \right\|_{C^{k+2,\alpha}_{\beta}(X_y)} \le C(k, \alpha, \beta) \left\| f \right\|_{C^{k,\alpha}_{\beta-2}(X_y)}.
\end{equation}

\end{lemma}

We can use the implicit function theorem to prove

\begin{proposition}\label{fiber deviation 2}

Let $-2<\beta<0$ and $0<|y|<\epsilon_1$. Then
\begin{align*}
\left\|\psi_y-\widetilde{\phi}_y\right\|_{C^{k+2,\alpha}_\beta(X_y)}\le C(k,\alpha,\beta)|y|^{\frac{7}{12}-\frac{5\beta}{24}}.
\end{align*}
In particular, we have
\begin{align*}
\left\|\omega_y-\omega_{\operatorname{SRF}}|_{X_y}\right\|_{C^{k,\alpha}_{\beta-2}(X_y)}<C(k,\alpha,\beta)|y|^{\frac{7}{12}-\frac{5\beta}{24}}.
\end{align*}

\end{proposition}

\begin{proof}

Consider the following operator between Banach spaces
\begin{align*}
\mathcal{N}_y:\mathscr{X}_y \to \mathscr{Y}_y,\; u\mapsto \frac{(\omega_y+\sqrt{-1}\partial\bar{\partial}u)^2}{\omega_y^2}-1,
\end{align*}where
\begin{align*}
\mathscr{X}_y\coloneqq\left\{u\in C^{2,\alpha}_\beta(X_y)\left|\int_{X_y}u\omega_y^2=0\right.\right\},\quad \mathscr{Y}_y\coloneqq\left\{f\in C^{0,\alpha}_{\beta-2}(X_y)\left|\int_{X_y}f\omega_y^2=0\right.\right\}.
\end{align*}
Note that the tangent map of $\mathcal{N}_y$ at $0$ is $d_0\mathcal{N}_y=\frac12\Delta_{\omega_y}:\mathscr{X}_y \to \mathscr{Y}_y$, which is an isomorphism by Lemma \ref{fiber iso}. Moreover, the inverse of $d_0 \mathcal{N}_y = \frac12\Delta_{\omega_y}$ is uniformly bounded by Lemma \ref{fiber iso}. Writing $\mathcal{N}_y(u)=\frac12\Delta_{\omega_y}u+\mathcal{Q}_y(u)$, the term $\mathcal{Q}_y(u)$ is quadratic in $\sqrt{-1}\,\partial\bar\partial u$. On the ball $\|u\|_{C^{2,\alpha}_{\beta}(X_y)} \le  |y|^{\frac{7}{12} - \frac{5\beta}{24}}$, the definition of the weighted norm, together with $r \ge |y|^{\frac14}$ imply $\left\|\sqrt{-1}\,\partial\bar\partial u\right\|_{L^\infty} \le C |y|^{\frac{\beta+2}{24}} = o(1)$. Hence the Lipschitz norm of $\mathcal{Q}_y$ on this ball is uniformly $o(1)$. The quantitative implicit-function theorem, or equivalently the contraction-mapping theorem applied to $u \mapsto 2\Delta_{\omega_y}^{-1} \left( e^{-h_y} - 1 - \mathcal{Q}_y(u) \right)$, therefore yields a solution $u_y$ with $\|u_y\|_{C^{k+2,\alpha}_{\beta}(X_y)}\le C|y|^{\frac{7}{12}-\frac{5\beta}{24}}$, where $C$ is a constant independent of $y$.

By definition, we have $u_y = \psi_y - \widetilde{\phi}_y  + c_y$, where $c_y$ is a normalizing constant chosen so that $\int_{X_y} u_y \, \omega_y^2 = 0$. Moreover, it holds that $\left|\int_{X_y} u_y \, \omega_X^2\right| \leq C |y|^{\frac{7}{12} - \frac{5\beta}{24}}$. On the other hand, the explicit gluing construction yields $\left|\int_{X_y}\widetilde{\phi}_y \omega_X^2\right| \leq C|y|^{\frac{7}{12} - \frac{5\beta}{24}}$. Consequently, we obtain $|c_y| \leq C |y|^{\frac{7}{12} - \frac{5\beta}{24}}$. As a result, we have the estimate
\begin{align*}
\|\psi_y - \widetilde{\phi}_y\|_{C^{k+2,\alpha}_\beta(X_y)} \le C |y|^{\frac{7}{12} - \frac{5\beta}{24}}.
\end{align*}

\end{proof}

We now consider the deformation of the Calabi--Yau metrics $\omega_{\operatorname{SRF}}|_{X_y}$ as the complex structure varies with $y$. The following discussion is mainly an exposition of Section 2.2 in \cite{li2019gluing}.

Let $\epsilon_1,\epsilon_2>0$ be small positive numbers. $\epsilon_2$ is chosen sufficiently small so as to define the local trivialization maps in both Case 1 and Case 2.

\textbf{Case 1:} $X_{y'}$ is uniformly bounded away from the singular fibers, i.e., $y'$ is in the region $\{|t|>\epsilon_1,|y|>\epsilon_1\}$. 

Without loss of generality, we assume that $y'$ lies in the coordinate chart $(W_\infty, y)$. If $y'$ belongs to the chart $(W_0, t)$, the discussion that follows applies verbatim. If $y'$ lies outside both charts, one may simply work in another local coordinate neighborhood on the base and repeat the same argument.

In our case, we can take a fiber-preserving trivialization $G_{y'}$ over the disc $\{y: |y - y'| \le \epsilon_2 \}$,
\begin{equation*}
G_{y'}: \{x \in X: |y' - \pi(x)| \le \epsilon_2 \} \to  \{|y - y'| \le \epsilon_2 \}\times  X_{y'}  \subset  \mathbb{C}\times X_{y'}.
\end{equation*}
Such a trivialization can be constructed, for instance, by flowing along the vector fields obtained as the orthogonal horizontal lifts of tangent vector fields on $Y$, using the ambient metric $\omega_X$. For $|y - y'| \le \epsilon_2$, this induces the diffeomorphisms $G_{y',y}$ from $X_y$ to $X_{y'}$, depending smoothly on $y$.

Moreover, we can arrange that the following estimates hold:
\begin{align*}
&\left\|G_{y',y}^*\Omega_{y'} - \Omega_y\right\|_{C^k(X_y)} \le C|y - y'|; \\
&\left\|G_{y',y}^*(\omega_X|_{X_{y'}}) - \omega_X|_{X_y}\right\|_{C^k(X_y)} \le C|y - y'|.
\end{align*}  
We then define an approximate Calabi--Yau metric on $X_y$ by
\begin{equation*}
    \widehat{\omega}_y = \omega_X|_{X_y} + \sqrt{-1}\partial\bar{\partial} G_{y',y}^*\psi_{y'}.
\end{equation*}

The following fiberwise estimates hold.

\begin{lemma}\label{fiber deviation 3}

Suppose that $y'$ lies in the region $\{|t|>\epsilon_1,|y|>\epsilon_1\}$. Then for $|y-y'|<\epsilon_2$, where $\epsilon_2$ is a fixed sufficiently small positive number, we have

(a) $\|\widehat{\omega}_y - G_{y',y}^*(\omega_{\operatorname{SRF}}|_{X_{y'}})\|_{C^k(X_y)} \le C|y - y'|;
$

(b) $\|(\widehat{\omega}_y)^2 - A_y^{-1}\Omega_y \wedge \overline{\Omega}_y\|_{C^{k,\alpha}(X_y)} \le C|y - y'|$;

(c) $\left\| G_{y',y}^* \psi_{y'} - \psi_y \right\|_{C^{k+2,\alpha}(X_y)} \le C |y-y'| $.

Here $C=C(k, \alpha,\epsilon_1)$.

\end{lemma}

\textbf{Case 2:} $X_{y'}$ is near a singular fiber, i.e., $y'$ is in the region $\{|y|<\epsilon_1\}$ or $\{|t|<\epsilon_1\}$.

By symmetry, it suffices to treat the case $y' \in W_\infty$, that is, $y'$ lies near the singular fiber $X_\infty$.

Given a fiber $X_{y'}$ with $|y'| < \epsilon_1$, we can take a fiber-preserving trivialization $G_{y'}$ over the disc $\{y: |y - y'| \le \epsilon_2|y'|\}$,
\begin{equation*}
G_{y'}: \{x \in X: |y' - \pi(x)| \le \epsilon_2|y'|\} \to \{|y - y'| \le \epsilon_2|y'|\}\times  X_{y'} \subset  \mathbb{C}\times X_{y'}.
\end{equation*}
For $|y - y'| \le \epsilon_2|y'|$, this induces the diffeomorphisms $G_{y',y}$ from $X_y$ to $X_{y'}$, depending smoothly on $y$.

We may arrange
\begin{equation}\label{prop of Gy'y}
\begin{aligned}
&\left\|G_{y',y}^*\Omega_{y'} - \Omega_y\right\|_{C_{-4}^k(X_y)} \le C|y - y'|; \\
&\left\|G_{y',y}^*(\omega_X|_{X_{y'}}) - \omega_X|_{X_y}\right\|_{C_{-2}^k(X_y)} \le C|y - y'|.
\end{aligned}
\end{equation}

Since the discussion in \cite{li2019gluing} does not provide explicit expressions, we will give a concrete construction of the diffeomorphism $G_{y',y}$, which makes the subsequent error estimates more transparent.

It suffices to work in the coordinate neighborhood $U_\infty$, as the same technique used in Lemma~2.1 of \cite{spotti2014deformations} can be employed to extend the map beyond this neighborhood. In these coordinates, the fibration is modeled on the standard Lefschetz fibration:
\begin{align*}
\pi:\mathbb{C}^3\to \mathbb{C},\; (z_1,z_2,z_3)\mapsto y=z_1^2+z_2^2+z_3^2.
\end{align*}
 Consider the distinguished lift $\Xi=\sum\limits_{i=1}^3\frac{\bar{z}_i}{2|z|^2}\frac{\partial}{\partial z_i}$ of the base vector field $\frac{\partial}{\partial y}$. Let $\gamma(s)=y+s(y'-y)\subset \mathbb{C}$ be a path connecting $y$ and $y'$, where $s\in [0,1]$. Then
\begin{align}\label{base line}
\dot\gamma(s)=(y'-y)\frac{\partial}{\partial y}+\overline{(y'-y)}\frac{\partial}{\partial \bar{y}}.
\end{align}Let $\widetilde{\gamma}(s)=(z_1(s),z_2(s),z_3(s))\subset \mathbb{C}^3$ be the horizontal lift of $\gamma(s)$ defined by
\begin{align}\label{ODE 1}
\frac{d\widetilde{\gamma}(s)}{ds}=(y'-y)\Xi+\overline{(y'-y)}\overline{\Xi}.
\end{align}Then
\begin{align}\label{ODE 2}
\dot{z}_i(s)=(y'-y)\frac{\overline{z_i(s)}}{2|z(s)|^2},\quad \forall i=1,2,3.
\end{align}
Choose a point $\widetilde{\gamma}(0)=(w_1,w_2,w_3)\in V_y$. We define $G_{y',y}(w)\coloneqq \widetilde{\gamma}(1)\in V_{y'}$. By the smooth dependence of solutions of ODEs on initial data and parameters, the maps $G_{y',y}$ and $G_{y'}$ are indeed smooth.

The remaining task is to solve the ODE system (\ref{ODE 2}). 

To simplify notation, we temporarily write $\varrho=y'-y$. Note that
\begin{align*}
\frac{d|z(s)|^2}{ds}&=\sum\limits_{i=1}^3\left(z_i(s)\overline{\dot{z}_i(s)}+\dot{z}_i(s)\overline{z_i(s)}\right)\\
&=\sum\limits_{i=1}^3\left(z_i(s)\varrho\frac{z_i(s)}{|z(s)|^2}+\overline{\varrho\frac{z_i(s)}{|z(s)|^2}}\overline{z_i(s)}\right)\\
&=\frac{\varrho\overline{(y+s\varrho)}+\bar{\varrho}(y+s\varrho)}{2|z(s)|^2}.
\end{align*}Set $R(s)=|z(s)|^2$, then
\begin{align*}
\frac{d}{ds}R^2=\varrho\overline{(y+s\varrho)}+\bar{\varrho}(y+s\varrho).
\end{align*}
Set $w_i=z_i(0)$. Then $R(0)=|w|^2$, which yields 
\begin{align*}
R(s)=|z(s)|^2=\sqrt{|w|^4+s(\varrho\bar{y}+\bar{\varrho}y)+s^2|\varrho|^2}.
\end{align*}In particular, $|z(1)|^2=\sqrt{|w|^4+|y'|^2-|y|^2}$.

Plugging the explicit formula for $R(s)$ back into the ODE (\ref{ODE 1}), we get
\begin{align}
\frac{d}{ds}\begin{pmatrix}
    z(s)\\
    \overline{z(s)}
\end{pmatrix}=\frac{1}{2R(s)}\begin{pmatrix}
    0 & \varrho\\
    \bar{\varrho} & 0
\end{pmatrix}\begin{pmatrix}
    z(s)\\
    \overline{z(s)}
\end{pmatrix}.
\end{align}We can obtain an explicit solution
\begin{align}\label{solution of ODE 1}
\begin{pmatrix}
    z(s)\\
    \overline{z(s)}
\end{pmatrix}=\exp\left(J(s)\begin{pmatrix}
    0 & \varrho\\
    \bar{\varrho} & 0
\end{pmatrix}\right)\begin{pmatrix}
   w\\
    \bar{w}
\end{pmatrix},
\end{align}or equivalently, 
\begin{align}\label{solution of ODE 2}
z(s)=\cosh(|\varrho|J(s))w+\sinh(|\varrho|J(s))\frac{\varrho}{|\varrho|}
\bar{w},
\end{align}
where $J(s)=\int_0^s\frac{dt}{2R(t)}$.

If $y'=0$, then
\begin{equation}\label{Yang Li trivialization}
\begin{aligned}
G_{0,y}(w)=&\frac{1}{\sqrt{2}}\left\{\sqrt{\frac{H+\sqrt{H^2-\eta}}{\sqrt{H^2-\eta}}}w+\sqrt{\frac{H-\sqrt{H^2-\eta}}{\sqrt{H^2-\eta}}}\frac{y}{|y|}\bar{w}\right\}\\
=&\frac{1}{2} \left\{ \left( \frac{H - \eta^{\frac{1}{2}}}{H + \eta^{\frac{1}{2}}} \right)^{\frac{1}{4}} + \left( \frac{H + \eta^{\frac{1}{2}}}{H - \eta^{\frac{1}{2}}} \right)^{\frac{1}{4}} \right\} w\\
&+ \frac{1}{2} \left\{ \left( \frac{H - \eta^{\frac{1}{2}}}{H + \eta^{\frac{1}{2}}} \right)^{\frac{1}{4}} - \left( \frac{H + \eta^{\frac{1}{2}}}{H - \eta^{\frac{1}{2}}} \right)^{\frac{1}{4}} \right\} \frac{y }{|y|}\bar{w}
\end{aligned},
\end{equation}where $\eta=|y|^2$, $H=|w|^2$. This formula is exactly the one given in Lemma 5.4 in \cite{li2019new}. The map $G_{0,y}$ can be regarded as a refinement of the map $G_{0,t}$ or $G_{\infty,y}$ defined by (\ref{fiber pre map 1}). In what follows, whenever we use the notation $G_0$ or $G_\infty$, we continue to refer to the coarser formula (\ref{fiber pre map 1}), as it is sufficiently well behaved for our purposes and simpler to work with.

If $y'\neq 0$, the integral $J(1)$ can be evaluated explicitly. For the estimates below, however, it is more convenient to expand its integrand in the small parameter $\varrho=y'-y$.

Set $a=|w|^2$, $b=\operatorname{Re}(y\bar\varrho)$, $c=|\varrho|^2$. 

Since $w\in V_y$, we have $|y|\le |w|^2=a$. Moreover,
\begin{align*}
J(1)=\frac12\int_0^1\frac{dt}{\sqrt{|w|^4-|y|^2+|y+t\varrho|^2}}=\frac1{2a}\int_0^1\left(1+\frac{2tb+t^2c}{a^2}\right)^{-\frac{1}{2}}dt.
\end{align*}
Since $y'=y+\varrho$, we have $|\varrho|\le\epsilon_2(|y|+|\varrho|)$, and hence $|\varrho|
\le\frac{\epsilon_2}{1-\epsilon_2}|y|\le C\epsilon_2a$. It follows that $\left|
\frac{2tb+t^2c}{a^2}
\right|
\le
2\frac{|\varrho|}{a}
+
\frac{|\varrho|^2}{a^2}$ is uniformly smaller than $1$ for $0\le t\le1$. We may therefore use the convergent binomial expansion $(1+x)^{-\frac12}=1-\frac12x+\frac38x^2+O(x^3)$ and integrate term by term. This gives
\begin{align*}
J(1)=\frac1{2a}-\frac{b}{4a^3}+\frac{b^2}{4a^5}-\frac{c}{12a^3}+O\!\left(\frac{|\varrho|^3}{a^5}\right).
\end{align*}
Since $b=\frac12\left(y\bar\varrho+\bar y\varrho\right)$, $c=\varrho\bar\varrho$, the function $J(1)$ is real-analytic in $(\varrho,\bar\varrho)$ whenever $\frac{|\varrho|}{|w|^2}$ is sufficiently small.

Recall that $G_{y',y}(w)
=
\cosh\left(|\varrho|J(1)\right)w
+
\frac{\varrho}{|\varrho|}
\sinh\left(|\varrho|J(1)\right)\bar w$. The apparent singularity at $\varrho=0$ disappears after using
\begin{align*}
\cosh\left(|\varrho|J\right)
=
\sum_{m=0}^{\infty}
\frac{|\varrho|^{2m}J^{2m}}{(2m)!},\quad \sinh\left(|\varrho|J\right)
=
|\varrho|
\sum_{m=0}^{\infty}
\frac{|\varrho|^{2m}J^{2m+1}}{(2m+1)!}.
\end{align*}
Because $|\varrho|^{2m}=(\varrho\bar\varrho)^m$ and $J$ is real-analytic in $(\varrho,\bar\varrho)$, these formulas show that $G_{y',y}(w)$ also admits a convergent real-analytic expansion in $\varrho$ and $\bar\varrho$. In conclusion, the parity of $\cosh$ and $\sinh$ ensures that no non-smooth terms involving bare odd powers of $|\varrho|$, or the factor $\frac{\varrho}{|\varrho|}$, remain in the final expansion. 

Substituting the expansion of $J(1)$, we obtain, through quadratic order,
\begin{equation}\label{expansion of Gy'y}
G_{y',y}(w)=w+\frac{\varrho}{2|w|^2}\bar w+\frac{|\varrho|^2}{8|w|^4}w-\frac{\varrho\,\operatorname{Re}(y\bar\varrho)}{4|w|^6}\bar w+O\!\left(|w|\left(\frac{|\varrho|}{|w|^2}\right)^3\right),
\end{equation}
uniformly whenever $|\varrho|<\epsilon_2|y'|$ and $\epsilon_2$ is sufficiently small. Thus $w+\frac{\varrho}{2|w|^2}\bar w$ provides a uniform first-order approximation to $G_{y',y}(w)$.

Using this explicitly defined $G_{y',y}$, (\ref{prop of Gy'y}) can be verified by direct computation.

We then define an approximate Calabi--Yau metric on $X_y$ by
\begin{align*}
\widehat{\omega}_y = \omega_X|_{X_y} + \sqrt{-1}\partial\bar{\partial} G_{y',y}^*\psi_{y'}.
\end{align*}

We have the following fiberwise estimates:

\begin{lemma}\label{fiber deviation 4}

For $-2<\beta<0$, $|y'| < \epsilon_1, |y-y'| \le \epsilon_2|y'|$, we have:

(a) $\left\|G^*_{y',y}\left(\partial\bar{\partial}\widetilde{\phi}_{y'}\right)-\partial\bar{\partial}G^*_{y',y}\widetilde{\phi}_{y'}\right\|_{C^k_{-4}(X_y)}\le C|y-y'|$;

(b) $\left\|G^*_{y',y}\left(\partial\bar{\partial}\psi_{y'}\right)-\partial\bar{\partial}G^*_{y',y}\psi_{y'}\right\|_{C^k_{-4}(X_y)}\le C|y-y'|$;

(c) $\|\widehat{\omega}_y - G_{y',y}^*(\omega_{\operatorname{SRF}}|_{X_{y'}})\|_{C_{-4}^k(X_y)} \le C|y - y'|;
$

(d) $\left\|(\widehat{\omega}_y)^2 - A_y^{-1}\Omega_y \wedge \overline{\Omega}_y\right\|_{C_{\beta-2}^{k,\alpha}(X_y)} \le C|y - y'||y'|^{-\frac{1}{4}(\beta+2)}$;

(e) $\left\| G_{y',y}^* \widetilde{\phi}_{y'} - \widetilde{\phi}_y \right\|_{C^{k+2,\alpha}_{\beta}(X_y)} \le C |y-y'| |y'|^{-\frac{1}{4}(\beta+2)}$;

(f) $\left\| G_{y',y}^* \psi_{y'} - \psi_y \right\|_{C^{k+2,\alpha}_{\beta}(X_y)} \le C |y-y'| |y'|^{-\frac{1}{4}(\beta+2)}$. 

Here $C=C(k, \alpha, \beta)$.

\end{lemma}

\begin{proof}

(a) Note that
\begin{align*}
r^2 \left| G^*_{y',y} (\partial \bar{\partial} \widetilde{\phi}_{y'}) - \partial \bar{\partial} G^*_{y',y} \widetilde{\phi}_{y'} \right|_{\operatorname{EH}_y} \leq \frac{C |y - y'|}{r^4} \sum_{j=1}^2 \left|r^j \nabla^j \widetilde{\phi}_{y'}\right|_{\operatorname{EH}_y} \leq \frac{C |y - y'|}{r^2},
\end{align*}where the first inequality uses the fact that the relative error caused by the variation of the complex structure is of order $O\left(\frac{|y-y'|}{r^4}\right)$, and the second inequality follows from the explicit construction of $\widetilde{\phi}_y$. Differentiating further gives the $C^k$ weighted estimate.

(b), (c) By Proposition \ref{fiber deviation 2}, we get
\begin{align*}
r^j\left|  \nabla^j_{X_y}(\psi_y-\widetilde{\phi}_y)\right|_{\operatorname{EH}_y}\leq C |y|^{\frac{7}{12}-\frac{5}{24}\beta}r^\beta\ll C r^2.
\end{align*}The same argument as in (a) then establishes (b) and (c).

(d) By (c), we obtain
\begin{align*}
\left\| \omega_y'^2 - G_{y',y}^* \omega_{\operatorname{SRF}} |_{X_{y'}}^2 \right\|_{C^k_{-4}(X_y)} \leq C |y-y'|.
\end{align*}
On the other hand,
\begin{align*}
&\omega_{\operatorname{SRF}} |_{X_{y'}}^2 = A_{y'}^{-1} \Omega_{y'} \wedge \overline{\Omega}_{y'},  \\
& |A_y - A_{y'}| \leq C |y-y'|,\\
& \left\| \Omega_y - G_{y',y}^* \Omega_{y'} \right\|_{C^{k}_{-4} (X_y )} \leq C |y-y'|.
\end{align*}
Combining these facts gives
\begin{align*}
\left\|(\widehat{\omega}_y)^2 - A_y^{-1}\Omega_y \wedge \overline{\Omega}_y\right\|_{C_{-4}^{k}(X_y)} \leq C|y - y'|.
\end{align*}

(e)  This follows directly from the explicit definition of $\widetilde{\phi}_y$ together with the expansion (\ref{expansion of Gy'y}) of $G_{y',y}$.

(f) Using Lemma~\ref{fiber iso} and the implicit function theorem, one can solve
\begin{align*}
\left(\widehat{\omega}_y+\sqrt{-1}\partial\bar{\partial}\hat{u}_y\right)^2=\frac{1}{A_y}\Omega_y\wedge \overline{\Omega}_y,
\end{align*}with the estimate $\left\|\hat{u}_y\right\|_{C^{k+2,\alpha}_\beta(X_y)} \leq C(k, \alpha, \beta) |y - y'| |y'|^{-\frac{1}{4}(\beta+2)}$. The normalization condition then implies
\begin{align*}
\left\|G_{y',y}^* \psi_{y'} - \psi_y\right\|_{C^{k+2,\alpha}_\beta(X_y)}=\left\|\hat{u}_y\right\|_{C^{k+2,\alpha}_\beta(X_y)} \leq C(k, \alpha, \beta) |y - y'| |y'|^{-\frac{1}{4}(\beta+2)}.
\end{align*}

\end{proof}

\begin{remark}\label{fiber deviation 6}

From the above proposition, we know that if $|y|\sim \varepsilon^{\frac{12}{2+\mu}}$ and $|y-y'|\sim\varepsilon^{\frac12}|y'|$, then $\widetilde{\phi}_y-\psi_y$ and $G^*_{y',y}\psi_{y'}-\psi_y$ have comparable $\|\cdot\|_{C^{k+2,\alpha}_\mu}$ norm estimates.

\end{remark}

\begin{remark}

Similarly, we can define 
\begin{align*}
\widetilde{\phi}_t= c_0+  \gamma_1\left(\frac{r}{|t|^{\frac16}}\right)G_{0,t}^*\left(\psi_0^{(0)}-c_0\right)+\gamma_2\left(\frac{r}{|t|^{\frac16}}\right) \sqrt{H+|t|} .
\end{align*}
All the properties for $\widetilde{\phi}_y$ and $\psi_y$ carry over verbatim to $\widetilde{\phi}_t$ and $\psi_t$ respectively.

Also note that if $|t|\sim \varepsilon^{\frac{6}{14+\tau}}$ and $|t-t'|\sim\varepsilon^{\frac12}$, then $\widetilde{\phi}_t-\psi_t$ and $G^*_{t',t}\psi_{t'}-\psi_t$ have comparable $\|\cdot\|_{C^{k+2,\alpha}_\tau}$ norm estimates.

\end{remark}

\begin{proposition}\label{fiber deviation 7}

Fix $-2 < \mu < 0$. When $|y| \ge \varepsilon^{\frac{12}{2+\mu}}$, the deviation of $\omega_\varepsilon|_{X_y}$ from the Calabi--Yau metric $\omega_{\operatorname{SRF}}|_{X_y}$ satisfies the estimate
\begin{equation}
\left\| \omega_\varepsilon|_{X_y} - \omega_{\operatorname{SRF}}|_{X_y} \right\|_{C^{k,\alpha}_{\mu-2}(X_y)} \le C(k, \alpha, \mu) \varepsilon^{\frac12} |y|^{1-\frac{1}{4}(\mu+2)}.
\end{equation}

\end{proposition}

\begin{proof}

When $|y| > 2\varepsilon^{\frac{12}{2+\mu}}$, the support of $\chi_i$ has $\omega_B$-length scale $\varepsilon^{\frac12}$, and consequently $|y - y_i| \sim \varepsilon^{\frac12} |y_i|$, where $y\in \operatorname{supp}\chi_i$. By Lemma \ref{fiber deviation 4}, we obtain
\begin{align*}
\left\|\psi_y - G_{y_i,y}^*\psi_{y_i}\right\|_{C^{k+2,\alpha}_\mu(X_y)} \leq C |y - y_i| \, |y_i|^{-\frac14(\mu+2)} \leq C \varepsilon^{\frac12} |y|^{1 - \frac14(\mu+2)}.
\end{align*}
Since, for any $y$, the number of non‑vanishing $\chi_i$ is bounded independently of $\varepsilon$, no accumulation of errors occurs. Therefore, summing $\chi_i\left(\psi_y - G_{y_i,y}^*\psi_{y_i}\right)$ and applying $\sqrt{-1}\partial\bar{\partial}$ in the fiber direction yields the desired conclusion.

When $\varepsilon^{\frac{12}{2+\mu}} < |y| < 2\varepsilon^{\frac{12}{2+\mu}}$, we have $\chi_\infty \neq 0$, so $\widetilde{\phi}_y$ also contributes. By Lemma \ref{fiber deviation 4} and Remark \ref{fiber deviation 6}, we have
\begin{align*}
\left\|\widetilde{\phi}_y - \psi_y\right\|_{C^{k+2,\alpha}_\mu(X_y)} \sim \left\|\psi_y - G_{y_i,y}^*\psi_{y_i}\right\|_{C^{k+2,\alpha}_\mu(X_y)}.
\end{align*}
Hence the conclusion remains valid.

\end{proof}

Next we focus on $\omega_\varepsilon|_{X_t}$, where $X_t$ is a smooth fiber near the finite singular fiber $X_0$.

\begin{proposition}[\text{cf.} \cite{li2019gluing} Lemma 2.6]\label{fiber deviation 5}

Fix $-2 < \tau < 0$. 

(a) When $|t| \ge \varepsilon^{\frac{6}{14+\tau}}$, the deviation of $\omega_\varepsilon|_{X_t}$ from the Calabi--Yau metric $\omega_{\operatorname{SRF}}|_{X_t}$ is estimated by
\begin{equation}
\left\| \omega_\varepsilon|_{X_t} - \omega_{\operatorname{SRF}}|_{X_t} \right\|_{C^{k,\alpha}_{\tau-2}(X_t)} \le C(k, \alpha, \tau) \varepsilon^{\frac12} |t|^{-\frac{1}{4}(\tau+2)}.
\end{equation}

(b) When $|t| < \varepsilon^{\frac{6}{14+\tau}}$, but $r \gtrsim \varepsilon^{\frac{1}{10}} + \varepsilon^{\frac{1}{12}}\rho'^{\frac16}$, the deviation of $\omega_\varepsilon|_{X_t}$ from $G_{0,t}^*(\omega_{\operatorname{SRF}}|_{X_0})$ is estimated by
\begin{equation}
\left\| \omega_\varepsilon|_{X_t} - G_{0,t}^*(\omega_{\operatorname{SRF}}|_{X_0}) \right\|_{C^{k,\alpha}_{-4}(X_t \cap \{r \gtrsim \varepsilon^{\frac{1}{10}} + \varepsilon^{\frac{1}{12}}\rho'^{\frac16}\})} \le C(k, \alpha, \sigma, \tau) \max\left\{|t|, \varepsilon^{\frac35-\sigma}\right\},
\end{equation}
where $\sigma > 0$ can be made arbitrarily small.

\end{proposition}

\subsection{Well-definedness of the metric ansatz}

\begin{proposition}\label{Well-definedness of the metric ansatz}

There exists a bump $(1,1)$-form $\beta$ on $\mathbb{P}^1\backslash\{p_\infty\}$ such that
\begin{itemize}

\item for $\varepsilon$ sufficiently small, the $(1,1)$-form $\omega_{\varepsilon}$ defined by (\ref{Ansatz 1}) is positive definite;

\item $\omega_\varepsilon$ satisfies the following integration condition: $\int_M\left(\varepsilon\omega_\varepsilon^3-3\sqrt{-1}\Omega\wedge \overline{\Omega}\right)=0$, where the integral is understood in the sense of Lebesgue;

\item $[\omega_\varepsilon]=\left[\omega_X|_M\right]$.

\end{itemize}

\end{proposition}

\begin{proof}

As a preliminary step, we omit the $\pi^*\beta$ term in the definition of $\omega_\varepsilon$ and establish the positive definiteness of the resulting form.

Following \cite{li2019gluing}, we can prove the positive definiteness of $\omega_\varepsilon$ outside $\left\{|y|<\varepsilon^{\frac{12}{2+\mu}}\right\}$. The idea is to first restrict the gluing ansatz $\omega_\varepsilon$ to a fiber $X_y$ and then compare it with certain model K\"ahler metrics. Lemma \ref{fiber deviation 0}, \ref{fiber deviation 1}, \ref{fiber deviation 3}, \ref{fiber deviation 4} and \ref{fiber deviation 5}, together with Proposition \ref{fiber deviation 2} will serve this purpose. For the horizontal derivative, one may apply the argument used in the proof of Lemma~2.7 of \cite{li2019gluing}; see also the proof of Lemma~\ref{Ricci potential decay 3} in Section~4. These arguments demonstrate that the pullback of the generalized K\"ahler--Einstein metric from the base is the dominant contribution, so positive definiteness is not an issue.

We only need to show the positive definiteness of $\omega_\varepsilon$ in the region $\left\{|y|<\varepsilon^{\frac{12}{2+\mu}}\right\}$ since $\omega_\varepsilon$ is described by a new geometric model in this region, namely
\begin{align*}
     \omega_\varepsilon =& \frac{\sqrt{-1}A_y}{\varepsilon|y|^2}dy\wedge d\bar{y}+\omega_X\\
     &+  \sqrt{-1}\partial\bar{\partial}\left(c_\infty+ \gamma_1\left(\frac{H}{|y|^{\frac56}}\right)G_\infty^*\left(\psi_0^{(\infty)}-c_\infty\right)+\gamma_2\left(\frac{H}{|y|^{\frac56}}\right)\left(\sqrt{H+|y|}-\Theta_\infty\right) \right).
\end{align*}

\textbf{Claim 1:} $\omega_\varepsilon$ is positive definite on $\left\{|y|<\varepsilon^{\frac{12}{2+\mu}}\right\}$.

By Lemma \ref{K{\"a}hler of ST}, we know that $\omega_\varepsilon$ is positive definite on $H<|y|^{\frac56}$.

On the region $H\geq 2|y|^{\frac56}$, the ansatz takes the form
\begin{align*}
\omega_\varepsilon=\frac{\sqrt{-1}A_y}{\varepsilon|y|^2}dy\wedge d\bar{y}+\omega_X+\sqrt{-1}\partial\bar{\partial}G_{\infty,y}^*\psi_0^{(\infty)}
\end{align*}
By Lemma \ref{fiber deviation 0} (d), we know that for $\epsilon_1>0$ sufficiently small, the restricted form 
\begin{align*}
\left.\left(\frac{\sqrt{-1}A_y}{\varepsilon|y|^2}dy\wedge d\bar{y}+\omega_X+\sqrt{-1}\partial\bar{\partial}G_{\infty,y}^*\psi_0^{(\infty)}\right)\right|_{X_y}=\omega_X|_{X_y} + \sqrt{-1} \partial \bar{\partial} G_{\infty,y}^* \psi_0^{(\infty)}
\end{align*}
is positive definite on $X_y$, where $0<|y|<\epsilon_1$.

On the subregion $2|y|^{\frac56}\le H\le 1$, we have $G_{\infty,y}^*(r^2)=\sqrt{\frac{|w|^2+\sqrt{|w|^4-|y|^2}}{2}}$. Thus
\begin{align*}
\left|\frac{\partial^2}{\partial y\partial\bar{ y}}G_{\infty,y}^*\psi_0^{(\infty)}\right|dy\wedge d\bar{y}= O(H^{-\frac32})dy\wedge d\bar{y}\ll \frac{A_y}{|y|^2}dy\wedge d\bar{y},
\end{align*}when $|y|<\epsilon_1$ sufficiently small. 

On the subregion $ H\ge 1$, $G_{\infty,y}^*\psi_0^{(\infty)}$ is a smooth function with respect to $y$, hence 
\begin{align*}
\left|\frac{\partial^2}{\partial y\partial\bar{ y}}G_{\infty,y}^*\psi_0^{(\infty)}\right|dy\wedge d\bar{y}= O(1)dy\wedge d\bar{y}\ll \frac{A_y}{|y|^2}dy\wedge d\bar{y}.
\end{align*}

Hence, for $0<|y|<\varepsilon^{\frac{12}{2+\mu}}<\epsilon_1$, the $(1,1)$-form $\frac{\sqrt{-1}A_y}{\varepsilon|y|^2}dy\wedge d\bar{y}+\omega_X+\sqrt{-1}\partial\bar{\partial}G_{\infty,y}^*\psi_0^{(\infty)}$ is positive definite.

On the gluing region $\{|y|^{\frac56}\leq H\leq 2|y|^{\frac56}\}$, we can use Corollary \ref{total deviation 1} to control the difference of two model metrics. Indeed, by the chain rule, we have
\begin{align*}
\left|\partial\bar{\partial}\left(\gamma_1\left(\frac{H}{|y|^{\frac56}}\right)\left(\Theta_\infty+G_{\infty,y}^*(\psi_0^{(\infty)}-c_\infty)-\sqrt{H+|y|}\right)\right)\right|\le C|y|^{\frac16},
\end{align*}where the difference is measured either by $\frac{\sqrt{-1}A_y}{\varepsilon|y|^2}dy\wedge d\bar{y}+\sqrt{-1}\partial\bar{\partial}\sqrt{H+|y|}$ or by $\frac{\sqrt{-1}A_y}{\varepsilon|y|^2}dy\wedge d\bar{y}+\omega_X+\sqrt{-1}\partial\bar{\partial}G_{\infty,y}^*(\psi_0^{(\infty)}-c_\infty)$. 

Hence $\omega_\varepsilon$ is positive definite everywhere.

Now we add a bump form $\pi^*\beta$ pulled back from $B$ such that $\omega_\varepsilon$ remains positive definite and satisfies the integration condition.

First, we analyze $\int_M \left( \omega_X + \sqrt{-1}\partial\bar{\partial}\phi \right)^3$.

\textbf{Claim 2:} The top-degree form $(\omega_X + \sqrt{-1}\,\partial\bar\partial \phi)^3$ is absolutely integrable on $M$, and
\begin{align*}
\int_M \left( \omega_X + \sqrt{-1}\,\partial\bar{\partial}\phi \right)^3 = \int_X \omega_X^3.
\end{align*}

It suffices to show that $(\sqrt{-1}\,\partial\bar{\partial}\phi)^3$, $\omega_X \wedge (\sqrt{-1}\,\partial\bar{\partial}\phi)^2$, and $\omega_X^2 \wedge \sqrt{-1}\,\partial\bar{\partial}\phi$ are all Lebesgue integrable on $M$, and that each of their integrals vanishes.

The potential $\phi$ extends smoothly across $X_\infty \setminus \{q_\infty\}$, so the only point requiring analysis is the nodal point $q_\infty$. We therefore work in the coordinate chart $(U_\infty, (z_1, z_2, z_3))$. Let $B(0,R)$ denote the Euclidean ball in $(\mathbb{C}^3, \omega_{\operatorname{Euc}})$.

We only need to prove that, for each $q = 1, 2, 3$,
\begin{align*}
\int_{B(0,R)\setminus\{0\}} \left|(\partial\bar{\partial}\phi)^q \wedge \omega_X^{3-q}\right| < +\infty,
\quad 
\lim_{R\to 0} \int_{\partial B(0,R)} \left|\bar{\partial}\phi \wedge \omega_X^{3-q} \wedge (\partial\bar{\partial}\phi)^{q-1}\right| = 0.
\end{align*}

Recall that
\begin{align*}
G_\infty^*(\psi_0^{(\infty)} - c_\infty)
= \sqrt{\frac{|z|^2 + \sqrt{|z|^4 - |y|^2}}{2}} + \text{smooth terms}.
\end{align*}
A direct computation then yields
\begin{align*}
\left|\frac{\partial}{\partial z_i}G_\infty^*\left(\psi_0^{(\infty)}-c_\infty\right)\right|\leq C;\; \left|\frac{\partial}{\partial \bar{z}_i}G_\infty^*\left(\psi_0^{(\infty)}-c_\infty\right)\right|\leq C;\;  \left|\frac{\partial^2}{\partial z_j\partial \bar{z}_i}G_\infty^*\left(\psi_0^{(\infty)}-c_\infty\right)\right|\leq \frac{C}{H^{\frac12}},
\end{align*}where $C$ is a uniform constant.

Moreover, direct calculations give
\begin{align*}
&\frac{\partial}{\partial z_1}\sqrt{H+|y|}=\frac{\bar{z}_1+\frac{\bar{y}z_1}{|y|}}{2\sqrt{H+|y|}};\quad\frac{\partial}{\partial \bar{z}_1}\sqrt{H+|y|}=\frac{z_1+\frac{y\bar{z}_1}{|y|}}{2\sqrt{H+|y|}};\\
&\frac{\partial^2}{\partial z_1\partial \bar{z}_1}\sqrt{H+|y|}=\frac{1+\frac{|z_1|^2}{|y|}}{2\sqrt{H+|y|}}-\frac{\left(z_1+\frac{y\bar{z}_1}{|y|}\right)\left(\bar{z}_1+\frac{\bar{y}z_1}{|y|}\right)}{4(H+|y|)^{\frac32}};\\
&\frac{\partial^2}{\partial z_1\partial \bar{z}_2}\sqrt{H+|y|}=\frac{z_1\bar{z}_2}{2|y|\sqrt{H+|y|}}-\frac{\left(z_2+\frac{y\bar{z}_2}{|y|}\right)\left(\bar{z}_1+\frac{\bar{y}z_1}{|y|}\right)}{4(H+|y|)^{\frac32}}.
\end{align*}
From these we obtain
\begin{align*}
\left|\frac{\partial}{\partial z_i}\sqrt{H+|y|}\right|\le C;\quad\left|\frac{\partial}{\partial \bar{z}_i}\sqrt{H+|y|}\right|\le C;\quad \left|\frac{\partial^2}{\partial z_j\partial \bar{z}_i}\sqrt{H+|y|}\right|\le C\frac{H^{\frac12}}{|y|},
\end{align*}where $C$ is a uniform constant.

Similarly, we compute the derivatives of the cutoff:
\begin{align*}
&\frac{\partial}{\partial z_1}\gamma_2\left(\frac{H}{|y|^\alpha}\right)=\gamma_2'\left(\frac{H}{|y|^\alpha}\right)\left(\frac{\bar{z}_1}{|y|^{\alpha}}-\alpha Hz_1\frac{\bar{y}}{|y|^{\alpha+2}}\right);\\
&\frac{\partial}{\partial \bar{z}_1}\gamma_2\left(\frac{H}{|y|^\alpha}\right)=\gamma_2'\left(\frac{H}{|y|^\alpha}\right)\left(\frac{z_1}{|y|^{\alpha}}-\alpha H\bar{z}_1\frac{y}{|y|^{\alpha+2}}\right);
\end{align*}and
\begin{align*}
\frac{\partial^2}{\partial \bar{z}_1 \partial z_1}\gamma_2\left(\frac{H}{|y|^\alpha}\right)=&\gamma_2''\left(\frac{H}{|y|^\alpha}\right)\left(\frac{\bar{z}_1}{|y|^{\alpha}}-\frac{\alpha Hz_1\bar{y}}{|y|^{\alpha+2}}\right)\left(\frac{z_1}{|y|^{\alpha}}-\frac{\alpha H\bar{z}_1y}{|y|^{\alpha+2}}\right)\\
&+\gamma_2'\left(\frac{H}{|y|^\alpha}\right)\left(\frac{1}{|y|^\alpha}-\frac{\alpha\bar{z}_1^2y}{|y|^{\alpha+2}}-\frac{\alpha z_1^2\bar{y}}{|y|^{\alpha+2}}+\frac{\alpha^2H|z_1|^2}{|y|^{2+\alpha}}\right);\\
\end{align*}
\begin{align*}
\frac{\partial^2}{\partial \bar{z}_2 \partial z_1}\gamma_2\left(\frac{H}{|y|^\alpha}\right)=&\gamma_2''\left(\frac{H}{|y|^\alpha}\right)\left(\frac{\bar{z}_1}{|y|^{\alpha}}-\frac{\alpha Hz_1\bar{y}}{|y|^{\alpha+2}}\right)\left(\frac{z_2}{|y|^{\alpha}}-\frac{\alpha H\bar{z}_2y}{|y|^{\alpha+2}}\right)\\
&+\gamma_2'\left(\frac{H}{|y|^\alpha}\right)\left(-\frac{\alpha\bar{z}_1\bar{z}_2y}{|y|^{\alpha+2}}-\frac{\alpha z_1z_2\bar{y}}{|y|^{\alpha+2}}+\frac{\alpha^2Hz_1\bar{z}_2}{|y|^{2+\alpha}}\right).
\end{align*}
These calculations yield the estimates
\begin{align*}
&\left|\frac{\partial}{\partial z_i}\gamma_2\left(\frac{H}{|y|^{\frac56}}\right)\right|\leq C\frac{H^{\frac32}}{|y|^{\frac{11}{6}}}\quad \left|\frac{\partial}{\partial \bar{z}_i}\gamma_2\left(\frac{H}{|y|^{\frac56}}\right)\right|\leq C\frac{H^{\frac32}}{|y|^{\frac{11}{6}}};\\
&\left|\frac{\partial^2}{\partial\bar{z}_j\partial z_i}\gamma_2\left(\frac{H}{|y|^{\frac56}}\right)\right|\leq C\left(\frac{H^2}{|y|^{\frac{17}{6}}}+\frac{H^3}{|y|^{\frac{11}{3}}}\right).
\end{align*}

Taking the cutoff scale into account, the above estimates imply
\begin{align*}
|\partial\phi|_{\omega_{\operatorname{Euc}}} \le C H^{-\frac15},\quad 
|\bar{\partial}\phi|_{\omega_{\operatorname{Euc}}} \le C H^{-\frac15},\quad 
|\partial\bar{\partial}\phi|_{\omega_{\operatorname{Euc}}} \le C H^{-\frac{9}{10}}.
\end{align*}

Hence, for $q = 1, 2, 3$,
\begin{align*}
\int_{B(0,R)\setminus\{0\}} \left|(\partial\bar{\partial}\phi)^q \wedge \omega_X^{3-q}\right|
&\le C\int_{B(0,R)\setminus\{0\}} |\partial\bar{\partial}\phi|_{\omega_{\operatorname{Euc}}}^q \,dV_{\operatorname{Euc}}\\
&\le C\int_{B(0,R)\setminus\{0\}} H^{-\frac{9q}{10}}\,dV_{\operatorname{Euc}}\\
&\le C\int_0^R \rho^{-\frac{9q}{5}}\,\rho^5\,d\rho\\
&= C R^{6 - \frac{9q}{5}} < +\infty.
\end{align*}
On the other hand,
\begin{align*}
\int_{\partial B(0,R)} \left|\bar{\partial}\phi \wedge \omega_X^{3-q} \wedge (\partial\bar{\partial}\phi)^{q-1}\right|
&\le \int_{\partial B(0,R)} |\bar{\partial}\phi|_{\omega_{\operatorname{Euc}}}\,|\partial\bar{\partial}\phi|_{\omega_{\operatorname{Euc}}}^{q-1}\,dS\\
&\le R^{5 - \frac25 - \frac{9(q-1)}{5}} \to 0 \quad \text{as } R \to 0,
\end{align*}
again for $q = 1, 2, 3$.

Consequently,
\begin{align*}
\int_M \left(\omega_X + \sqrt{-1}\,\partial\bar{\partial}\phi\right)^3
= \int_{X\setminus\{q_\infty\}} \left(\omega_X + \sqrt{-1}\,\partial\bar{\partial}\phi\right)^3
= \int_X \omega_X^3.
\end{align*}

\textbf{Claim 3:} The following equality holds in the sense of Lebesgue integral \begin{align*}
\int_M\left(\pi^*\omega_B\wedge (\omega_X+\sqrt{-1}\partial\bar{\partial}\phi)^2-\sqrt{-1}\Omega\wedge \overline{\Omega}\right)=0.
\end{align*}

By definition, we have
\begin{align*}
\pi^*\omega_B\wedge (\omega_X+\sqrt{-1}\partial\bar{\partial}\psi)^2=\sqrt{-1}\Omega\wedge \overline{\Omega}
\end{align*}
on $M$. So it suffices to prove that
\begin{align*}
\int_M\pi^*\omega_B\wedge \left((\omega_X+\sqrt{-1}\partial\bar{\partial}\psi)^2-(\omega_X+\sqrt{-1}\partial\bar{\partial}\phi)^2\right)=0
\end{align*}
as a Lebesgue integral.

To this end, we first establish that the integrand is Lebesgue integrable. By Proposition \ref{fiber deviation 2}, we obtain
\begin{align*}
&\int_{\pi^{-1}(B(p_\infty,1))\backslash \pi^{-1}(B(p_\infty,R))}\left|\pi^*\omega_B\wedge \left((\omega_X+\sqrt{-1}\partial\bar{\partial}\psi)^2-(\omega_X+\sqrt{-1}\partial\bar{\partial}\phi)^2\right)\right|\\
\leq&C\int_{\pi^{-1}(B(p_\infty,1))\backslash \pi^{-1}(B(p_\infty,R))}|y|^{\frac{1}{24}}\pi^*\omega_B\wedge (\omega_X+\sqrt{-1}\partial\bar{\partial}\phi)^2\\
\leq  &C\int_{B(0,1)\backslash B(0,R)}(\pi_*(\omega_X+\sqrt{-1}\partial\bar{\partial}\phi)^2)|y|^{\frac{1}{24}}\omega_B \\
\leq  &C\int_{B(0,1)\backslash B(0,R)}\sqrt{-1}\frac{|y|^{\frac{1}{24}}}{|y|^2}dy\wedge d\bar{y}<+\infty
\end{align*}
Here, the third inequality relies on the basic fact that $$\pi_*(\omega_X+\sqrt{-1}\partial\bar{\partial}\phi)^2(y)=\int_{X_y}(\omega_X+\sqrt{-1}\partial\bar{\partial}\phi)^2=\int_{X_y}\omega_X^2=1.$$ Thus 
\begin{align*}
\lim\limits_{R\to 0}\int_{\pi^{-1}(B(p_\infty,1))\backslash \pi^{-1}(B(p_\infty,R))}\left|\pi^*\omega_B\wedge \left((\omega_X+\sqrt{-1}\partial\bar{\partial}\psi)^2-(\omega_X+\sqrt{-1}\partial\bar{\partial}\phi)^2\right)\right|\leq C
\end{align*}

Since the integrand is Lebesgue integrable, we may compute the integral as an improper integral. Hence by the definition of the pushforward,
\begin{align*}
&\int_M\pi^*\omega_B\wedge \left((\omega_X+\sqrt{-1}\partial\bar{\partial}\psi)^2-(\omega_X+\sqrt{-1}\partial\bar{\partial}\phi)^2\right)\\
=&\lim\limits_{R\to 0}\int_{M\backslash \pi^{-1}(B(p_\infty,R))}\pi^*\omega_B\wedge \left((\omega_X+\sqrt{-1}\partial\bar{\partial}\psi)^2-(\omega_X+\sqrt{-1}\partial\bar{\partial}\phi)^2\right)\\
=&\lim\limits_{R\to 0}\int_{B\backslash B(p_\infty,R)}\omega_B\wedge \pi_*\left((\omega_X+\sqrt{-1}\partial\bar{\partial}\psi)^2-(\omega_X+\sqrt{-1}\partial\bar{\partial}\phi)^2\right)\\
=&\lim\limits_{R\to 0}\int_{B\backslash B(p_\infty,R)}\omega_B\wedge 0=0.
\end{align*}

Combining the above calculations, we get
\begin{align*}
\int_M\varepsilon \omega_\varepsilon^3-3\sqrt{-1}\Omega\wedge \overline{\Omega}=&\int_M3\pi^*\omega_B\wedge   (\omega_X+\sqrt{-1}\partial\bar{\partial}\phi)^2-3\pi^*\omega_B\wedge (\omega_X+\sqrt{-1}\partial\bar{\partial}\psi)^2\\
&+3\varepsilon\int_X\pi^*\beta\wedge (\omega_X+\sqrt{-1}\partial\bar{\partial}\phi)^2+\varepsilon\int_X(\omega_X+\sqrt{-1}\partial\bar{\partial}\phi)^3\\
=&3\varepsilon\int_X\pi^*\beta\wedge \omega_X^2+\varepsilon\int_X\omega_X^3,
\end{align*}where the integral is understood in the Lebesgue sense. Thus, it suffices to find a bump form $\beta$ such that 
\begin{align*}
3 \int_X\pi^*\beta\wedge \omega_X^2=-\int_X\omega_X^3.
\end{align*}To this end, choose a bump form $\beta=-\frac{\sqrt{-1}}{2|y|^2}\widetilde{\chi}dy\wedge d\bar{y}$, where $0\le \widetilde{\chi}\le 1$, $\widetilde{\chi}=1$ on $\{\delta_1<|y|<\delta_2\}$ and $\widetilde\chi$ is supported in $\{\frac12\delta_1<|y|<2\delta_2\}$. Since $\int_{\{\delta_1<|y|<\delta_2\}}\frac{\sqrt{-1}}{2|y|^2}dy\wedge d\bar{y}$ can achieve any positive real number if we carefully choose $\delta_1,\delta_2$, it is always possible to find such a bump form $\beta$ to guarantee the integration condition. Moreover, $\pi^*\beta$ does not destroy positive definiteness. 

Since $B \cong \mathbb{C}$ is contractible, the closed forms $\omega_B$ and $\beta$ are exact on $B$. Therefore their pullbacks are exact on $M$, and (\ref{Ansatz 1}) gives $[\omega_\varepsilon]=\left[\omega_X|_M\right]$.

\end{proof}

\begin{remark}

The above calculation explains why we glue the Eguchi--Hanson metric and the semi-Ricci-flat metric at the scale $H \sim |y|^{\frac{5}{6}}$ instead of $H \sim |y|^{\frac{2}{3}}$. Although the latter scale is the optimal approximation scale between the two gluing metrics, the derivative estimates of the cutoff function are poor at this scale, making it impossible to obtain the aforementioned integration condition. In fact, it suffices to glue at a scale of the form $H \sim |y|^{\sigma}$ with $\sigma > \frac{4}{5}$ in order to satisfy the integration condition.

\end{remark}

\section{Mapping properties of the Laplacian on model spaces}\label{Harmonic analysis on model spaces}

Later, we will construct an approximate inverse for $\Delta_{\omega_\varepsilon}$ directly by patching together local Green operators on several model regions, and then perturb this approximate inverse to obtain the genuine one. The advantage of this approach is that it yields an optimal bound on the inverse map that is independent of $\varepsilon$. To carry out this strategy, we must first review or develop the harmonic analysis on each of the relevant geometric models in this section.

Throughout, we adopt the normalization that the Laplacian associated to a K\"ahler metric $\omega$ is
\begin{align*}
\Delta_{\omega} \coloneqq 2\operatorname{Tr}_{\omega} \sqrt{-1}\,\partial\bar{\partial},
\end{align*}
for consistency.

\subsection{\texorpdfstring{Analysis of the Laplacian for $\omega_{\mathbb{C}^3}$}{Analysis of the Laplacian for C3}}In this subsection, we recall the mapping property of the Laplacian on $\left(\mathbb{C}^3,\omega_{\mathbb{C}^3}\right)$.

\begin{proposition}[\text{cf.} \cite{li2019gluing} Proposition 3.3]\label{Laplace for C3}

Let $-2 < \tau < 0$, and assume that $\delta > -4$ and that $\delta$ avoids a discrete set of values. Then there exists a bounded right inverse $P_{\mathbb{C}^3} : C^{0,\alpha}_{\delta-2, \tau-2}(\mathbb{C}^3,\omega_{\mathbb{C}^3}) \to C^{2,\alpha}_{\delta, \tau}(\mathbb{C}^3,\omega_{\mathbb{C}^3})$ to the Laplacian $\Delta_{\omega_{\mathbb{C}^3}}$.

\end{proposition}

\begin{remark}

When $-2 < \delta < 0$, $\delta$ automatically lies outside the discrete set of indicial roots.

\end{remark}

\subsection{\texorpdfstring{Analysis of the Laplacian on $\mathbb{C}\times \mathrm{K}3$}{Analysis of the Laplacian on C x K3}}In this subsection, we discuss the mapping properties of the Laplacian on $\mathbb{C} \times \mathrm{K}3$, where the $\mathrm{K}3$ surface may be either smooth or singular. Most of the material is adapted from Section 3.2 of \cite{li2019gluing}. However, to invert the Laplacian of our ansatz $\omega_\varepsilon$ at infinity, we will prove a new weighted coercive Schauder estimate. This will provide us with the necessary information on the Laplacian on $\mathbb{C} \times X_y$ as $y$ approaches the singular point $p_\infty$.

\noindent\textit{\textbf{Case 1:}} Smooth fiber with the usual H{\"o}lder norm.

First let $X_y$ be a smooth K3 fiber equipped with the Calabi--Yau metric $\omega_{\operatorname{SRF}}|_{X_y}$ in the class $[\omega_X|_{X_y}]$. Equip $\mathbb{C}\times X_y$ with the product metric $\omega_{\operatorname{SRF},y}\coloneqq \frac{\sqrt{-1}}{2} d\zeta \wedge d\bar{\zeta}+\omega_{\operatorname{SRF}}|_{X_y} $, where $\zeta$ denotes the standard coordinate on $\mathbb{C}$.

We first consider the usual H{\"o}lder space $C^{k,\alpha}(\mathbb{C}\times X_y)$. Let $C^{k,\alpha,\operatorname{ave}}(\mathbb{C}\times X_y)$ denote the subspace of functions with zero average on each $X_y$-fiber.

\begin{proposition}[\text{cf.} \cite{li2019gluing} Lemma 3.4]\label{Laplace for C times K3}

The Laplacian $\Delta_{\omega_{\operatorname{SRF},y}}: C^{2,\alpha,\operatorname{ave}}(\mathbb{C}\times X_y) \to C^{0,\alpha,\operatorname{ave}}(\mathbb{C}\times X_y)$ is an isomorphism, with bounded inverse $P_y = \Delta_{\omega_{\operatorname{SRF},y}}^{-1}, \|P_y\| \le C(y,\alpha)$. 

Moreover, if the forcing term $f \in C^{0,\alpha,\operatorname{ave}}$ is supported in $\{|\zeta| < B\}\times X_y $ for some $B > 1$, then outside of $\{|\zeta| < B+1\}\times X_y$ the function $P_y f$ has exponential decay: 
\begin{equation}\label{exponential decay coarse version}
\begin{aligned}
&|P_y f| \le C(y,\alpha) e^{-m(|\zeta|-B)} \|f\|_{C^{0,\alpha}},  \\
&|\nabla_{\omega_{\operatorname{SRF},y}} P_y f|_{\omega_{\operatorname{SRF},y}} \le C(y,\alpha)e^{-m(|\zeta|-B)} \|f\|_{C^{0,\alpha}},  \\
&|\nabla^2_{\omega_{\operatorname{SRF},y}}P_y f|_{\omega_{\operatorname{SRF},y}} \le C(y, \alpha) e^{-m(|\zeta|-B)} \|f\|_{C^{0,\alpha}},
\end{aligned}
\end{equation} 
where $m>0$ is a uniform constant independent of $y$. 

The constants $C(y,\alpha)$ are independent of $B$ and are uniform for $y$ uniformly bounded away from $S=\{p_0,p_\infty\}$.

\end{proposition}

\noindent\textit{\textbf{Case 2:}} Smooth fiber with weighted H{\"o}lder norm.

Let $X_y$ be a smooth K3 fiber with $0<|y|<\epsilon_1$, equipped with the K{\"a}hler metric $\omega_y=\omega_X|_{X_y}+\sqrt{-1}\partial\bar{\partial}\widetilde{\phi}_y$. We endow the product manifold $\mathbb{C}\times X_y$ with the product metric $\omega_{\mathrm{P},y}\coloneqq \frac{\sqrt{-1}}{2}\,d\zeta\wedge d\bar{\zeta}+\omega_y$. 

We now introduce a doubly weighted H\"older space on $\mathbb{C}\times X_y$. Recall that a regularized distance function $r$ has been defined on $\pi^{-1}(\pi(U_\infty))$. Restricting $r$ to the fiber $X_y$, we continue to denote the resulting function by the same symbol. For parameters $\lambda,\mu\in\mathbb{R}$, define
\begin{equation}\label{Weighted norm on C Xy}
\begin{aligned}
\|f\|_{C^{k,\alpha}_{\lambda,\mu}(\mathbb{C}\times X_y,\omega_{\mathrm{P},y})}
 \coloneqq &
\sum_{i=0}^{k}
\sup_{x\in \mathbb{C}\times X_y}
\frac{r^{-\mu+i}}{|y|^\lambda}
\left|\nabla^i_{\omega_{\mathrm{P},y}}f\right|_{\omega_{\mathrm{P},y}}\\
&+\sup_{\substack{x,x'\in \mathbb{C}\times X_y\\d_{\omega_{\mathrm{P},y}}(x,x')\ll r(x)}}
\frac{r(x)^{-\mu+\alpha+k}}{|y|^\lambda}
\frac{|\nabla^k_{\omega_{\mathrm{P},y}}f(x)-\nabla^k_{\omega_{\mathrm{P},y}}f(x')|}{d_{\omega_{\mathrm{P},y}}(x,x')^\alpha}
\end{aligned}
\end{equation}
We further denote by $C^{k,\alpha,\mathrm{ave}}_{\lambda,\mu}(\mathbb{C}\times X_y)$ the subspace consisting of functions satisfying $\int_{X_y} f(\zeta,\cdot)\,\omega_y^2=0$ for every $\zeta\in\mathbb C$.

We first prove a weighted Schauder estimate on $\mathbb{C}\times \mathrm{K}3$.

\begin{lemma}[Weighted Schauder estimate]\label{weighted Schauder 1}

If $-2<\mu<0$ and $0<|y|<\epsilon_1$, then
\begin{equation}\label{weighted Schauder 2}
\|u\|_{C^{2,\alpha}_{\lambda,\mu}(\mathbb{C}\times X_y)}\leq C\left(\|\Delta_{\omega_{\operatorname{P},y}} u\|_{C^{0,\alpha}_{\lambda,\mu-2}(\mathbb{C}\times X_y)}+\|u\|_{C^0_{\lambda,\mu}(\mathbb{C}\times X_y)}\right),
\end{equation}
where $C=C(\alpha,\mu)$ is a uniform constant independent of $y$.

\end{lemma}

\begin{proof}

Without loss of generality, we may assume $\lambda=0$ and $y\in \mathbb{R}^+$.

\textbf{Region 1:} Consider $\{|\zeta-\zeta_{k,l}|<1\}\times X_y\cap \{r\geq 1\}$, where $\zeta_{k,l}=k+\sqrt{-1}l$ and $k,l\in \mathbb{Z}$. 

In this area, $r\sim 1$ and the weighted Schauder estimate (\ref{weighted Schauder 2}) is just the usual Schauder estimate. Note that in this area the product metric is also uniformly equivalent to the Euclidean metric independent of the choice of $k,l$, so the conclusion follows.

\textbf{Region 2:} Consider $\{|\zeta-\zeta_{k,l}|<1\}\times X_y\cap \{y^{\frac14}\le r\le2y^{\frac14}\}$, $k,l\in \mathbb{Z}$. 

Since we are considering the product metric, and the Euclidean metric factor is translation invariant, we may just assume $k=l=0$. The following estimates are independent of the choice of $k,l$.

Consider the map
\begin{align*}
\Phi_y : \{|\tilde{\zeta}|<y^{-\frac14}\}\times (X_1 \cap \{1 \leq r \leq 2\}) &\to \{|\zeta|<1\} \times \left( X_y \cap \{y^{\frac{1}{4}} \le r \le 2y^{\frac{1}{4}}\} \right),\\
(\tilde{\zeta},(z_1, z_2, z_3)) &\mapsto (\zeta,(w_1, w_2, w_3)) = \left(y^{\frac14}\tilde{\zeta},(y^{\frac{1}{2}}z_1, y^{\frac{1}{2}}z_2, y^{\frac{1}{2}}z_3)\right).
\end{align*}
Let $\tilde{u}(\tilde{\zeta},z)=\Phi^*_yu(\zeta,w)=u(\zeta,w(z))$, $\widetilde{\omega}_{\operatorname{P},y}=y^{-\frac12}\Phi_y^*\omega_{\operatorname{P},y}$.

By definition, we have $\widetilde{\omega}_{\operatorname{P},y}=\frac{\sqrt{-1}}{2}d\tilde{\zeta}\wedge \bar{\tilde{\zeta}}+\omega_{\operatorname{EH},1}$ on $ \{|\tilde{\zeta}|<y^{-\frac14}\}\times (X_1 \cap \{1 \leq r \leq 2\})$. This metric is uniformly equivalent to the Euclidean metric to every order. 

Now we can use the usual Schauder estimate on a unit ball $B(\tilde{p},1)\subset \{ |\tilde{\zeta}|<y^{-\frac14}\}\times (X_1 \cap \{1 \leq r \leq 2\})$, which reads
\begin{align*}
\|\tilde{u}\|_{C^{2,\alpha}_{\widetilde{\omega}_{\operatorname{P},y}}(B(\tilde{p},\frac12))}\leq C\left(\left\|\Delta_{\widetilde{\omega}_{\operatorname{P},y}} \tilde{u}\right\|_{C^{0,\alpha}_{\widetilde{\omega}_{\operatorname{P},y}}(B(\tilde{p},1))}+\|\tilde{u}\|_{L^\infty(B(\tilde{p},1))}\right).
\end{align*}Pulling back the above estimate to $X_y$, and setting $B'=\Phi_y(B(\tilde{p},\frac12))$, $B=\Phi_y(B(\tilde{p},1))$, we have
\begin{align*}
&\|u\|_{L^\infty(B')}+|y|^{\frac14}\|\nabla_{\omega_{\operatorname{P},y}} u\|_{L^\infty(B')}+|y|^{\frac12}\|\nabla^2_{\omega_{\operatorname{P},y}} u\|_{L^\infty(B')}+|y|^{\frac12+\frac\alpha 4}\left[\nabla^2_{\omega_{\operatorname{P},y}} u\right]_{C^{0,\alpha}_{\omega_{\operatorname{P},y}}(B')}\\
\leq& C\left(|y|^{\frac12}\|\Delta_{\omega_{\operatorname{P},y}}u\|_{L^\infty(B)}+|y|^{\frac12+\frac\alpha 4}\left[\Delta_{\omega_{\operatorname{P},y}}u\right]_{C^{0,\alpha}_{\omega_{\operatorname{P},y}}(B)}+\|u\|_{L^\infty(B)}\right)
\end{align*}Multiply both sides by $r^{-\mu}$, we obtain the desired weighted Schauder estimate in the corresponding region.

\textbf{Region 3:} Consider $\{|\zeta-\zeta_{k,l}|<1\}\times X_y\cap \{2y^{\frac14\sigma}\le r\le4y^{\frac14\sigma}\}$, $k,l\in \mathbb{Z}$, where $0\leq \sigma<1$.

We may still assume $k=l=0$. The following estimates are independent of the choice of $k,l$.

Consider the map
\begin{align*}
\Phi_y : \{|\tilde{\zeta}|<y^{-\frac14\sigma}\}\times (X_{y^{1-\sigma}} \cap \{2 \leq r \leq 4\}) &\to \{|\zeta|<1\} \times \left( X_y \cap \{2y^{\frac{1}{4}\sigma} \le r \le 4y^{\frac{1}{4}\sigma}\} \right),\\
(\tilde{\zeta},(z_1, z_2, z_3)) &\mapsto (\zeta,(w_1, w_2, w_3)) = \left(y^{\frac14\sigma}\tilde{\zeta},(y^{\frac{1}{2}\sigma}z_1, y^{\frac{1}{2}\sigma}z_2, y^{\frac{1}{2}\sigma}z_3)\right).
\end{align*}
Let $\tilde{u}(\tilde{\zeta},z)=\Phi^*_yu(\zeta,w)=u(\zeta,w(z))$, $\widetilde{\omega}_{\operatorname{P},y}=y^{-\frac12\sigma}\Phi_y^*\omega_{\operatorname{P},y}$.

Now the gluing ansatz $\omega_y$ is determined by both $\sqrt{-1}\partial\bar{\partial}\sqrt{H+|y|}$ and $\omega_X|_{X_y}+\sqrt{-1}\partial\bar{\partial}G_{\infty,y}^*(\psi^{(\infty)}_0-c_\infty)$. By direct calculation, we have
\begin{align*}
y^{-\frac12\sigma}\Phi_y^*\left(\frac{\sqrt{-1}}{2}d\zeta\wedge d\bar{\zeta}+\sqrt{-1}\partial\bar{\partial}\sqrt{H+|y|}\right)=\frac{\sqrt{-1}}{2}d\tilde{\zeta}\wedge d\bar{\tilde{\zeta}}+\sqrt{-1}\partial\bar{\partial}\sqrt{|z|^2+y^{1-\sigma}}
\end{align*}and
\begin{align*}
&y^{-\frac12\sigma}\Phi_y^*\left(\frac{\sqrt{-1}}{2}d\zeta\wedge d\bar{\zeta}+\omega_X|_{X_y}+\sqrt{-1}\partial\bar{\partial}G_{\infty,y}^*(\psi^{(\infty)}_0-c_\infty)\right)\\
&\sim \frac{\sqrt{-1}}{2}d\tilde{\zeta}\wedge d\bar{\tilde{\zeta}}+\sqrt{-1}\partial\bar{\partial}\sqrt{\frac{|z|^2+\sqrt{|z|^4-y^{2-2\sigma}}}{2}}.
\end{align*}In either case, the metrics when restricted to $\{|z|\geq 2\}$ are uniformly equivalent to the Euclidean metric. Therefore, the same procedure as in Region 2 can be applied.

Finally, by covering the whole manifold with the three types of regions, we obtain a global weighted Schauder estimate.

\end{proof}

\begin{proposition}\label{Laplace for C times K3 weighed version}

Let $\lambda=\frac{13}{24}$, $\mu=-\frac{1}{12}$. 

There exists a sufficiently small number $\epsilon_1>0$ such that for all $y$ with $0<|y|<\epsilon_1$, the Laplacian $\Delta_{\omega_{\operatorname{P},y}}: C^{2,\alpha,\operatorname{ave}}_{\lambda,\mu}(\mathbb{C}\times X_y) \to C^{0,\alpha,\operatorname{ave}}_{\lambda,\mu-2}(\mathbb{C}\times X_y)$ is an isomorphism, with bounded inverse $P_y = \Delta^{-1}_{\omega_{\operatorname{P},y}}, \|P_y\| \le C(\alpha)$. 

Moreover, if the forcing term $f \in C^{0,\alpha,\operatorname{ave}}_{\lambda,\mu-2}$ is supported in $\{|\zeta| < B\}\times X_y $ for some $B > 1$, then outside of $\{|\zeta| < B+1\}\times X_y  $ the function $P_y f$ has exponential decay property: for all $|\zeta|\ge B+1$,
\begin{equation}\label{exponential decay with weight}
\begin{aligned}
&|P_y f| \le C( \alpha)|y|^{\frac{13}{24}} e^{-m(|\zeta|-B)} \|f\|_{C^{0,\alpha}_{\lambda,\mu-2}},  \\
&|\nabla_{\omega_{\operatorname{P},y}} P_y f|_{\omega_{\operatorname{P},y}} \le C( \alpha)\frac{|y|^{\frac{13}{24}}}{r} e^{-m(|\zeta|-B)} \|f\|_{C^{0,\alpha}_{\lambda,\mu-2}},  \\
&|\nabla^2_{\omega_{\operatorname{P},y}}P_y f|_{\omega_{\operatorname{P},y}} \le C( \alpha)\frac{|y|^{\frac{13}{24}}}{r^2} e^{-m(|\zeta|-B)} \|f\|_{C^{0,\alpha}_{\lambda,\mu-2}} .
\end{aligned}
\end{equation}

The constants $C(\alpha)$ and $m$ are independent of $y$ and $B$.

\end{proposition}

\begin{proof}

First, by Proposition \ref{Laplace for C times K3}, for every $f\in C^{0,\alpha,\operatorname{ave}}_{\lambda,\mu-2}(\mathbb{C}\times X_y)$, there exists a unique $u\in C^{2,\alpha}(\mathbb{C}\times X_y)$ such that $\Delta u=f$. Hence it suffices to prove the boundedness of the inverse, namely, for $0<|y|\ll1$, for all $u\in C^{2,\alpha,\operatorname{ave}}_{\lambda,\mu}(\mathbb{C}\times X_y)$:
\begin{equation}\label{coercive Schauder}
    \|u\|_{C^{2,\alpha}_{\lambda,\mu}}\le C\|\Delta u\|_{C^{0,\alpha}_{\lambda,\mu-2}}.
\end{equation}
Without loss of generality, we may assume $\lambda=0$ and $y\in \mathbb{R}^+$.

By the weighted Schauder inequality (\ref{weighted Schauder 2}), we have
\begin{align*}
\|u\|_{C^{2,\alpha}_{0,\mu} }\le C\left(\|\Delta u\|_{C^{0,\alpha}_{0,\mu-2}}+\|r^{-\mu}u\|_{L^\infty}\right).
\end{align*}

We argue by contradiction. If \eqref{coercive Schauder} were false, there would exist a sequence $y_i \to 0$ and a sequence of smooth functions $u_i$ on $\mathbb{C} \times X_{y_i}$ with $\|u_i\|_{C^{2,\alpha}_{0,\mu}} = 2$ and $\|\Delta u_i\|_{C^{0,\alpha}_{0,\mu-2}} \to 0$.

By the weighted Schauder inequality, $|r^{-\mu}u_i|_{L^\infty}$ is bounded away from zero. Thus we may assume there exists a sequence of points $p_i\in X_{y_i}$ and a sequence of points $\zeta_i \in \mathbb{C}$ such that $r^{-\mu}(p_i)|u_i(\zeta_i,p_i)|=1$. By translation invariance of the metric $\frac{\sqrt{-1}}{2}d\zeta \wedge d\bar{\zeta}$, we can take $\zeta_i = 0$ for all $i$.

\textbf{Situation 1:} $r(p_i)\ge c>0$. Then $p_i\to x_*\in X_\infty\backslash \operatorname{Sing}(X_\infty)$.

By the Arzel\`a--Ascoli theorem, after passing to a subsequence (still denoted $\{u_i\}$), we may assume that on every compact subset of $\mathbb{C} \times (X_\infty \setminus \operatorname{Sing}(X_\infty))$, the sequence converges in the $C^{2,\alpha-\epsilon}$ topology to a limit function $u_\infty$ on $\mathbb{C}\times X_\infty$ satisfying $u_\infty(0,x_*) = c > 0$ and $|r^{-\mu} u_\infty| \le C$. Locally pulling back $u_\infty$ to $\mathbb{C} \times \mathbb{C}^2$ via the orbifold covering map and taking into account its behavior near the origin, it follows that $u_\infty$ is a weak solution of $\Delta u = 0$ on $\mathbb{C} \times X_\infty$, provided that $\mu > -2$.

Indeed, for all $\varphi\in C^\infty_c(B_\zeta(0,l)\times B(0,R))$, we have
\begin{align*}
\int_{B_\zeta(0,l)\times \left(B(0,R)\backslash B(0,\rho) \right)}u_\infty \Delta \varphi dV=&\int_{B_\zeta(0,l)\times \left(B(0,R)\backslash B(0,\rho) \right)}\varphi\Delta u_\infty  dV\\
&+\int_{B_\zeta(0,l)\times \partial B(0,\rho)}\left(u_\infty\frac{\partial \varphi}{\partial \nu}-\varphi\frac{\partial u_\infty}{\partial \nu}\right)dS\\
\le &C_1\rho^3\sup\limits_{B_\zeta(0,l)\times B(0,\rho)}|\nabla \varphi|\sup\limits_{B_\zeta(0,l)\times \partial B(0,\rho)}|u_\infty|\\
&+C_2\rho^3\sup\limits_{B_\zeta(0,l)\times B(0,\rho)}|\varphi|\sup\limits_{B_\zeta(0,l)\times \partial B(0,\rho)}|\nabla u_\infty|\\
\le & C\rho^3(\rho^\mu+\rho^{\mu-1})\to 0
\end{align*}if $\mu>-2$.

By regularity theory of elliptic operators, it follows that the weak solution $u_\infty$ is actually orbifold smooth. Thus we get a harmonic function $u_\infty$ on $\mathbb{C}\times X_\infty$ which is orbifold smooth and satisfies fiberwise average zero property. Hence $u_\infty=0$, a contradiction to $u_\infty(0,x_*) \neq 0$.

\textbf{Situation 2:} $\frac{|y_i|}{r^4(p_i)}\to c>0$.

Define the map 
\begin{align*}
\Phi_y : \mathbb{C}\times V_1 &\to \mathbb{C}\times V_y,\\
(\tilde{\zeta},(z_1, z_2, z_3)) &\mapsto (\zeta,(w_1, w_2, w_3)) = \left(y^{\frac14}\tilde{\zeta},(y^{\frac{1}{2}}z_1, y^{\frac{1}{2}}z_2, y^{\frac{1}{2}}z_3)\right).
\end{align*}

Consider $\widetilde{u}_i(\tilde{\zeta},z)=y_i^{-\frac{\mu}{4}}u_i(\zeta,w(z))=y_i^{-\frac{\mu}{4}}\Phi_{y_i}^*u_i(\zeta,w)$; $\widetilde{\omega}_{\operatorname{P},y_i}=y_i^{-\frac12}\Phi^*_{y_i}\omega_{\operatorname{P},y_i}$.

Then we have

$\bullet\;\left\|\widetilde{u}_i\right\|_{C^{2,\alpha}_{\widetilde{\omega}_{\operatorname{P},y_i}}(K)}\le 2$ for all compact subsets $K\subset \mathbb{C}\times V_1$. 

This follows from the assumption $\|u_i\|_{C^{2,\alpha}_{0,\mu}} = 2$ and $|y_i|\sim cr^4(p_i)$ for $i\gg 1$.

$\bullet\;\left|\widetilde{u}_i(\tilde{\zeta},z)\right|\le C|z|^{\frac{\mu}{2}}\to 0$ as $|z|\to \infty$ if $\mu<0$. Here the constant $C$ is independent of $\zeta$.

By definition $r(w)=|w|^{\frac12}=y_i^{\frac14}|z|^{\frac12}$. $|u_i(\zeta,w)|\le C r(w)^\mu$. Hence
\begin{align*}
\left|\widetilde{u}_i(\tilde{\zeta},z)\right|=|y_i|^{\frac{-\mu}{4}}\left|u_i(\zeta,w(z))\right|\le C|y_i|^{\frac{-\mu}{4}}|y_i|^{\frac{\mu}{4}}|z|^{\frac12\mu}=C|z|^{\frac12 \mu}.
\end{align*}

$\bullet\;\left|\widetilde{u}_i(0,\tilde{p}_i)\right|\sim c>0$, where $\tilde{p}_i=\Phi_{y_i}^{-1}(0,p_i)$.

Note $\left\|\widetilde{u}_i\left(\Phi^{-1}_{y_i}(0,p_i)\right)\right\|=|y_i|^{\frac{-\mu}{4}}|u_i(0,p_i)|\to c r^{-\mu}(p_i)|u_i(0,p_i)|=c$.

By the Arzel\`a--Ascoli theorem, there exists a subsequence, still denoted by $\widetilde{u}_i$, that converges to a function $\widetilde{u}_\infty$ in the $C^{2,\alpha-\varepsilon}(K)$ sense for every compact subset $K \subset \mathbb{C} \times V_1$.

By uniform convergence, we obtain $|\widetilde{u}_\infty(\tilde{\zeta},z)|\le C|z|^{\frac{\mu}{2}}$ and $|\widetilde{u}_\infty(0,\tilde{x}_*)|=c>0$, where $\tilde{x}_*\in V_1$ is some point at finite distance from the vanishing cycle.

Since $\|\Delta_{\omega_{\operatorname{P},y_i}} u_i\|_{C^{0,\alpha}_{0,\mu-2}}\to 0$, we have $\left\|\Delta_{ \widetilde{\omega}_{\operatorname{P},y_i}} \widetilde{u}_i\right\|_{L^\infty(K)}\to 0$, which implies $\Delta_{\omega_{\mathbb{C}}+\operatorname{EH}_1}\widetilde{u}_\infty=0$. Hence $\widetilde{u}_\infty$ is a harmonic function on the product manifold $\mathbb{C}\times V_1$ equipped with the product metric $\omega_\mathbb{C}+\operatorname{EH}_1$. 

If $\mu \le 0$, then $\widetilde{u}_\infty$ is uniformly bounded. 
Since $(\mathbb{C}\times V_1,\omega_\mathbb{C}+\operatorname{EH}_1)$ is a noncompact Ricci-flat manifold, 
we can apply the Cheng--Yau inequality to obtain a Liouville-type theorem, 
which implies that $\widetilde{u}_\infty$ must be constant. 
Moreover, if $\mu < 0$, then $\widetilde{u}_\infty(0,z) \to 0$ as $|z| \to \infty$, 
forcing $\widetilde{u}_\infty = 0$. 
This, however, contradicts the fact that $\widetilde{u}_\infty(0,\tilde{x}_*) = c > 0$.

\textbf{Situation 3:} $\frac{|y_i|}{r^4(p_i)}\to 0$ and $r(p_i)\to 0$.

Consider the map
\begin{align*}
\Psi_i:\mathbb{C}&\times \left\{\xi_1^2+\xi_2^2+\xi_3^2=\frac{y_i}{r^4(p_i)}\right\}\to\mathbb{C}\times V_{y_i}, \\
&(\tilde{\zeta},(\xi_1,\xi_2,\xi_3))\mapsto (\zeta,(w_1,w_2,w_3))= \left(r(p_i)\tilde{\zeta},(r(p_i)^2\xi_1,r(p_i)^2\xi_2,r(p_i)^2\xi_3)\right).
\end{align*}Let $\widetilde{u}_i(\tilde{\zeta},\xi)=r(p_i)^{-\mu}u_i(\zeta,w(\xi))=r(p_i)^{-\mu}\Psi_i^*(u_i(\zeta,w))$ and $\widetilde{\omega}_{\operatorname{P},y_i}=r^{-2}(p_i)\Psi^*_i\omega_{\operatorname{P},y_i}\to \omega_{\mathbb{C}}+\operatorname{EH}_0$ on $\mathbb{C}\times \mathbb{C}^2/\mathbb{Z}_2$.

The assumption $\|u_i\|_{C^{2,\alpha}_{0,\mu}}=2$ implies that $\left\|\widetilde{u}_i\right\|_{C^{2,\alpha}_{\widetilde{\omega}_{\operatorname{P},y_i}}}\leq2$ and $\left\|\Delta_{\omega_{\operatorname{P},y_i}} u_i\right\|_{C^{0,\alpha}_{0,\mu-2}}\to 0$ implies that $\left\|\Delta_{\widetilde{\omega}_{\operatorname{P},y_i}} \widetilde{u}_i\right\|_{L^\infty(K)}\to 0$. Passing to the limit, we obtain $\Delta_{\omega_{\mathbb{C}}+\operatorname{EH}_0}\widetilde{u}_\infty=0$ on $\mathbb{C}\times \mathbb{C}^2/\mathbb{Z}_2$ and $\widetilde{u}_\infty(0,\tilde{x}_*)=1$, where $\tilde{x}_*\in V_0$ is some point at finite distance from the node.

Also observe that
\begin{align*}
\left|\widetilde{u}_i(\tilde{\zeta},\xi)\right|=r(p_i)^{-\mu}u_i(\zeta,w(\xi))\le C r(p_i)^{-\mu}r(p_i)^{\mu}|\xi|^{\frac\mu2}=C|\xi|^{\frac\mu2}.
\end{align*}Hence $\left|\widetilde{u}_\infty(\tilde{\zeta},\xi)\right|\le C|\xi|^{\frac{\mu}{2}}$.

Pulling $\widetilde{u}_\infty$ back to $\mathbb{C}\times \mathbb{C}^2$ and using the estimate $\left|\widetilde{u}_\infty(\tilde{\zeta},\xi)\right|\le C|\xi|^{\frac{\mu}{2}}$ with $\mu>-2$, we conclude that the singularities along $\mathbb{C}\times \{0\}$ are removable. Hence, $\widetilde{u}_\infty$ extends to a smooth harmonic function on $\mathbb{C}^3$. Since $\mu<0$, $\widetilde{u}_\infty(\tilde{\zeta},\xi)$ is uniformly bounded on $\mathbb{C}^3$ and satisfies $\left|\widetilde{u}_\infty(\tilde{\zeta},\xi)\right|\le C|\xi|^{\frac\mu2}\to 0$ as $|\xi|\to \infty$. By Liouville's theorem, this forces $\widetilde{u}_\infty \equiv 0$, contradicting the condition $\widetilde{u}_\infty(0,\tilde{x}_*)=1$.

Finally, we prove the exponential decay property.

Let $f \in C^{0,\alpha,\operatorname{ave}}_{\lambda,\mu-2}$ be supported in $ \{|\zeta| < B\}\times X_y$ for some $B > 1$, let $u=P_yf\in C^{2,\alpha}_{\lambda,\mu}$ be the solution to the Poisson equation $\Delta u=f$. 

Define $g(\zeta)\coloneqq\int_{X_y}u^2(\zeta,\cdot )\omega_y^2$.

By the coercive weighted Schauder estimate, we have $\|u\|_{C^{2,\alpha}_{\lambda,\mu}}\leq C\|\Delta u\|_{C^{0,\alpha}_{\lambda,\mu-2}}$. Hence $|u(\zeta,\cdot)|\leq \frac{|y|^{\frac{13}{24}}}{r^{\frac{1}{12}}}\|u\|_{C^{2,\alpha}_{\lambda,\mu}}\leq C\frac{|y|^{\frac{13}{24}}}{r^{\frac{1}{12}}}\|\Delta u\|_{C^{0,\alpha}_{\lambda,\mu-2}}$. Then
\begin{align*}
g(\zeta)\leq C\|\Delta u\|_{C^{0,\alpha}_{\lambda,\mu-2}}^2\int_{X_y}\frac{|y|^{\frac{13}{12}}}{r^{\frac16}}\omega_y^2\leq C|y|^{\frac{13}{12}}\|\Delta u\|_{C^{0,\alpha}_{\lambda,\mu-2}}^2,
\end{align*}where the last inequality uses a simple integration trick which will be discussed carefully in Lemma \ref{decay of base function}.

By Example~2.3 in \cite{Li2021Oncollapsing}, the family $\left\{(X_y,\omega_{\operatorname{SRF}}|_{X_y})\right\}_{0<|y|<\epsilon_1}$ is uniformly noncollapsed and has uniformly bounded Ricci curvature and diameter. Consequently, by \cite{Li1980Estimates}, there exists a constant $m''>0$, independent of $y$, such that the first positive eigenvalue of $-\Delta_{\omega_{\operatorname{SRF}}|_{X_y}}$ satisfies $\lambda_1(X_y, \omega_{\operatorname{SRF}}|_{X_y})>m'' > 0$, where $m''$ is independent of $y$. By Proposition~\ref{fiber deviation 2}, the metrics $\omega_{\operatorname{SRF}}|_{X_y}$ and $\omega_y$ are uniformly equivalent, hence $\lambda_1(X_y, \omega_y) >m'> 0$, where $m'$ is independent of $y$. Equivalently, there exists a uniform constant $m' > 0$ such that
\begin{align*}
\int_{X_y} |\nabla_{\omega_y} u|^2 \, \omega_y^2 \ge \frac{1}{2} m'^2 \int_{X_y} u^2 \, \omega_y^2
\end{align*}for all functions $u$ with fiberwise average zero property on $X_y$.

Then
\begin{align*}
\Delta_{\mathbb{C}} g &= \int_{X_y}\left( 2|\nabla_{\mathbb{C}} u|^2 + 2u\Delta_{\mathbb{C}} u\right)\omega_{y}^2 = \int_{X_y} \left(2|\nabla_{\mathbb{C}} u|^2 - 2u\Delta_{\omega_y} u\right)\omega_{y}^2 = \int_{X_y} 2|\nabla u|^2\omega_{y}^2\geq   m'^2g.
\end{align*}

Consider 
\begin{align*}
\tilde{g}_\epsilon = C|y|^{\frac{13}{12}} \|\Delta u\|_{C^{0,\alpha}_{\lambda,\mu-2}}^2 \log\left(\frac{2|\zeta|}{B}\right) e^{-m'(|\zeta|-B)} + |y|^{\frac{13}{12}}\epsilon e^{\epsilon|\zeta|}, 
\end{align*}for some $\quad 0 < \epsilon \ll 1$. By direct computation
\begin{align*}
\Delta_\mathbb{C}\tilde{g}_\epsilon&=C|y|^{\frac{13}{12}}e^{-m'(|\zeta|-B)}\left\{\left(m'^2-\frac{m'}{|\zeta|}\right)\log\left(\frac{2|\zeta|}{B}\right)-\frac{2m'}{|\zeta|}\right\}\|\Delta u\|_{C^{0,\alpha}_{\lambda,\mu-2}}^2+|y|^{\frac{13}{12}}\epsilon^2e^{\epsilon|\zeta|}\left(\epsilon+\frac{1}{|\zeta|}\right)\\
&\leq m'^2\tilde{g}_\epsilon.
\end{align*}
Now
\begin{align*}
\left\{
	\begin{aligned}
&\left(\Delta_{\mathbb{C}}-m'^2\right)\left(g-\tilde{g}_\epsilon\right)\geq 0; \\
   &  g- \tilde{g}_\epsilon\leq 0\quad \text{on}\quad \{|\zeta|=B\};\\
   & g- \tilde{g}_\epsilon\leq 0\quad \text{on}\quad \{|\zeta|=R\},\quad R\gg B.
    \end{aligned}
	\right.
\end{align*}
By the maximum principle for the uniformly elliptic operator $\Delta_{\mathbb{C}}-m'^2$ and letting $R\to \infty$, we get
\begin{align*}
g(\zeta)\leq \tilde{g}_\epsilon(\zeta),\quad \forall \;|\zeta|\geq B. 
\end{align*}Letting $\epsilon \to 0$, we obtain
\begin{align*}
g(\zeta)\leq C|y|^{\frac{13}{12}}\log\left(\frac{2|\zeta|}{B}\right) e^{-m'(|\zeta|-B)}\|\Delta u\|_{C^{0,\alpha}_{\lambda,\mu-2}}^2,\quad \forall \;|\zeta|>B.
\end{align*}Choosing $0 < 2m < m'$ then yields
\begin{align*}
g(\zeta)\leq C|y|^{\frac{13}{12}}  e^{-2m(|\zeta|-B)}\|\Delta u\|_{C^{0,\alpha}_{\lambda,\mu-2}}^2,\quad \forall \;|\zeta|>B.
\end{align*}
Integrating $g(\zeta)$ on $\{B+k<|\zeta|<B+k+1\}$, we then obtain the $L^2$ estimate of $u$ on a fixed ball
\begin{align*}
\|u\|_{L^2(B_1(\zeta,p))}\leq C|y|^{\frac{13}{24}}  e^{-m(|\zeta|-B)}\|\Delta u\|_{C^{0,\alpha}_{\lambda,\mu-2}},
\end{align*}where $(\zeta,p)\in \{B+k<|\zeta|<B+k+1\}\times X_y$. From this, one can imitate the classical local Moser iteration for second-order elliptic PDE, as presented in Theorem~4.1 of \cite{han2011elliptic}, to conclude that
\begin{align*}
\|u\|_{L^\infty(B_{\frac12}(\zeta,p))}\leq C \|u\|_{L^2(B_1(\zeta,p))}\leq C|y|^{\frac{13}{24}}  e^{-m(|\zeta|-B)}\|\Delta u\|_{C^{0,\alpha}_{\lambda,\mu-2}},
\end{align*}where $(\zeta,p)\in \{B+k<|\zeta|<B+k+1\}\times X_y$. We remark that when proving the local Moser iteration, only integration by parts, Sobolev inequalities and Poincar{\'e} inequalities are used. Since the product metric $\omega_{\operatorname{P},y}$ is uniformly equivalent to $\omega_{\operatorname{SRF},y}$, the constant $C$ can be chosen uniformly.

For higher-order estimates, we can use the weighted Schauder estimate on $\{|\zeta|>B\}$, which gives $\|u\|_{C^{2,\alpha}_{\lambda,\mu}(\mathbb{C}\times X_y)}\leq C\|u\|_{C^0_{\lambda,\mu}(\mathbb{C}\times X_y)}$.

\end{proof}

\begin{remark}\label{Everything holds for SRF,y}

On $\mathbb{C}\times X_y$, one can also consider the weighted H\"older norm with respect to the product metric $\omega_{\operatorname{SRF},y}=\frac{\sqrt{-1}}{2}d\zeta\wedge d\bar{\zeta}+\omega_{\operatorname{SRF}}|_{X_y}$, defined as follows:
\begin{equation}
\begin{aligned}
\|f\|_{C^{k,\alpha}_{\lambda,\mu}(\mathbb{C}\times X_y,\omega_{\operatorname{SRF},y})}\coloneqq & \sum\limits_{i=0}^k\sup\limits_{x\in \mathbb{C}\times X_y}\frac{r^{-\mu+i}}{|y|^\lambda}\left|\nabla^i_{\omega_{\operatorname{SRF},y}}f\right|_{\omega_{\operatorname{SRF},y}}\\
&+\sup_{\substack{x,x'\in \mathbb{C}\times X_y\\d_{\omega_{\operatorname{SRF},y}}(x,x')\ll r(x)}}
\frac{r(x)^{-\mu+\alpha+k}}{|y|^\lambda}
\frac{|\nabla^k_{\omega_{\operatorname{SRF},y}}f(x)-\nabla^k_{\omega_{\operatorname{SRF},y}}f(x')|}{d_{\omega_{\operatorname{SRF},y}}(x,x')^\alpha}
\end{aligned}
\end{equation}
By an argument completely parallel to that of Lemma \ref{weighted Schauder 1}, we can also establish a weighted Schauder estimate for $\omega_{\operatorname{SRF},y}$. More precisely, if $-2<\mu<0$ and $|y|<\epsilon_1$, then
\begin{equation}\label{weighted Schauder 3}
\|u\|_{C^{2,\alpha}_{\lambda,\mu}(\mathbb{C}\times X_y,\omega_{\operatorname{SRF},y})}\leq C\left(\|\Delta_{\omega_{\operatorname{SRF},y}} u\|_{C^{0,\alpha}_{\lambda,\mu-2}(\mathbb{C}\times X_y,\omega_{\operatorname{SRF},y})}+\|u\|_{C^0_{\lambda,\mu}(\mathbb{C}\times X_y,\omega_{\operatorname{SRF},y})}\right),
\end{equation}
where $C=C(\alpha,\mu)$ is a uniform constant independent of $y$.

\begin{proof}

When $H\geq 1$, the metric $\frac{\sqrt{-1}}{2} d\zeta \wedge d\bar{\zeta}+\omega_{\operatorname{SRF}}|_{X_y}$ is uniformly equivalent to the Euclidean metric, so we have the uniform Schauder estimate. 

When $H<1$, consider the region $\{|\zeta|<1\} \times \left( X_y \cap \{|y|^{\frac{1}{4}\sigma} \le r \le 2|y|^{\frac14\sigma}\} \right)$, since the other regions are similar.

Consider the map
\begin{align*}
\Phi_y : \{|\tilde{\zeta}|<y^{-\frac14\sigma}\}\times \left(X_{y^{1-\sigma}} \cap \{1 \leq r \leq 2\}\}\right) &\to \{|\zeta|<1\} \times \left( X_y \cap \{y^{\frac{1}{4}\sigma} \le r \le 2y^{\frac14\sigma}\} \right),\\
(\tilde{\zeta},(z_1, z_2, z_3)) &\mapsto (\zeta,(w_1, w_2, w_3)) = \left(y^{\frac14\sigma}\tilde{\zeta},(y^{\frac{1}{2}\sigma}z_1, y^{\frac{1}{2}\sigma}z_2, y^{\frac{1}{2}\sigma}z_3)\right).
\end{align*}
Consider the metrics
\begin{align*}
\widetilde{\omega}_{\operatorname{SRF},y}=y^{-\frac12\sigma}\Phi_y^*\omega_{\operatorname{SRF},y},\quad \widetilde{\omega}_{\operatorname{P},y}=y^{-\frac12\sigma}\Phi_y^*\omega_{\operatorname{P},y}.
\end{align*}Note that $\frac{\sqrt{-1}}{2}d\tilde{\zeta}\wedge d\bar{\tilde{\zeta}}+\operatorname{EH}_{y^{1-\sigma}}=y^{-\frac12\sigma}\Phi_y^*\left(\frac{\sqrt{-1}}{2}d\zeta\wedge d\bar{\zeta}+\operatorname{EH}_y\right)$. Hence by Proposition \ref{fiber deviation 2}, we have
\begin{align*}
\left|\nabla^k_{\widetilde{\omega}_{\operatorname{P},y}}\left(\widetilde{\omega}_{\operatorname{SRF},y}-\widetilde{\omega}_{\operatorname{P},y}\right)\right|_{\widetilde{\omega}_{\operatorname{P},y}}&=\left|\nabla^k_{y^{-\frac12\sigma}\Phi_y^*\omega_{y}}\left(\widetilde{\omega}_{\operatorname{SRF},y}-\widetilde{\omega}_{\operatorname{P},y}\right)\right|_{y^{-\frac12\sigma}\Phi_y^*\omega_{y}}\\
&\leq C\left|\nabla^k_{y^{-\frac12\sigma}\Phi_y^*\operatorname{EH}_y}y^{-\frac12\sigma}\Phi_y^*\left(\omega_{\operatorname{SRF},y}-\omega_{\operatorname{P},y}\right)\right|_{y^{-\frac12\sigma}\Phi_y^*\operatorname{EH}_y}\\
&\leq C|y|^{\frac k4\sigma}\left|\nabla^k_{ \operatorname{EH}_y}  \left(\omega_{\operatorname{SRF},y}-\omega_{\operatorname{P},y}\right)\right|_{ \operatorname{EH}_y}\\
&\leq C\frac{|y|^{\frac{7}{12}-\frac{5}{24}\mu+\frac{1}{4}k\sigma}}{H^{\frac12-\frac{1}{4}\mu+\frac14k}}\\
&\leq C|y|^{\frac{1}{24}}.
\end{align*}Hence $\widetilde{\omega}_{\operatorname{SRF},y}$ is uniformly equivalent to $\widetilde{\omega}_{\operatorname{P},y}$ to any order. Consequently, their $C^{k,\alpha}$ H{\"o}lder norms are uniformly equivalent on $\{|\tilde{\zeta}|<y^{-\frac14\sigma}\}\times (X_{y^{1-\sigma}} \cap \{1 \leq r \leq 2\})$. Now $\widetilde{\omega}_{\operatorname{P},y}$, $\widetilde{\omega}_{\operatorname{SRF},y}$ are both uniformly equivalent to the Euclidean metric to every order. The usual Schauder estimate holds for $\widetilde{\omega}_{\operatorname{SRF},y}$ on a unit ball. Pulling the estimate back to the $X_y$-fiber, we obtain the desired conclusion.

\end{proof}

Furthermore, the statements in Proposition \ref{Laplace for C times K3 weighed version} also hold for $\omega_{\operatorname{SRF},y}$. Specifically, let $\lambda=\frac{13}{24}$ and $\mu=-\frac{1}{12}$. Then there exists a sufficiently small number $\epsilon_1>0$ such that for all $y$ with $0<|y|<\epsilon_1$, the Laplacian 
\begin{align*}
\Delta_{\omega_{\operatorname{SRF},y}}: C^{2,\alpha,\operatorname{ave}}_{\lambda,\mu}(\mathbb{C}\times X_y,\omega_{\operatorname{SRF},y}) \to C^{0,\alpha,\operatorname{ave}}_{\lambda,\mu-2}(\mathbb{C}\times X_y,\omega_{\operatorname{SRF},y})
\end{align*}
is an isomorphism, with bounded inverse $P_y = \Delta^{-1}_{\omega_{\operatorname{SRF},y}}$ satisfying $\|P_y\| \le C$, where the constant $C$ is independent of $y$.

This result follows from the closeness between $\omega_{\operatorname{P},y}$ and $\omega_{\operatorname{SRF},y}$, which allows the proof of Proposition~\ref{Laplace for C times K3 weighed version} to be adapted to the case of $\omega_{\operatorname{SRF},y}$.

In particular, for $X_y$ with $0<|y|<\epsilon_1$, we have for any $f\in C^{0,\alpha,\operatorname{ave}}(\mathbb{C}\times X_y)$,
\begin{align}\label{coarse weighted inverse}
\left\|\Delta_{\omega_{\operatorname{SRF},y}}^{-1}f\right\|_{C^{2,\alpha}(\mathbb{C}\times X_y)}\leq \frac{C}{|y|}\left\| f\right\|_{C^{0,\alpha}(\mathbb{C}\times X_y)},
\end{align}where the constant $C$ is independent of $y$.

Moreover, we can obtain an exponential decay property completely analogous to that in formula (\ref{exponential decay with weight}). If we replace the weighted norm by the usual H\"older norm and use the elementary inequality $\|f\|_{C^{0,\alpha}_{\lambda,\mu-2}} \le C|y|^{-\frac{13}{24}} \|f\|_{C^{0,\alpha}}$, we obtain the following exponential decay estimate: if the forcing term $f \in C^{0,\alpha,\operatorname{ave}}$ is supported in $\{|\zeta| < B\} \times X_y$ for some $B > 1$, then on the region $\{|\zeta| > B+1\} \times X_y$, it holds that  
\begin{equation}\label{exponential decay without weight}
\begin{aligned}
&|P_y f| \le C( \alpha) e^{-m(|\zeta|-B)} \|f\|_{C^{0,\alpha}},  \\
&|\nabla_{\omega_{\operatorname{SRF},y}} P_y f|_{\omega_{\operatorname{SRF},y}} \le C( \alpha)r^{-1} e^{-m(|\zeta|-B)} \|f\|_{C^{0,\alpha}},  \\
&|\nabla^2_{\omega_{\operatorname{SRF},y}}P_y f|_{\omega_{\operatorname{SRF},y}} \le C( \alpha)r^{-2} e^{-m(|\zeta|-B)} \|f\|_{C^{0,\alpha}},
\end{aligned}
\end{equation}
which makes the coefficient $C(y,\alpha)$ claimed in Proposition~\ref{Laplace for C times K3} completely explicit.

We fix the choice of the constant $\epsilon_1$ from now on.

\end{remark}

\noindent\textit{\textbf{Case 3:}} Singular fiber with weighted H{\"o}lder norm.

Let $X_0$ be the nodal K3 fiber equipped with the Ricci-flat orbifold metric $\omega_{\operatorname{SRF}}|_{X_0}$. We endow $\mathbb{C}\times X_0$ with the singular product metric $\omega_{\operatorname{SRF},0}\coloneqq \frac{\sqrt{-1}}{2}d\zeta\wedge d\bar{\zeta} +\omega_{\operatorname{SRF}}|_{X_0}$.

We next introduce a doubly weighted H\"older space on $\mathbb{C}\times X_0$. Recall that there exists a function $r$ on $X_0$ which is uniformly equivalent to the distance from the nodal point. Placing the origin at $\zeta=0$ on the singular line of $\mathbb{C}\times X_0$, we consider the auxiliary weight functions
\begin{align*}
\rho' = \sqrt{r^2 + |\zeta|^2},
\quad 
w' =
\begin{cases}
1, & \text{if } r > \kappa \rho', \\[4pt]
\dfrac{r}{\kappa \rho'}, & \text{if } r \le \kappa \rho'.
\end{cases}
\end{align*}

Using these weights, we define the norm
\begin{equation}\label{weighted Weighted norm on C X0}
\| f \|_{C^{k,\alpha}_{\delta,\tau}( \mathbb{C}\times X_0)}
=
\sum_{j=0}^k
\sup
\rho'^{-\delta+j}
w'^{-\tau+j}
|\nabla^j_{\omega_{\operatorname{SRF},0}} f|_{\omega_{\operatorname{SRF},0}}
+
[\rho'^{-\delta+k}w'^{-\tau+k}\nabla^k_{\omega_{\operatorname{SRF},0}}f]_{C^{0,\alpha}_{\omega_{\operatorname{SRF},0}}},
\end{equation}
where the H\"older seminorm of a tensor field $T$ is given by
\begin{equation}\label{Semi norm on C X0}
[T]_{C^{0,\alpha}_{\omega_{\operatorname{SRF},0}}}
=
\sup_{\substack{x,x'\in \mathbb{C}\times X_0 \\ d(x,x')\ll r(x)}}
\min\{\rho'(x)^\alpha w'(x)^\alpha,\rho'(x')^\alpha w'(x')^\alpha\}
\frac{|T(x)-T(x')|}{d(x,x')^\alpha},
\end{equation}
where the values of $T$ at nearby points are compared via parallel transport along the unique minimizing geodesic joining them.

\begin{lemma}[\text{cf.} \cite{li2019gluing} Lemma 3.5]\label{holder on pushforward}

Let $-2 < \tau < 0$ and $\delta > -2$. Then every function $f \in C^{0,\alpha}_{\delta-2,\tau-2}(\mathbb{C}\times X_0)$ can be integrated over the $X_0$ fibers, and
\begin{align*}
\pi_* f(\zeta) = \int_{X_0 } f(\zeta,\cdot)\omega_{\operatorname{SRF}}|_{X_0}^2, \quad |\pi_* f(\zeta)| \le C (1+|\zeta|)^{\delta-\tau} \|f\|_{C^{0,\alpha}_{\delta-2,\tau-2}(\mathbb{C}\times X_0)}.
\end{align*}
If, in addition, $\min\{\tau, \delta\} > -2 + \alpha$, then for all $\zeta, \zeta'$ with $|\zeta - \zeta'| \le 1$,
\begin{align*}
|\pi_* f(\zeta) - \pi_* f(\zeta')| \le C (1+|\zeta|)^{\delta-\tau} |\zeta - \zeta'|^\alpha \|f\|_{C^{0,\alpha}_{\delta-2,\tau-2}(\mathbb{C}\times X_0)}.
\end{align*}

\end{lemma}

With these weights in hand, we denote by $C^{0,\alpha,\operatorname{ave}}_{\delta-2,\tau-2}(\mathbb{C}\times X_0)$ the subspace of $C^{0,\alpha}_{\delta-2,\tau-2}(\mathbb{C}\times X_0)$ consisting of functions satisfying $\int_{X_0} f(\zeta,\cdot)\,\omega_{\operatorname{SRF}}|_{X_0}^2=0$ for every $\zeta\in\mathbb C$. Likewise, we define $C^{2,\alpha,\operatorname{ave}}_{\delta,\tau}(\mathbb{C}\times X_0)$ as the corresponding subspace of $C^{2,\alpha}_{\delta,\tau}(\mathbb{C}\times X_0)$. 

Finally, the weighted growth conditions ensure that fiberwise $L^2$ integrals are well-defined for all functions in $C^{2,\alpha}_{\delta,\tau}(\mathbb{C}\times X_0)$.

\begin{proposition}[\text{cf.} \cite{li2019gluing} Proposition 3.6]\label{Laplace for C times K3 singular version}

Let $-2 + \alpha < \tau < 0$, and $-2 + \alpha < \delta < 0$. Then the Laplacian $\Delta : C^{2,\alpha,\operatorname{ave}}_{\delta,\tau}(\mathbb{C}\times X_0) \to C^{0,\alpha,\operatorname{ave}}_{\delta-2,\tau-2}(\mathbb{C}\times X_0)$ is an isomorphism whose inverse $P_{t=0} = \Delta^{-1}_{\omega_{\operatorname{SRF},0}}$ is bounded.

\end{proposition}

\section{Global weighted norms and deviation estimates}\label{Global weighted norms and deviation estimates}

\subsection{\texorpdfstring{Weighted H{\"o}lder norm on $M$}{Weighted Holder norm on M}} To enable harmonic analysis on $M$, we introduce a weighted H\"older space. This is done by decomposing $M$ into typical regions, defining H\"older norms according to the geometry of each region, and then patching them together. 

Recall we have fixed the constant $\epsilon_1>0$ in Section \ref{Harmonic analysis on model spaces}.

\textbf{Region 1:} $U_0\cap \{|t|<\epsilon_1\}$, where $U_0$ is the coordinate neighborhood around the nodal point $q_0$, $\epsilon_1>0$ is a small positive number.

Via the embedding map $\Upsilon_\varepsilon$, we have $U_0\cap  \{|t|<\epsilon_1\}\cong \Upsilon_\varepsilon^{-1}\left(U_0\cap  \{|t|<\epsilon_1\}\right)\subset \mathbb{C}^3$.

In this region $\omega_\varepsilon\sim \left(\frac{\varepsilon}{2A_0}\right)^{\frac13}(\Upsilon_\varepsilon^{-1})^*\omega_{\mathbb{C}^3}$.

Recall we have already defined the weighted norm for $(\mathbb{C}^3,\omega_{\mathbb{C}^3})$ by (\ref{Weighted norm on C3}). Via the diffeomorphism $\Upsilon_\varepsilon$, we define the weighted norm $\|\cdot \|_{C^{k,\alpha}_{\delta,\tau}(U_0\cap \{|t|<\epsilon_1\})}\coloneqq \|\cdot\|_{C^{k,\alpha}_{\delta,\tau}\left(\Upsilon_\varepsilon^{-1}(U_0\cap \{|t|<\epsilon_1\}),\omega_{\mathbb{C}^3}\right)}$.

\textbf{Region 2:}   $U_1=\left\{|t|<\epsilon_1,\; r>\Lambda_1|t|^{\frac14},\; r>\varepsilon^{\frac16}\right\}\subset X$. This is the region close to the singular fiber $X_0$ while remaining uniformly away from its nodal point.

By construction, the diffeomorphism $G_0$ identifies $U_1$ with a subset of $\mathbb{C}\times X_0$, namely $U_1 \cong G_0(U_1)\subset \mathbb{C}\times X_0$ and the pulled-back metric satisfies $(G_0^{-1})^*\omega_\varepsilon \sim \frac{\sqrt{-1}A_0}{\varepsilon}\,dt\wedge d\bar t +\omega_{\operatorname{SRF}}|_{X_0}$. 

To normalize the scale in the base direction, consider the rescaling map
\begin{align*}
\upsilon:\mathbb{C}_t\times X_0\longrightarrow \mathbb{C}_\zeta\times X_0,
\quad 
(t,x)\longmapsto
\left(\left(\frac{\varepsilon}{2A_0}\right)^{-\frac{1}{2}}t,x\right).
\end{align*}
The composition $\upsilon\circ G_0$ identifies $U_1$ with a region in the model space $\mathbb{C}_\zeta\times X_0$. 

We then use the weighted H\"older spaces
$C^{k,\alpha}_{\delta,\tau}(\mathbb{C}_\zeta\times X_0)$ with respect to the product metric $\omega_{\operatorname{SRF},0}
=\frac{\sqrt{-1}}{2}\,d\zeta\wedge d\bar\zeta +\omega_{\operatorname{SRF}}|_{X_0}$
introduced in Section 3.2 to define corresponding weighted norms on $U_1$. More precisely, for a function $f$ on $U_1$, we set
\begin{align*}
\|f\|_{C^{k,\alpha}_{\delta,\tau}(U_1)}
\coloneqq
\left\|
\left((\upsilon\circ G_0)^{-1}\right)^*f
\right\|_{C^{k,\alpha}_{\delta,\tau}(\mathbb{C}_\zeta\times X_0,\omega_{\operatorname{SRF},0})}.
\end{align*}

\textbf{Region 3:} $U_2=\{|t|>\frac12\epsilon_1,|y|>\frac12\epsilon_1 \}\subset M$. This is the region uniformly away from the singular fibers.

In this region, $\omega_{\varepsilon}\sim \frac{1}{\varepsilon}\pi^*\omega_B+\omega_X$. We employ the usual $C^{k,\alpha}$ norm with respect to the metric $ \frac{1}{\varepsilon}\pi^*\omega_B+\omega_X$. This weighted norm is denoted by $\|\cdot \|_{C^{k,\alpha}(U_2)}$.

\textbf{Region 4:} $U_3=\{|y|<\epsilon_1\}$. This is the end near $X_\infty$.

We introduce the following weighted H{\"o}lder norm
\begin{equation}\label{weighted norm on U3 1}
\begin{aligned}
\|f\|_{C^{k,\alpha}_{\lambda,\mu}(U_3,\omega_\varepsilon)}=& \sum\limits_{i=0}^k\sup\limits_{x\in U_3}\frac{r^{-\mu+i}}{|y|^\lambda}\left|\nabla^i_{\omega_\varepsilon}f\right|_{\omega_\varepsilon}\\
&+\sup_{\substack{x,x'\in U_3\\d_{\omega_\varepsilon}(x,x')\ll r(x)}}
\frac{r(x)^{-\mu+\alpha+k}}{|y|^\lambda}
\frac{|\nabla^k_{\omega_\varepsilon}f(x)-\nabla^k_{\omega_\varepsilon}f(x')|}{d_{\omega_\varepsilon}(x,x')^\alpha}.
\end{aligned}
\end{equation}

With all the above preparation, we can define the global weighted H{\"o}lder norm for smooth functions on $M$ as follows:
\begin{equation}\label{global weighted norm}
\|f\|_{C^{k,\alpha}_{\delta,\tau,\lambda,\mu,\varepsilon}(M, \omega_\varepsilon)}\coloneqq\varepsilon^{-\frac{\delta}{6}}\|f\|_{C^{k,\alpha}_{\delta,\tau}(U_0\cap \{|t|<\epsilon_1\})}+\|f\|_{C^{k,\alpha}_{\delta,\tau}\left(U_1\right)}+\varepsilon^{\frac12(\delta-\tau)}\|f\|_{C^{k,\alpha}(U_2)}+\varepsilon^{\frac12(\delta-\tau)}\|f\|_{C^{k,\alpha}_{\lambda,\mu}(U_3)}.
\end{equation}

 By \cite{li2019gluing}, on the overlapping regions of $U_0\cap \{|t|<\epsilon_1\}$, $U_1$ and $U_2$, the different weighted norms appearing in (\ref{global weighted norm}) are uniformly equivalent to one another. On the overlapping region $U_2\cap U_3$, both $|y|$ and $r$ are uniformly bounded away from zero, so $\|f\|_{C^{k,\alpha}(U_2)}$ and $\|f\|_{C^{k,\alpha}_{\lambda,\mu}(U_3)}$ are uniformly equivalent to each other on $U_2\cap U_3$. Here, the uniform equivalence can depend on $\epsilon_1$ but not on the collapsing parameter $\varepsilon$.

Similarly, we can define weighted H{\"o}lder norms for tensor fields on $M$. For example, for a $(0,2)$-tensor field $\theta$ on $M$, we define 
\begin{align*}
\|\theta\|_{C^{k,\alpha}_{\delta-2,\tau-2,\lambda,\mu-2,\varepsilon}(M, \omega_\varepsilon)}\coloneqq&\varepsilon^{-\frac{\delta}{6}}\|\theta\|_{C^{k,\alpha}_{\delta-2,\tau-2}(U_0\cap \{|t|<\epsilon_1\})}+\|\theta\|_{C^{k,\alpha}_{\delta-2,\tau-2}\left(U_1\right)}\\
&+\varepsilon^{\frac12(\delta-\tau)}\|\theta\|_{C^{k,\alpha}(U_2)}+\varepsilon^{\frac12(\delta-\tau)}\|\theta\|_{C^{k,\alpha}_{\lambda,\mu-2}(U_3)}.
\end{align*}

\begin{remark}\label{numerical properties of the weighted Holder spaces}

The elementary numerical properties of the weighted H{\"o}lder spaces are summarized as
\begin{itemize}

\item $\| fg \|_{C^{0,\alpha}_{\delta-2,\tau-2,\lambda,\mu-2,\varepsilon}(M)} \le C \| f \|_{C^{0,\alpha}_{0,0,0,0,\varepsilon}(M)} \| g\|_{C^{0,\alpha}_{\delta-2,\tau-2,\lambda,\mu-2,\varepsilon}(M)}$;

\item $\| f \|_{C^{0,\alpha}_{0,0,0,0,\varepsilon}(M)}\le C\max\{\varepsilon^{\frac{\delta-2}{6}},\varepsilon^{\frac12(\tau-\delta)}\}\| f \|_{C^{0,\alpha}_{\delta-2,\tau-2,\lambda,\mu-2,\varepsilon}(M)}$ whenever $\lambda+\frac{\mu-2}{4}>0$.

\end{itemize}
Here $f$ and $g$ are arbitrary elements of the function spaces. The same statements hold for  differential forms.

\end{remark}

\subsection{Estimate for the Ricci potential}\label{Estimate for the Ricci potential}In this subsection, we analyze the $\varepsilon$-smallness of the Ricci potential of the metric ansatz.  

The Ricci potential $h_\varepsilon$ with respect to the metric ansatz $\omega_\varepsilon$ is defined as
\begin{equation}\label{def of Ricci potential}
h_\varepsilon\coloneqq \log\left(\frac{\varepsilon\omega_\varepsilon^3}{3\sqrt{-1}\Omega\wedge \overline{\Omega}}\right).
\end{equation}

\begin{proposition}\label{Ricci potential decay 1}

(1) Let $-\frac{1}{40}<\delta<0$, $\tau=-\frac23$, $\lambda=\frac{13}{24}$ ,$\mu=-\frac{1}{12}$.  Then the Ricci potential satisfies the global estimate
\begin{equation}\label{Ricci potential est 1}
\|h_\varepsilon\|_{C^{0,\alpha}_{\delta-2,\tau-2,\lambda,\mu-2,\varepsilon}(M)}< C\varepsilon^{\frac12\delta+\frac{7}{20}},
\end{equation}where $C=C(\alpha,\delta,\tau,\lambda,\mu,\epsilon_1)$ is a uniform constant independent of $\varepsilon$.

(2) Without weights, the Ricci potential satisfies the global estimate
\begin{equation}\label{Ricci potential est 2}
    \|h_\varepsilon\|_{C^{0,\alpha}_{0,0,0,0,\varepsilon}(M)}< C\varepsilon^{\frac{2}{3}\delta+\frac{1}{60}},
\end{equation}where $C=C(\alpha,\delta,\tau,\lambda,\mu,\epsilon_1)$ is a uniform constant independent of $\varepsilon$.

\end{proposition}

We divide the manifold $M$ into two regions $\{|y|>\frac12\epsilon_1\}$ and $\{|y|<\epsilon_1\}$. The first region lies uniformly away from infinity, and its Ricci potential behavior is already covered by \cite{li2019gluing}. Therefore, we only need to analyze the second region.

First, by an argument identical to that used in the proofs of Lemma 2.7, Lemma 2.8 in \cite{li2019gluing}, we obtain the following estimate on the Ricci potential $h_\varepsilon$.

\begin{lemma}[\text{cf.} \cite{li2019gluing} Lemma 2.7, Lemma 2.8 and Proposition 2.9]\label{Ricci potential decay 2}

Let $-2 < \tau < 0$ and $\frac{3}{4}\tau - \frac{1}{2} < \delta < \frac{2}{3} + \frac{5\tau}{6}$. The Ricci potential on $\{|y|>\frac12\epsilon_1\}$ is estimated by
\begin{equation}
\| h_\varepsilon \|_{C^{0,\alpha}_{\delta-2, \tau-2,\lambda,\mu-2 ,\varepsilon}(\{|y|>\frac12\epsilon_1\})} \le C \varepsilon^{\delta'},
\end{equation}
where $\delta' = \frac{6}{14+\tau}(\frac{2}{3} + \frac{5\tau}{6} - \delta) + \frac{1}{2}\delta - \frac{1}{2}\tau$, $C=C(\alpha,\delta,\tau,\lambda,\mu,\epsilon_1)$ is a uniform constant independent of $\varepsilon$.

In particular, if $\tau=-\frac23$, $\lambda=\frac{13}{24}$, $\mu=-\frac{1}{12}$ and $-\frac{1}{40}<\delta<0$, then
\begin{equation}
\| h_\varepsilon \|_{C^{0,\alpha}_{\delta-2, \tau-2,\lambda,\mu-2 ,\varepsilon}(\{|y|>\frac12\epsilon_1\})} \le C \varepsilon^{\frac{1}{20}\delta+\frac{23}{60}},
\end{equation}where $C=C(\alpha,\delta,\epsilon_1)$ is a uniform constant independent of $\varepsilon$.

\end{lemma}

Now we turn to the region $\{|y|<\epsilon_1\}$.

\begin{lemma}\label{Ricci potential decay 3}

In the region $U_3=\{|y|<\epsilon_1\}$, the Ricci potential is estimated by
\begin{align*}
\|h_\varepsilon\|_{C^{0,\alpha}_{\frac{13}{24},-\frac{25}{12}}(U_3)}<C(\alpha)\varepsilon^{\frac{1}{60}}.
\end{align*}

\end{lemma}

\begin{proof}
Set $\mu=-\frac{1}{12}$, $\lambda=\frac{13}{24}$. 

By direct computation,
\begin{align*}
\exp(h_\varepsilon)=
\frac{\varepsilon\omega_\varepsilon^3}
{3\sqrt{-1}\,\Omega\wedge\overline{\Omega}}
=
1+E_\varepsilon^{(1)}
 +E_\varepsilon^{(2)}
 +E_\varepsilon^{(3)},
\end{align*}
where
\begin{align*}
E_\varepsilon^{(1)}
&=\frac{A_y\,dy\wedge d\bar y\wedge
 \left(\omega_X+\sqrt{-1}\,\partial\bar\partial\phi\right)^2}{dy\wedge d\bar y\wedge
 \Omega_y\wedge\overline{\Omega}_y
}
-1,\\
E_\varepsilon^{(2)}
&=
\frac{
 \varepsilon |y|^2
 \left(\omega_X+\sqrt{-1}\,\partial\bar\partial\phi\right)^3
}{
 3\sqrt{-1}\,dy\wedge d\bar y\wedge
 \Omega_y\wedge\overline{\Omega}_y
},\\
E_\varepsilon^{(3)}
&=
\frac{
 \varepsilon |y|^2\pi^*\beta\wedge
 \left(\omega_X+\sqrt{-1}\,\partial\bar\partial\phi\right)^2
}{
 \sqrt{-1}\,dy\wedge d\bar y\wedge
 \Omega_y\wedge\overline{\Omega}_y
}.
\end{align*}

Set $E_\varepsilon=E_\varepsilon^{(1)}+E_\varepsilon^{(2)}+E_\varepsilon^{(3)}$. It suffices to prove $\left\|E_\varepsilon\right\|_{C^{0,\alpha}_{\frac{13}{24},-\frac{25}{12}}(U_3)}\leq C\varepsilon^{\frac{1}{60}}$. Indeed, the weighted pointwise estimate therefore implies $\|E_\varepsilon\|_{C^0(U_3)}\leq C\varepsilon^{\frac1{60}}$. After decreasing $\varepsilon$ if necessary, we may assume that
$|E_\varepsilon|\leq\frac12$. Since
\[
h_\varepsilon=\log(1+E_\varepsilon),
\]
we have
\[
|h_\varepsilon|\leq 2|E_\varepsilon|,\quad |dh_\varepsilon|_{\omega_\varepsilon}=\frac{|dE_\varepsilon|_{\omega_\varepsilon}}{|1+E_\varepsilon|}\leq 2|dE_\varepsilon|_{\omega_\varepsilon}.
\]
Moreover, the map $t\mapsto\log(1+t)$ is uniformly Lipschitz on
$[-\frac12,\frac12]$, so the same estimate holds for the local
H{\"o}lder seminorms on the rescaled coordinate balls. Consequently,
\[
\|h_\varepsilon\|_{C^{0,\alpha}_{\frac{13}{24},-\frac{25}{12}}(U_3)}
\leq C\left\|E_\varepsilon\right\|_{C^{0,\alpha}_{\frac{13}{24},-\frac{25}{12}}(U_3)}\leq C\varepsilon^{\frac1{60}}.
\]

Since $\beta$ is supported in $\left\{
\frac{1}{2}\delta_1<|y|<2\delta_2
\right\}$, all the weights are uniformly equivalent to constants on its support. The explicit factor $\varepsilon$ therefore gives
\[
\left\|E_\varepsilon^{(3)}\right\|_
{C^{0,\alpha}_{\frac{13}{24},-\frac{25}{12}}(U_3)}
\leq C\varepsilon.
\]
It remains to estimate the first two terms.

\medskip

\noindent
\textbf{Case 1:} $\varepsilon^{\frac{12}{2+\mu}}\leq |y|<\epsilon_1$.

In this region,
\begin{align*}
\omega_\varepsilon
=
\frac{1}{\varepsilon}\pi^*\omega_B+\omega_X
+\sqrt{-1}\,\partial\bar\partial
\left(
\sum_{i=1}^{N}\chi_iG_{y_i}^*\psi_{y_i}
+\chi_\infty\widetilde\phi
\right).
\end{align*}

We first estimate $E_\varepsilon^{(1)}$. Note that
\begin{align*}
E_\varepsilon^{(1)}
=\frac{ \left. \left(\omega_X+\sqrt{-1}\,\partial\bar\partial\phi\right)^2
 \right|_{X_y}}{ \left.\omega_{\mathrm{SRF}}^2\right|_{X_y}}
-1.
\end{align*}
Consequently, Proposition~\ref{fiber deviation 7} gives
\begin{align*}
\left|E_\varepsilon^{(1)}\right|
&\leq
C
\left\|
\left.\omega_\varepsilon\right|_{X_y}
-
\left.\omega_{\mathrm{SRF}}\right|_{X_y}
\right\|_{C^0_{\mu-2}(X_y)}
r^{\mu-2}\\
&\leq
C\varepsilon^{\frac12}
|y|^{1-\frac14(\mu+2)}
r^{\mu-2}\\
&\leq
C\varepsilon^{\frac{1}{60}}
|y|^{\frac{13}{24}}r^{\mu-2}.
\end{align*}
The last inequality follows from $|y|\geq\varepsilon^{\frac{12}{2+\mu}}$, $\mu=-\frac{1}{12}$.

If $g_{\mathrm{SRF}}$ denotes the matrix of
$\left.\omega_{\mathrm{SRF}}\right|_{X_y}$, and $R$ is the
matrix associated with
\[
\left.
\left(\omega_X+\sqrt{-1}\,\partial\bar\partial\phi\right)
\right|_{X_y}
-
\left.\omega_{\mathrm{SRF}}\right|_{X_y},
\]
then
\[
E_\varepsilon^{(1)}
=
\operatorname{tr}\left(g_{\mathrm{SRF}}^{-1}R\right)
+
\det\left(g_{\mathrm{SRF}}^{-1}R\right).
\]
Proposition~\ref{fiber deviation 7} therefore also gives the required estimates for all
vertical derivatives.

We next consider horizontal differentiation. In the region $D=\left\{|z_3|>\frac12\max\{|z_1|,|z_2|\}
\right\}$, we use the holomorphic coordinates $(y,z_1,z_2)$. The corresponding
holomorphic lift of $\frac{\partial}{\partial y}$ is $v=\frac{\partial}{\partial y}=\frac{1}{2z_3}\frac{\partial}{\partial z_3}$ in the $(z_1,z_2,z_3)$-coordinates. Recall that $G_{y_i}$ is
constructed by flowing along $\Xi=\sum_{k=1}^{3}
\frac{\bar z_k}{2|z|^2}
\frac{\partial}{\partial z_k}$. Set $v'=(G_{y_i}^{-1})_*\left(\frac{\partial}{\partial y}\right)$, $\bar v'=(G_{y_i}^{-1})_*\left(\frac{\partial}{\partial\bar y}\right)$, where the vector fields on the right-hand side are the trivial vector
fields on the product. Then $v'(G_{y_i}^*f)=0$, $\bar v'(G_{y_i}^*f)=0$ for every smooth function $f$ on $X_{y_i}$.

Since $v-v'$ and $\bar v-\bar v'$ are vertical vector fields, the
explicit formula for $G_{y_i}$ gives
\begin{align*}
&|v-v'|_{\mathrm{EH}_y}+|\bar v-\bar v'|_{\mathrm{EH}_y}\leq Cr^{-3},\\
&|[v,v']|_{\mathrm{EH}_y}+|[v,\bar v']|_{\mathrm{EH}_y}+|[\bar v,v']|_{\mathrm{EH}_y}+|[\bar v,\bar v']|_{\mathrm{EH}_y}\leq Cr^{-7}.
\end{align*}
More generally, direct differentiation of the same formulas yields
\[
\left|
\nabla_{X_y}^{m}(v-v')
\right|_{\mathrm{EH}_y}
+
\left|
\nabla_{X_y}^{m}(\bar v-\bar v')
\right|_{\mathrm{EH}_y}
\leq C_mr^{-3-m}.
\]
These estimates also control the derivatives of the coefficients of
$v-v'$ that arise when ordinary horizontal derivatives are converted
into vertical covariant derivatives.

Indeed,
\begin{align*}
\left|v(G_{y_i}^*f)\right|=\left|(v-v')(G_{y_i}^*f)\right|\leq
Cr^{-3}
\left|\nabla_{X_{y_i}}f\right|_{\mathrm{EH}_{y_i}}.
\end{align*}
Moreover,
\[
v\left(v(G_{y_i}^*f)\right)
=
(v-v')(v-v')G_{y_i}^*f
-
[v,v']G_{y_i}^*f,
\]
and hence
\[
\left|
\frac{\partial^2}{\partial y^2}G_{y_i}^*f
\right|
\leq
Cr^{-6}
\left|\nabla_{X_{y_i}}^2f\right|_{\mathrm{EH}_{y_i}}
+
Cr^{-7}
\left|\nabla_{X_{y_i}}f\right|_{\mathrm{EH}_{y_i}}.
\]
The analogous estimates hold for
$\partial_{\bar y}^2G_{y_i}^*f$ and
$\partial_y\partial_{\bar y}G_{y_i}^*f$. Proceeding inductively, we
obtain, for $0\leq j\leq k$,
\[
\left|
\frac{\partial^k}
{\partial y^j\partial\bar y^{k-j}}
G_{y_i}^*f
\right|
\leq
Cr^{-3k}r^{\mu-k}
\|f\|_{C^{k,\alpha}_\mu(X_{y_i})}.
\]
The same argument remains valid after taking vertical derivatives,
because the estimates for
$\nabla_{X_y}^{m}(v-v')$ and
$\nabla_{X_y}^{m}(\bar v-\bar v')$ control every additional
commutator.

It remains to estimate the variation of the fiberwise potential.
Lemma~\ref{fiber deviation 4} (f) gives
\[
\left\|
G_{y',y}^*\psi_{y'}-\psi_y
\right\|_{C^{k+2,\alpha}_\beta(X_y)}
\leq
C|y-y'|\,|y'|^{-\frac14(\beta+2)}.
\]
Letting $y'\to y$, we obtain
\[
\left\|\Xi\psi_y\right\|_{C^{k+2,\alpha}_\beta(X_y)}+\left\|\bar\Xi\psi_y\right\|_{C^{k+2,\alpha}_\beta(X_y)}\leq C|y|^{-\frac14(\beta+2)}.
\]
In particular, for $1\leq k,l\leq2$,
\begin{align*}
\left|
\partial_y\partial_{z_k}\partial_{\bar z_l}\psi_y
\right|
&\leq
C\left(
r^{-6}
+
|y|^{-\frac14(\beta+2)}r^{\beta-4}
\right),\\
\left|
\partial_{\bar y}\partial_{z_k}\partial_{\bar z_l}\psi_y
\right|
&\leq
C\left(
r^{-6}
+
|y|^{-\frac14(\beta+2)}r^{\beta-4}
\right).
\end{align*}
Combining these estimates with Proposition~\ref{fiber deviation 7} and using the
horizontal and vertical scales of $\omega_\varepsilon$, we obtain
\[
\left\|E_\varepsilon^{(1)}\right\|_
{C^{0,\alpha}_{\frac{13}{24},-\frac{25}{12}}
\left(
U_3\cap
\{|y|\geq\varepsilon^{\frac{12}{2+\mu}}\}
\right)}
\leq
C\varepsilon^{\frac{1}{60}}.
\]

We now estimate $E_\varepsilon^{(2)}$. Fix a point in $D=\left\{|z_3|>\frac12\max\{|z_1|,|z_2|\}
\right\}$ and use the coordinates $(y,z_1,z_2)$. Write
\[
\omega_X+\sqrt{-1}\,\partial\bar\partial\phi
=
A+B+C+D,
\]
where
\begin{align*}
A
&=
\sqrt{-1}\,
\eta_{\alpha\bar\beta}\,
dz_\alpha\wedge d\bar z_\beta,\\
B
&=
\sqrt{-1}\,
\eta_{y\bar y}\,
dy\wedge d\bar y,\\
C
&=
\sqrt{-1}\,
\eta_{y\bar\beta}\,
dy\wedge d\bar z_\beta,\\
D
&=
\sqrt{-1}\,
\eta_{\alpha\bar y}\,
dz_\alpha\wedge d\bar y.
\end{align*}
A direct expansion gives
\[
\left(
\omega_X+\sqrt{-1}\,\partial\bar\partial\phi
\right)^3
=
3B\wedge A\wedge A
+
6A\wedge C\wedge D.
\]
Let $(\eta^{\alpha\bar\beta})$ denote the inverse of
$(\eta_{\alpha\bar\beta})$. Then
\begin{align*}
B\wedge A\wedge A
&=
\sqrt{-1}\,\eta_{y\bar y}\,
dy\wedge d\bar y\wedge
\left.
\left(
\omega_X+\sqrt{-1}\,\partial\bar\partial\phi
\right)^2
\right|_{X_y},\\
A\wedge C\wedge D
&=
-\frac{\sqrt{-1}}{2}
\sum_{\alpha,\beta}
\eta^{\alpha\bar\beta}
\eta_{y\bar\beta}
\eta_{\alpha\bar y}\,
dy\wedge d\bar y\wedge
\left.
\left(
\omega_X+\sqrt{-1}\,\partial\bar\partial\phi
\right)^2
\right|_{X_y}.
\end{align*}
Therefore,
\[
\left(
\omega_X+\sqrt{-1}\,\partial\bar\partial\phi
\right)^3
=
3\sqrt{-1}
\left(
\eta_{y\bar y}
-
\sum_{\alpha,\beta}
\eta^{\alpha\bar\beta}
\eta_{y\bar\beta}
\eta_{\alpha\bar y}
\right)
dy\wedge d\bar y\wedge
\left.
\left(
\omega_X+\sqrt{-1}\,\partial\bar\partial\phi
\right)^2
\right|_{X_y}.
\]
It follows that
\[
E_\varepsilon^{(2)}
=
\frac{\varepsilon|y|^2}{A_y}
\left(1+E_\varepsilon^{(1)}\right)
\left(
\eta_{y\bar y}
-
\sum_{\alpha,\beta}
\eta^{\alpha\bar\beta}
\eta_{y\bar\beta}
\eta_{\alpha\bar y}
\right).
\]
Thus it suffices to prove the corresponding weighted $C^1$-estimate for the expression in parentheses.

We first estimate $\eta_{y\bar y}$. Fix $i$ such that the point
under consideration lies in the region in which the $i$-th local
model is used. Since
\[
\sum_i\chi_i+\chi_\infty=1,
\]
we have the identity
\[
\phi
=
G_{y_i}^*\psi_{y_i}
+
\sum_{j\neq i}
\chi_j
\left(
G_{y_j}^*\psi_{y_j}
-
G_{y_i}^*\psi_{y_i}
\right)
+
\chi_\infty
\left(
\widetilde\phi-G_{y_i}^*\psi_{y_i}
\right).
\]
Notice that the order of the two terms in the first difference is
important for this identity.

For the first term, the horizontal differentiation estimate gives
\[
\left|
\partial_y\partial_{\bar y}
G_{y_i}^*\psi_{y_i}
\right|
\leq
Cr^{-6}
\left|
\nabla_{X_{y_i}}^2\psi_{y_i}
\right|_{\mathrm{EH}_{y_i}}
\leq Cr^{-6}.
\]
For $j\neq i$, Lemma~\ref{fiber deviation 4} and the cutoff estimates give
\begin{align*}
&\left|\chi_j\partial_y\partial_{\bar y}\left(G_{y_j}^*\psi_{y_j}-G_{y_i}^*\psi_{y_i}\right)\right|\leq Cr^{-6},\\
&\left|
\partial_y\chi_j\,
\partial_{\bar y}
\left(
G_{y_j}^*\psi_{y_j}
-
G_{y_i}^*\psi_{y_i}
\right)
\right|
\leq
C\varepsilon^{-\frac12}|y|^{-1}r^{-2},\\
&\left|
\partial_{\bar y}\chi_j\,
\partial_y\left(G_{y_j}^*\psi_{y_j}-G_{y_i}^*\psi_{y_i}\right)\right|\leq C\varepsilon^{-\frac12}|y|^{-1}r^{-2},\\
&\left|\partial_y\partial_{\bar y}\chi_j\left(G_{y_j}^*\psi_{y_j}-G_{y_i}^*\psi_{y_i}\right)\right|\leq C\varepsilon^{-\frac12}
|y|^{-1-\frac14(\mu+2)}r^\mu.
\end{align*}
Similarly,
\begin{align*}
&\left|\partial_y\chi_\infty\,\partial_{\bar y}\left(\widetilde\phi-G_{y_i}^*\psi_{y_i}\right)\right|\leq C\varepsilon^{-\frac12}|y|^{-1}r^{-2},\\
&\left|\partial_{\bar y}\chi_\infty\,
\partial_y\left(\widetilde\phi-G_{y_i}^*\psi_{y_i}\right)\right|\leq C\varepsilon^{-\frac12}|y|^{-1}r^{-2},\\
&\left|\partial_y\partial_{\bar y}\chi_\infty\left(\widetilde\phi-G_{y_i}^*\psi_{y_i}\right)\right|
\leq C\varepsilon^{-\frac12}|y|^{-1-\frac14(\mu+2)}r^\mu.
\end{align*}

When $H\geq2|y|^{\frac{5}{6}}$, the horizontal differentiation argument
gives
\[
\left|
\chi_\infty
\partial_y\partial_{\bar y}
\left(
\widetilde\phi-G_{y_i}^*\psi_{y_i}
\right)
\right|
\leq Cr^{-6}.
\]
When $H\leq2|y|^{\frac{5}{6}}$, we need to analyze the term $\sqrt{H+|y|}-\Theta_\infty$. Since
\[
\partial_y|y|=O(1),
\quad
\partial_y\partial_{\bar y}|y|=O(|y|^{-1}),
\]
direct differentiation gives
\begin{align*}
&\left|\partial_y\widetilde\phi|+|\partial_{\bar y}\widetilde\phi\right|\leq Cr^{-2},\\
&\left|\partial_y\partial_{\bar y}\widetilde\phi\right|\leq C\left(r^{-6}+r^2|y|^{-2}\right).
\end{align*}
Consequently,
\[
|\eta_{y\bar y}|\leq C\left(r^{-6}+\varepsilon^{-\frac12}|y|^{-1}r^{-2}+\varepsilon^{-\frac12}|y|^{-1-\frac14(\mu+2)}r^\mu\right)
\]
when $H\geq2|y|^{\frac{5}{6}}$; whereas
\[|\eta_{y\bar y}|\leq C\left(r^{-6}+r^2|y|^{-2}+\varepsilon^{-\frac12}|y|^{-1}r^{-2}+\varepsilon^{-\frac12}|y|^{-1-\frac14(\mu+2)}r^\mu\right)
\]
when $H\leq2|y|^{\frac{5}{6}}$.

Using $|y|\geq\varepsilon^{\frac{12}{2+\mu}}$, $|y|\leq H$, together with $r\leq C|y|^{\frac{5}{24}}$ when $H\leq2|y|^{\frac{5}{6}}$, we obtain
\[
\varepsilon|y|^2|\eta_{y\bar y}|
\leq
C\varepsilon^{\frac{1}{60}}
|y|^{\frac{13}{24}}r^{\mu-2}.
\]

We next estimate the first derivatives of $\eta_{y\bar y}$. The
general horizontal differentiation argument gives
\[
\left|
\partial_y^2\partial_{\bar y}
G_{y_i}^*\psi_{y_i}
\right|
+
\left|
\partial_y\partial_{\bar y}^2
G_{y_i}^*\psi_{y_i}
\right|
\leq Cr^{-10}.
\]
Differentiating each cutoff term one more time and applying the same
estimates as above gives
\begin{align*}
&\left|\partial_y^2\partial_{\bar y}\left[\chi_j\left(G_{y_j}^*\psi_{y_j}-G_{y_i}^*\psi_{y_i}\right)\right]\right|+\left|\partial_y\partial_{\bar y}^2\left[\chi_j\left(G_{y_j}^*\psi_{y_j}-G_{y_i}^*\psi_{y_i}\right)\right]\right|\\
\leq&
C\left(r^{-10}+\varepsilon^{-1}|y|^{-2}r^{-2}+\varepsilon^{-1}|y|^{-2-\frac14(\mu+2)}r^\mu\right).
\end{align*}

When $H\geq2|y|^{\frac{5}{6}}$, the same estimate holds with $\chi_j$
replaced by $\chi_\infty$ and
$G_{y_j}^*\psi_{y_j}$ replaced by $\widetilde\phi$.

When $H\leq2|y|^{\frac{5}{6}}$, direct differentiation of
$\sqrt{H+|y|}$ gives
\begin{align*}
&|\partial_y^2\partial_{\bar y}\widetilde\phi|+|\partial_y\partial_{\bar y}^2\widetilde\phi|\leq C\left(r^{-10}+r^2|y|^{-3}\right),\\
&|\partial_{z_\alpha}\partial_y\partial_{\bar y}\widetilde\phi|+|\partial_{\bar z_\alpha}\partial_y\partial_{\bar y}\widetilde\phi|\leq C\left(r^{-8}+|y|^{-2}\right).
\end{align*}
It follows that, when $H\geq2|y|^{\frac{5}{6}}$,
\begin{align*}
&\left|\partial_y\eta_{y\bar y}\right|+\left|\partial_{\bar y}\eta_{y\bar y}\right|\leq C\left(r^{-10}+\varepsilon^{-1}|y|^{-2}r^{-2}+\varepsilon^{-1}|y|^{-2-\frac14(\mu+2)}r^\mu
\right),\\
&\left|\partial_{z_\alpha}\eta_{y\bar y}\right|+\left|\partial_{\bar z_\alpha}\eta_{y\bar y}\right|\leq C\left(r^{-8}+|y|^{-1}r^{-4}+\varepsilon^{-\frac12}|y|^{-1-\frac14(\mu+2)}r^{\mu-2}\right).
\end{align*}
When $H\leq2|y|^{\frac{5}{6}}$, we instead have
\begin{align*}
&\left|\partial_y\eta_{y\bar y}|+|\partial_{\bar y}\eta_{y\bar y}\right|\leq C\left(r^{-10}+r^2|y|^{-3}+\varepsilon^{-1}|y|^{-2}r^{-2}+\varepsilon^{-1}|y|^{-2-\frac14(\mu+2)}r^\mu
\right),\\
&\left|\partial_{z_\alpha}\eta_{y\bar y}\right|+\left|\partial_{\bar z_\alpha}\eta_{y\bar y}\right|\leq C\left(r^{-8}+|y|^{-2}+|y|^{-1}r^{-4}+\varepsilon^{-\frac12}|y|^{-1-\frac14(\mu+2)}r^{\mu-2}\right).
\end{align*}

The horizontal and vertical covectors satisfy
\[
|dy|_{\omega_\varepsilon}
+
|d\bar y|_{\omega_\varepsilon}
\leq C\varepsilon^{\frac12}|y|,
\quad
|dz_\alpha|_{\omega_\varepsilon}
+
|d\bar z_\alpha|_{\omega_\varepsilon}
\leq Cr.
\]
Therefore,
\begin{align*}
|d\eta_{y\bar y}|_{\omega_\varepsilon}
\leq C\left[\varepsilon^{\frac12}|y|\left(|\partial_y\eta_{y\bar y}|+|\partial_{\bar y}\eta_{y\bar y}|
\right)+r\sum_{\alpha=1}^{2}\left(|\partial_{z_\alpha}\eta_{y\bar y}|+|\partial_{\bar z_\alpha}\eta_{y\bar y}|\right)
\right].
\end{align*}
Substituting the preceding estimates and using again $|y|\geq\varepsilon^{\frac{12}{2+\mu}}$, $r\leq C|y|^{\frac{5}{24}}$ when $H\leq2|y|^{\frac{5}{6}}$, we obtain
\[
\varepsilon|y|^2
|d\eta_{y\bar y}|_{\omega_\varepsilon}
\leq
C\varepsilon^{\frac{1}{60}}
|y|^{\frac{13}{24}}r^{\mu-3}.
\]

We now consider the mixed term. Since the fiber metric is comparable
to the Eguchi--Hanson metric in this region, its inverse satisfies $\left|\eta^{\alpha\bar\beta}\right|\leq Cr^2$. Moreover,
\[
|\eta_{y\bar\beta}|
+
|\eta_{\alpha\bar y}|
\leq
C\left(
1+r^{-4}
+
|y|^{-\frac14(\mu+2)}r^{\mu-2}
\right).
\]
Hence, for each fixed pair $(\alpha,\beta)$, with no summation over
$\alpha,\beta$,
\begin{align*}
\left|
\eta^{\alpha\bar\beta}
\eta_{y\bar\beta}
\eta_{\alpha\bar y}
\right|
&\leq
Cr^2
\left(
1+r^{-4}
+
|y|^{-\frac14(\mu+2)}r^{\mu-2}
\right)^2\\
&\leq
C\left(
r^{-6}
+
|y|^{-\frac14(\mu+2)}r^{\mu-4}
+
|y|^{-\frac12(\mu+2)}r^{2\mu-2}
\right).
\end{align*}
There are only four pairs $(\alpha,\beta)$, so summing the
componentwise estimate gives
\[
\varepsilon|y|^2
\left|
\sum_{\alpha,\beta}
\eta^{\alpha\bar\beta}
\eta_{y\bar\beta}
\eta_{\alpha\bar y}
\right|
\leq
C\varepsilon^{\frac{1}{60}}
|y|^{\frac{13}{24}}r^{\mu-2}.
\]

For the first derivative, we use
\[
d\eta^{\alpha\bar\beta}
=
-
\sum_{\gamma,\delta}
\eta^{\alpha\bar\gamma}
\left(d\eta_{\delta\bar\gamma}\right)
\eta^{\delta\bar\beta}.
\]
The vertical metric estimates and Lemma~\ref{fiber deviation 4} give
\begin{align*}
&\left|\partial_y\eta^{\alpha\bar\beta}\right|+\left|\partial_{\bar y}\eta^{\alpha\bar\beta}\right|\leq C\left(r^{-2}+|y|^{-\frac14(\mu+2)}r^\mu
\right),\\
&\left|\partial_{z_p}\eta^{\alpha\bar\beta}\right|+\left|\partial_{\bar z_p}\eta^{\alpha\bar\beta}\right|
\leq C.
\end{align*}
They also give
\begin{align*}
&\left|\partial_{z_p}\eta_{y\bar\beta}\right|+\left|\partial_{\bar z_p}\eta_{y\bar\beta}\right|\leq C\left(r^{-6}+|y|^{-\frac14(\mu+2)}r^{\mu-4}\right),\\
&\left|\partial_{z_p}\eta_{\alpha\bar y}\right|+\left|\partial_{\bar z_p}\eta_{\alpha\bar y}\right| \leq C\left(r^{-6}+|y|^{-\frac14(\mu+2)}r^{\mu-4}\right).
\end{align*}
For the holomorphic vertical derivative, for example, the K{\"a}hler $\partial_{z_p}\eta_{y\bar\beta}
=
\partial_y\eta_{p\bar\beta}$
identity reduces the estimate to the previously established horizontal
variation of the fiber metric. The antiholomorphic derivative is
estimated directly by differentiating the corresponding mixed
component. The conjugate estimates apply to
$\eta_{\alpha\bar y}$.

When $H\geq2|y|^{\frac{5}{6}}$, we have
\begin{align*}
|\partial_y\eta_{y\bar\beta}|
+
|\partial_{\bar y}\eta_{y\bar\beta}|
+
|\partial_y\eta_{\alpha\bar y}|
+
|\partial_{\bar y}\eta_{\alpha\bar y}|\leq
C\left(
|y|^{-1}r^{-4}
+
\varepsilon^{\frac12}
|y|^{-1-\frac14(\mu+2)}r^{\mu-2}
\right).
\end{align*}
When $H\leq2|y|^{\frac{5}{6}}$, the explicit formula for
$\widetilde\phi$ instead gives
\begin{align*}
|\partial_y\eta_{y\bar\beta}|
+
|\partial_{\bar y}\eta_{y\bar\beta}|
+
|\partial_y\eta_{\alpha\bar y}|
+
|\partial_{\bar y}\eta_{\alpha\bar y}|\leq
C\left(
|y|^{-2}
+
\varepsilon^{\frac12}
|y|^{-1-\frac14(\mu+2)}r^{\mu-2}
\right).
\end{align*}

For each fixed pair $(\alpha,\beta)$, with no summation over
$\alpha,\beta$, the product rule gives
\begin{align*}
d\left(\eta^{\alpha\bar\beta}\eta_{y\bar\beta}\eta_{\alpha\bar y}\right)=\left(d\eta^{\alpha\bar\beta}\right)
\eta_{y\bar\beta}
\eta_{\alpha\bar y}+\eta^{\alpha\bar\beta}\left(d\eta_{y\bar\beta}\right)\eta_{\alpha\bar y}+
\eta^{\alpha\bar\beta}
\eta_{y\bar\beta}
\left(d\eta_{\alpha\bar y}\right).
\end{align*}
Using the horizontal and vertical covector bounds above, together with
the corresponding component estimates, we obtain separately
\begin{align*}
&\varepsilon|y|^2
\left|
\left(d\eta^{\alpha\bar\beta}\right)
\eta_{y\bar\beta}
\eta_{\alpha\bar y}
\right|_{\omega_\varepsilon}\leq C\varepsilon^{\frac{1}{60}}
|y|^{\frac{13}{24}}r^{\mu-3},\\
&\varepsilon|y|^2
\left|
\eta^{\alpha\bar\beta}
\left(d\eta_{y\bar\beta}\right)
\eta_{\alpha\bar y}
\right|_{\omega_\varepsilon}
\leq
C\varepsilon^{\frac{1}{60}}
|y|^{\frac{13}{24}}r^{\mu-3},\\
&\varepsilon|y|^2
\left|
\eta^{\alpha\bar\beta}
\eta_{y\bar\beta}
\left(d\eta_{\alpha\bar y}\right)
\right|_{\omega_\varepsilon}
\leq
C\varepsilon^{\frac{1}{60}}
|y|^{\frac{13}{24}}r^{\mu-3}.
\end{align*}
Summing over the four possible pairs gives
\[
\varepsilon|y|^2
\left|
d\left(
\sum_{\alpha,\beta}
\eta^{\alpha\bar\beta}
\eta_{y\bar\beta}
\eta_{\alpha\bar y}
\right)
\right|_{\omega_\varepsilon}
\leq
C\varepsilon^{\frac{1}{60}}
|y|^{\frac{13}{24}}r^{\mu-3}.
\]

Finally, differentiating the formula for $E_\varepsilon^{(2)}$
gives
\begin{align*}
dE_\varepsilon^{(2)}
={}&
\varepsilon\,
d\left(\frac{|y|^2}{A_y}\right)
\left(1+E_\varepsilon^{(1)}\right)
\left(
\eta_{y\bar y}
-
\sum_{\alpha,\beta}
\eta^{\alpha\bar\beta}
\eta_{y\bar\beta}
\eta_{\alpha\bar y}
\right)\\
&+
\frac{\varepsilon|y|^2}{A_y}
\,dE_\varepsilon^{(1)}
\left(
\eta_{y\bar y}
-
\sum_{\alpha,\beta}
\eta^{\alpha\bar\beta}
\eta_{y\bar\beta}
\eta_{\alpha\bar y}
\right)\\
&+
\frac{\varepsilon|y|^2}{A_y}
\left(1+E_\varepsilon^{(1)}\right)
d\left(
\eta_{y\bar y}
-
\sum_{\alpha,\beta}
\eta^{\alpha\bar\beta}
\eta_{y\bar\beta}
\eta_{\alpha\bar y}
\right).
\end{align*}
The expansion of $A_y$, the estimate for $E_\varepsilon^{(1)}$,
and the estimates just established imply
\begin{align*}
&\left|E_\varepsilon^{(2)}\right|\leq C\varepsilon^{\frac{1}{60}}
|y|^{\frac{13}{24}}r^{\mu-2},\\
&\left|dE_\varepsilon^{(2)}\right|_{\omega_\varepsilon}\leq
C\varepsilon^{\frac{1}{60}}
|y|^{\frac{13}{24}}r^{\mu-3}.
\end{align*}
The mean-value inequality on the rescaled coordinate balls used in the
definition of the weighted H{\"o}lder norm now gives
\[
\left\|E_\varepsilon^{(2)}\right\|_
{C^{0,\alpha}_{\frac{13}{24},-\frac{25}{12}}
\left(
U_3\cap
\{|y|\geq\varepsilon^{\frac{12}{2+\mu}}\}
\right)}
\leq
C\varepsilon^{\frac{1}{60}}.
\]
Thus
\[
\left\|h_\varepsilon\right\|_
{C^{0,\alpha}_{\frac{13}{24},-\frac{25}{12}}
\left(
U_3\cap
\{|y|\geq\varepsilon^{\frac{12}{2+\mu}}\}
\right)}
\leq
C\varepsilon^{\frac{1}{60}}.
\]

\medskip

\noindent
\textbf{Case 2:} $|y|\leq\varepsilon^{\frac{12}{2+\mu}}$.

In this region,
\[
\omega_\varepsilon=\frac{\sqrt{-1}\,A_y}{\varepsilon|y|^2}dy\wedge d\bar y
+
\omega_X+\sqrt{-1}\,\partial\bar\partial\phi,
\]
where
\[
\phi=c_\infty+\gamma_1\left(\frac{H}{|y|^{\frac{5}{6}}}\right)
G_\infty^*\left(\psi_0^{(\infty)}-c_\infty\right)+\gamma_2\left(\frac{H}{|y|^{\frac{5}{6}}}\right)
\left(\sqrt{H+|y|}-\Theta_\infty\right).
\]
As before,
\[
\exp(h_\varepsilon)
=
1+E_\varepsilon^{(1)}+E_\varepsilon^{(2)}.
\]

Proposition~\ref{fiber deviation 2} gives
\begin{align*}
\left|E_\varepsilon^{(1)}\right|
&\leq C\left\|\omega_y-\left.\omega_{\mathrm{SRF}}\right|_{X_y}\right\|_{C^0_{\mu-2}(X_y)}r^{\mu-2}\\
&\leq C|y|^{\frac{7}{12}-\frac{5\mu}{24}}r^{\mu-2}\\
&\leq C\varepsilon^{\frac{1}{60}}
|y|^{\frac{13}{24}}r^{\mu-2}.
\end{align*}
Applying the horizontal differentiation argument from Case~1, we
obtain
\[
\left\|E_\varepsilon^{(1)}\right\|_{C^{0,\alpha}_{\frac{13}{24},-\frac{25}{12}}\left(U_3\cap\{|y|\leq\varepsilon^{\frac{12}{2+\mu}}\}
\right)}
\leq C\varepsilon^{\frac{1}{60}}.
\]

It remains to estimate $E_\varepsilon^{(2)}$. On the region
$H<|y|^{\frac{5}{6}}$, Proposition~\ref{prop of omega 1} gives
\[
\frac{\varepsilon|y|^2\left(\omega_X+\sqrt{-1}\,\partial\bar\partial\phi\right)^3}{\operatorname{Vol}_E}=\frac{\varepsilon|y|}{8(H+|y|)^{\frac{1}{2}}}.
\]
Differentiating the right-hand side and using
\[
|dy|_{\omega_\varepsilon}
\leq C\varepsilon^{\frac12}|y|,
\quad
|dH|_{\omega_\varepsilon}\leq Cr^3,
\]
we obtain
\begin{align*}
\varepsilon
\left|d\left(\frac{|y|}{(H+|y|)^{\frac{1}{2}}}\right)\right|_{\omega_\varepsilon}\leq C\varepsilon\left(\frac{\varepsilon^{\frac12}|y|}{(H+|y|)^{\frac{1}{2}}}+\frac{\varepsilon^{\frac12}|y|^2}{(H+|y|)^{\frac{3}{2}}}+\frac{|y|r^3}{(H+|y|)^{\frac{3}{2}}}
\right).
\end{align*}
Since $|y|\leq H= r^4$, these estimates imply
\[
\left\|E_\varepsilon^{(2)}\right\|_{C^{0,\alpha}_{\frac{13}{24},-\frac{25}{12}}\left(U_3\cap
\left\{|y|\leq\varepsilon^{\frac{12}{2+\mu}},\ H<|y|^{\frac{5}{6}}\right\}\right)}\leq C\varepsilon^{\frac{1}{60}}.
\]

On the region $H>2|y|^{\frac{5}{6}}$, the explicit formula for $G_\infty$
gives
\[
\frac{\varepsilon|y|^2\left(\omega_X+\sqrt{-1}\,\partial\bar\partial\phi\right)^3}{\sqrt{-1}\,dy\wedge d\bar y\wedge\Omega_y\wedge\overline{\Omega}_y}\leq C\varepsilon\frac{|y|^2}{H^{\frac{3}{2}}}\leq C\varepsilon^{\frac{1}{60}}|y|^{\frac{13}{24}}r^{\mu-2},
\]
and
\[
\varepsilon|y|^2\left|d\left(\frac{\left(\omega_X+\sqrt{-1}\,\partial\bar\partial\phi\right)^3}{\sqrt{-1}\,dy\wedge d\bar y\wedge\Omega_y\wedge\overline{\Omega}_y}\right)
\right|_{\omega_\varepsilon}\leq C\varepsilon\frac{|y|^2}{H^{\frac{7}{4}}}\leq C\varepsilon^{\frac{1}{60}}
|y|^{\frac{13}{24}}r^{\mu-3}.
\]

Finally, on the transition region $|y|^{\frac{5}{6}}\leq H\leq2|y|^{\frac{5}{6}}$, the derivatives of the cutoff functions must also be included. The
resulting estimates are
\[
\frac{\varepsilon|y|^2\left(\omega_X+\sqrt{-1}\,\partial\bar\partial\phi\right)^3}{\sqrt{-1}\,dy\wedge d\bar y\wedge\Omega_y\wedge\overline{\Omega}_y}\leq C\varepsilon\left(H^{\frac{1}{2}}+\frac{|y|^2}{H^{\frac{3}{2}}}+\frac{|y|}{H^{\frac{1}{2}}}\right)\leq C\varepsilon^{\frac{1}{60}}|y|^{\frac{13}{24}}r^{\mu-2},
\]
and
\begin{align*}
\varepsilon|y|^2\left|d\left(\frac{\left(\omega_X+\sqrt{-1}\,\partial\bar\partial\phi\right)^3}{\sqrt{-1}\,dy\wedge d\bar y\wedge \Omega_y\wedge\overline{\Omega}_y}\right)\right|_{\omega_\varepsilon}\leq C\varepsilon
\left(H^{\frac{1}{4}}+\frac{|y|^2}{H^{\frac{7}{4}}}+\frac{|y|}{H^{\frac{3}{4}}}
\right)\leq C\varepsilon^{\frac{1}{60}}|y|^{\frac{13}{24}}r^{\mu-3}.
\end{align*}

The three regions $H<|y|^{\frac{5}{6}}$, $|y|^{\frac{5}{6}}\leq H\leq2|y|^{\frac{5}{6}}$, $H>2|y|^{\frac{5}{6}}$ cover the whole region under consideration. The weighted
$C^1$-estimates above, followed by the mean-value inequality on the
rescaled coordinate balls, therefore give
\[
\left\|E_\varepsilon^{(2)}\right\|_
{C^{0,\alpha}_{\frac{13}{24},-\frac{25}{12}}
\left(
U_3\cap
\{|y|\leq\varepsilon^{\frac{12}{2+\mu}}\}
\right)}
\leq
C\varepsilon^{\frac{1}{60}}.
\]
Combining this with the estimate for $E_\varepsilon^{(1)}$, we conclude that
\[
\left\|h_\varepsilon\right\|_
{C^{0,\alpha}_{\frac{13}{24},-\frac{25}{12}}(U_3)}
\leq
C(\alpha)\varepsilon^{\frac{1}{60}}.
\]
\end{proof}

\begin{remark}

From the above computation, one observes that the main loss comes from the fiberwise deviation term $\frac{A_y(\omega_X+\sqrt{-1}\partial\bar{\partial}\phi)\big{|}_{X_y}^2}{\Omega_y\wedge\overline{\Omega}_y}-1$. In this term, we do not have any collapsing factor a priori. Therefore, we must rely on the condition $|y|\lesssim \varepsilon^{\frac{12}{2+\mu}}$ and the decay rate in $|y|$ to extract a small collapsing factor. This is precisely what Yang Li did when dealing with the corresponding term for his ansatz. 

The essential difference from his case lies in the behavior of $\varepsilon|y|^2\frac{(\omega_X+\sqrt{-1}\partial\bar{\partial}\phi)^3}{dy\wedge d\bar{y}\wedge\Omega_y\wedge\overline{\Omega}_y}$ in our model at infinity: this term is much better behaved for the semi-Ricci-flat approximation. The reason is that the base term in our ansatz at infinity carries an extra factor of $\frac{1}{|y|^2}$, which ensures decay in $|y|$. In Yang Li's case, however, when he computes the Ricci potential, he encounters a term of the form $\varepsilon \frac{(\omega_X+\sqrt{-1}\partial\bar{\partial}\phi)^3}{dy\wedge d\bar{y}\wedge\Omega_y\wedge\overline{\Omega}_y}$. Without regularizing his ansatz near the node, the semi-Ricci-flat approximation fails inside the quantum scale, meaning the Ricci potential becomes unbounded inside $\{r < \varepsilon^{\frac16}\}$ as $|y| \to 0$.

The above observation might suggest that we could simply use the genuine semi-Ricci-flat ansatz $\omega_{\operatorname{SRF}}^{\varepsilon}$ at infinity instead of the glued one, which would significantly improve the behavior of the Ricci potential. However, we do not adopt this approach because of the difficulty in analyzing the precise dependence of the semi-Ricci-flat potential $\psi_y$ on $y$ near the singular node. Our gluing ansatz, while not optimal in terms of behavior, is quite explicit and easier to handle. More importantly, the collapsing and decaying properties described above are sufficient for our purposes.

\end{remark}

Finally, we prove Proposition \ref{Ricci potential decay 1}.

\begin{proof}

(1) Combine Lemmas \ref{Ricci potential decay 2} and \ref{Ricci potential decay 3}. Then convert the norm on $U_3$ to the global norm.

(2) Use the elementary numerical properties of the weighted H{\"o}lder spaces listed in Remark \ref{numerical properties of the weighted Holder spaces}.

\end{proof}

\subsection{Metric deviation}\label{Metric deviation}When estimating the Ricci potential, we have in essence proved that $\omega_\varepsilon$ is close to various simpler model metrics on their respective regions. We now record a coarser formulation of these bounds, which is valid over slightly larger domains. These estimates will be useful in constructing an approximate inverse for the Laplacian of the collapsing metric.

Let $\Lambda_1$, $\epsilon_1$, and $\epsilon_2$ be as introduced earlier. Let $\Lambda_2 \gg 1$ and $\Lambda_3 \gg 1$ be two large constants, and let $0 < \epsilon_3, \epsilon_4 \le \epsilon_1$ be small parameters, all of which will be chosen independently of $\varepsilon$. We require $\epsilon_3 \ll (\frac{1}{\Lambda_1 \Lambda_2})^{12}$ to be sufficiently small that the set $\{r > \Lambda_2^{-2}(\varepsilon^{\frac{1}{10}} + \varepsilon^{\frac{1}{12}}\rho'^{\frac{1}{6}}),\ |t| < \epsilon_3\}$ is contained in $U_1$. For technical convenience, we additionally impose $\epsilon_3^{\frac{1}{3}}\Lambda_2^8 \ll \epsilon_3^{\frac{1}{7}}$.

\begin{proposition}\label{metric deviation}

Assume that $\epsilon_3,\epsilon_4, \Lambda_2, \Lambda_3$ are given as above. Let $g_\varepsilon$, $g_{\operatorname{SRF}}$ and $g_{y'}$ be the corresponding Riemannian metrics associated to the K{\"a}hler forms $\omega_\varepsilon$, $\omega_{\operatorname{SRF}}$ and $\omega_{y'}$ respectively. Set $dt\odot d\bar{t}\coloneqq \frac12(dt\otimes  d\bar{t}+d\bar{t}\otimes dt)$.

Then, for all sufficiently small $\varepsilon>0$, the following estimates hold.

(a) In the region $\left\{r \lesssim \varepsilon^{\frac{1}{10}} + \varepsilon^{\frac{1}{12}}\rho'^{\frac16}, |t| < \epsilon_3\right\} \subset U_0$, the metric $\omega_\varepsilon$ deviates from the scaled $\mathbb{C}^3$ model metric by
\begin{align*}
  \left\| \omega_\varepsilon - \left( \frac{\varepsilon}{2A_0} \right)^{\frac{1}{3}} (\Upsilon_\varepsilon^{-1})^*\omega_{\mathbb{C}^3} \right\|_{C^{0,\alpha}_{0,0,0,0,\varepsilon }} \le C(\alpha)\epsilon_3^{\frac{1}{7}}.
\end{align*}

(b)  In the region $\left\{r > \Lambda_2^{-2}(\varepsilon^{\frac{1}{10}} + \varepsilon^{\frac{1}{12}}\rho'^{\frac16}), |t| < \epsilon_3\right\} \subset U_1$, which can be identified as a subset of $\mathbb{C}\times X_0 $ via the trivialization $G_0$, the metric $\omega_\varepsilon$ deviates from the product metric by
 \begin{align*}
    \left\| g_\varepsilon - G_0^* \left(\frac{2A_0 }{\varepsilon}  dt \odot d\bar{t}+ g_{\operatorname{SRF}}|_{X_0} \right) \right\|_{C^{0,\alpha}_{0,0,0,0,\varepsilon}} \le C(\alpha)\epsilon_3^{\frac{1}{7}}.
  \end{align*}

 (c) Let $X_{y'}$ be any fiber with $y'$ in the region $\left\{|t|>\frac12\epsilon _3,|y|>\frac12\epsilon_4\right\}$.
 
 If $y'=t'$ is in the coordinate neighborhood $(W_0,t)$, then for $|t-t'|\leq \Lambda_3\varepsilon^{\frac12}\ll \epsilon_2|t'|$, the trivialization $G_{t'}$ is well-defined. The metric $\omega_\varepsilon$ deviates from the product metric by
  \begin{align*}
    \left\| g_\varepsilon - G_{t'}^* \left(\frac{ 2A_{t'}}{\varepsilon}  dt \odot d\bar{t} + g_{\operatorname{SRF}}|_{X_{t'}} \right) \right\|_{C^{1,\alpha}_{0,0,0,0,\varepsilon}} \le C(\alpha, \epsilon_3) \Lambda_3 \varepsilon^{\frac12}.
  \end{align*}

If $y'$ is in the coordinate neighborhood $(W_\infty,y)$, then for $|y-y'|\leq \Lambda_3\varepsilon^{\frac12}|y'|\ll \epsilon_2|y'|$, the trivialization $G_{y'}$ is well-defined. The metric $\omega_\varepsilon$ deviates from the product metric by
\begin{align*}
    \left\| g_\varepsilon - G_{y'}^* \left(\frac{ 2A_{y'}}{\varepsilon|y'|^2}  dy \odot d\bar{y} + g_{\operatorname{SRF}}|_{X_{y'}} \right) \right\|_{C^{1,\alpha}_{0,0,0,0,\varepsilon}} \le C(\alpha) \Lambda_3 \varepsilon^{\frac12}.
  \end{align*}

If $y'$ lies outside both coordinate charts, one may simply work in a local coordinate $w$ centered at $y'$. On the region $\left\{w \in B \mid d_{\omega_B}(w, y') < \Lambda_3\varepsilon^{\frac12}\right\}$, an estimate of the same form holds, with the base term in the product approximation replaced by the local Euclidean metric on the base that best approximates $\omega_B$; the dominant term on the right-hand side remains to be $C(\alpha)\,\Lambda_3\varepsilon^{\frac12}$.

(d) Let $X_{y'}$ be any fiber with $y'$ in the region $\left\{0<|y|<\epsilon_4\right\}$. For $|y - y'| \le \Lambda_3 |y'|\varepsilon^{\frac{1}{2}} \ll \epsilon_2 |y'|$, the trivialization $G_{y'}$ is well defined. The metric $\omega_\varepsilon$ deviates from the product metric $g_{\operatorname{prod},y'}\coloneqq \frac{2A_{y'}}{\varepsilon |y'|^2}   dy \odot d\bar{y} +g_{y'} $ in the region $U_{y'}=\{|y - y'| \le \Lambda_3|y'| \varepsilon^{\frac12}\}$ by
 \begin{align*}
    \left\| \left(G_{y'}^{-1}\right)^*g_\varepsilon -g_{\operatorname{prod},y'}\right\|_{C^{1,\alpha}_{0,0}\left(U_{y'},G_{y'}^* \left(g_{\operatorname{prod},y'}\right)\right)} \le C(\alpha) \epsilon_4^{\frac{1}{24}},
  \end{align*}where the weighted H{\"o}lder norm $C^{1,\alpha}_{0,0}\left(G_{y'}(U_{y'}), g_{\operatorname{prod},y'}\right)$ is defined by replacing $\mathbb{C}\times X_{y}$ and $g_{\operatorname{P},y}$ in (\ref{Weighted norm on C Xy}) by $G_{y'}(U_{y'})$ and $g_{\operatorname{prod},y'}$, respectively.

\end{proposition}

\begin{proof}

Parts (a)–(c) follow from Proposition 2.10 in \cite{li2019gluing}. We therefore explain only part (d) in detail.

Use $G_{y'}$ to identify $U_{y'}$ with $G_{y'}(U_{y'})=\left\{y\in \mathbb{C}:|y-y'|\leq\Lambda_3\varepsilon^{\frac{1}{2}}|y'|\right\}\times X_{y'}$.

Set
\[
\widehat\omega_\varepsilon\coloneqq \left(G_{y'}^{-1}\right)^*\omega_\varepsilon,\qquad\widehat J\coloneqq \left(G_{y'}^{-1}\right)^*J,\qquad\widehat g_\varepsilon\coloneqq \left(G_{y'}^{-1}\right)^*g_\varepsilon.
\]
The product K{\"a}hler form and product metric are $\omega_{\operatorname{prod},y'}=\frac{\sqrt{-1}\,A_{y'}}{\varepsilon |y'|^2}\,dy\wedge d\bar y+\omega_{y'}$ and $g_{\operatorname{prod},y'}=\frac{2A_{y'}}{\varepsilon |y'|^2}\,dy\odot d\bar y+g_{y'}$ respectively.

We use the normalized horizontal vector fields $\mathcal{H}_y=\varepsilon^{\frac{1}{2}}|y'|\partial_y$, $\mathcal{H}_{\bar y}=\varepsilon^{\frac{1}{2}}|y'|\partial_{\bar y}$. Then $\mathcal{H}^a\nabla_{X_{y'}}^b f$ denotes a component of
the full product covariant derivative of $f$ with $a$
horizontal slots and $b$ vertical slots. All such components are
measured with respect to $g_{\operatorname{prod},y'}$.

The explicit expansion of $G_{y',y}$ gives
\begin{align*}
G_{y',y}(w)=w+\frac{y'-y}{2|w|^2}\bar w+O\left(\frac{|y-y'|^2}{|w|^3}\right).
\end{align*}
Since $\frac{|y-y'|}{r^4}\leq C\Lambda_3\varepsilon^{\frac12}$ on $U_{y'}$, differentiating this expansion gives
\begin{equation}\label{eq:complex-structure-estimate}
\left\|\widehat J-J_{\operatorname{prod}}\right\|_{C^3_{0,0}}\leq C\Lambda_3\varepsilon^{\frac{1}{2}}.
\end{equation}

We have the following decomposition
\begin{equation*}
\begin{aligned}
\left(G_{y'}^{-1}\right)^*\left(\omega_X+\sqrt{-1}\,\partial_J\bar\partial_J\phi\right)-\omega_{y'}=&
\left[\left(G_{y'}^{-1}\right)^*\omega_X-\left.\omega_X\right|_{X_{y'}}\right]\\
&+\left[\left(G_{y'}^{-1}\right)^*\left(\sqrt{-1}\,\partial_J\bar\partial_J\phi\right)-\sqrt{-1}\,\partial_{J_{\mathrm{prod}}}\bar\partial_{J_{\mathrm{prod}}}
\left(\left(G_{y'}^{-1}\right)^*\phi\right)\right]\\
&+\sqrt{-1}\,\partial_{J_{\mathrm{prod}}}
\bar\partial_{J_{\mathrm{prod}}}\left[\left(G_{y'}^{-1}\right)^*\phi-\widetilde\phi_{y'}
\right].
\end{aligned}
\end{equation*}
We estimate these three terms separately.

The explicit expansion
\[
G_{y',y}(w)=w+\frac{y'-y}{2|w|^2}\,\bar w+O\left(\frac{|y-y'|^2}{|w|^3}\right)
\]
and the smoothness of $\omega_X$ gives, for $a+b\leq2$,
\begin{align*}
\left|\mathcal H^a\nabla_{X_{y'}}^b\left[\left(G_{y'}^{-1}\right)^*\omega_X-\left.\omega_X\right|_{X_{y'}}\right]\right|_{g_{\mathrm{prod},y'}}\leq C\Lambda_3\varepsilon^{\frac{1}{2}}r^{2-b}.
\end{align*}
Hence
\begin{equation*}
\left\|\left(G_{y'}^{-1}\right)^*\omega_X-\left.\omega_X\right|_{X_{y'}}\right\|_{C^2_{0,0}}\leq C\Lambda_3\varepsilon^{\frac{1}{2}}.
\end{equation*}

The explicit potential estimates and the calculation used in Lemma~\ref{Ricci potential decay 3} imply that, for $0\leq a+b\leq4$,
\begin{align*}
&\left|\mathcal{H}^a\nabla_{X_{y'}}^b\left[\left(G_{y'}^{-1}\right)^*\left(G_{y_i}^*\psi_{y_i}\right)-c_\infty
\right]\right|_{g_{\operatorname{prod},y'}}\leq C\varepsilon^{\frac{a}{2}}r^{2-b},\\
&\left|\mathcal{H}^a\nabla_{X_{y'}}^b\left[\left(G_{y'}^{-1}\right)^*\widetilde\phi-c_\infty\right]\right|_{g_{\operatorname{prod},y'}}\leq C\varepsilon^{\frac{a}{2}}r^{2-b}.
\end{align*}

The cutoff functions satisfy $\left|\mathcal{H}^a\chi_i\right|+\left|\mathcal{H}^a\chi_\infty\right|\leq C_a$. Consequently, using the local finiteness of the partition of unity, we obtain
\begin{equation}\label{eq:Phi-product-derivatives}
\left|\mathcal{H}^a\nabla_{X_{y'}}^b\left[\left(G_{y'}^{-1}\right)^*\phi
\right]\right|_{g_{\operatorname{prod},y'}}\leq Cr^{2-b},\quad 1\leq a+b\leq4.
\end{equation}

For a real-valued function $f$, we have by definition that $\sqrt{-1}\,\partial_J\bar\partial_Jf=-\frac12d(df\circ J)$. Therefore,
\begin{align*}
\left(G_{y'}^{-1}\right)^*\left(\sqrt{-1}\,\partial_J\bar\partial_J\phi\right)-\sqrt{-1}\,\partial_{J_{\operatorname{prod}}}\bar\partial_{J_{\operatorname{prod}}}\left(\left(G_{y'}^{-1}\right)^*\phi\right)=-\frac12d\left[d\left(\left(G_{y'}^{-1}\right)^*\phi\right)\circ\left(\widehat J-J_{\operatorname{prod}}\right)\right].
\end{align*}
Applying the product rule to the above identity, and using\eqref{eq:complex-structure-estimate} and\eqref{eq:Phi-product-derivatives}, gives
\begin{align*}
\left\|\left(G_{y'}^{-1}\right)^*\left(\sqrt{-1}\,\partial_J\bar\partial_J\phi\right)-\sqrt{-1}\,\partial_{J_{\operatorname{prod}}}\bar\partial_{J_{\operatorname{prod}}}\left(\left(G_{y'}^{-1}\right)^*\phi\right)\right\|_{C^2_{0,0}}\leq C\Lambda_3\varepsilon^{\frac{1}{2}}.
\end{align*}

We now estimate $\sqrt{-1}\, \partial_{J_{\operatorname{prod}}}\bar\partial_{J_{\operatorname{prod}}}\left(\left(G_{y'}^{-1}\right)^*\phi-\widetilde{\phi}_{y'}\right)$.

If $\chi_i$ is nonzero somewhere in the product neighborhood, then $|y_i-y'|\leq C\Lambda_3\varepsilon^{\frac{1}{2}}|y'|$. The argument is divided according to whether the product
covariant derivative contains horizontal slots.

If $a=0$, all the covariant slots are vertical. Lemma~\ref{fiber deviation 4} and Proposition~\ref{fiber deviation 2} give, for $b\leq4$,
\begin{align}\label{eq:vertical-psi-difference}
\left|\nabla_{X_{y'}}^b\left[\left(G_{y'}^{-1}\right)^*\left(G_{y_i}^*\psi_{y_i}\right)-\widetilde\phi_{y'}\right]\right|_{g_{\operatorname{prod},y'}}\leq C\left(\Lambda_3\varepsilon^{\frac{1}{2}}
|y'|^{1-\frac{\beta+2}{4}}+|y'|^{\frac{7}{12}-\frac{5\beta}{24}}
\right)
r^{\beta-b}.
\end{align}
Similarly,
\begin{align}\label{eq:vertical-tilde-phi-difference}
\left|\nabla_{X_{y'}}^b\left[\left(G_{y'}^{-1}\right)^*\widetilde\phi-\widetilde\phi_{y'}\right]\right|_{g_{\operatorname{prod},y'}}\leq C\Lambda_3\varepsilon^{\frac{1}{2}}|y'|^{1-\frac{\beta+2}{4}}r^{\beta-b}.
\end{align}
These are estimates for the all-vertical components of the full product covariant derivative tensors. They do not assert that the functions themselves are restricted to the fibers.

If $a>0$, then at least one horizontal slot occurs. In this case, we do not differentiate the fiberwise smallness estimates \eqref{eq:vertical-psi-difference} and \eqref{eq:vertical-tilde-phi-difference}. Instead, we estimate the two potentials separately.

Since $\widetilde\phi_{y'}$ is constant in the horizontal
direction, we have
\begin{align}\label{eq:horizontal-psi-difference}
\left|\mathcal{H}^a\nabla_{X_{y'}}^b\left[\left(G_{y'}^{-1}\right)^*\left(G_{y_i}^*\psi_{y_i}\right)-\widetilde\phi_{y'}\right]\right|_{g_{\operatorname{prod},y'}}\leq C\varepsilon^{\frac{a}{2}}r^{2-b},\quad
a>0,\quad a+b\leq4
\end{align}
and 
\begin{align}\label{eq:horizontal-tilde-phi-difference}
\left|\mathcal{H}^a\nabla_{X_{y'}}^b\left[\left(G_{y'}^{-1}\right)^*\widetilde\phi-\widetilde\phi_{y'}
\right]\right|_{g_{\operatorname{prod},y'}}\leq C\varepsilon^{\frac{a}{2}}r^{2-b},\quad a>0,\quad a+b\leq4.
\end{align}

Using $\chi_\infty+\sum_i\chi_i=1$, we write
\begin{align*}
\left(G_{y'}^{-1}\right)^*\phi-\widetilde\phi_{y'}=\chi_\infty\left[\left(G_{y'}^{-1}\right)^*\widetilde\phi-\widetilde\phi_{y'}
\right]+\sum_i\chi_i\left[\left(G_{y'}^{-1}\right)^*\left(G_{y_i}^*\psi_{y_i}\right)-\widetilde\phi_{y'}\right].
\end{align*}
When differentiating the above identity, if every horizontal derivative falls on the cutoffs, we use \eqref{eq:vertical-psi-difference} and \eqref{eq:vertical-tilde-phi-difference}. If at least one horizontal derivative falls on a potential, we use \eqref{eq:horizontal-psi-difference} and \eqref{eq:horizontal-tilde-phi-difference}.

Fix $k\leq2$, and consider a component of the $k$-th product
covariant derivative of
\[
\sqrt{-1}\,
\partial_{J_{\operatorname{prod}}}
\bar\partial_{J_{\operatorname{prod}}}
\left(
\left(G_{y'}^{-1}\right)^*\phi-\widetilde\phi_{y'}
\right)
\]
having $a$ horizontal slots. The total number of horizontal and
vertical slots is $2+k$.

If $a=0$, then the component is purely vertical, and
\begin{equation*}
\begin{aligned}
&r^k\left|\left(\nabla_{g_{\operatorname{prod},y'}}\right)^k\left[\sqrt{-1}\,\partial_{J_{\operatorname{prod}}}
\bar\partial_{J_{\operatorname{prod}}}\left(\left(G_{y'}^{-1}\right)^*\phi-\widetilde\phi_{y'}
\right)\right]_{\mathrm{vert}}\right|_{g_{\operatorname{prod},y'}}\\
\leq &C
\left(\Lambda_3\varepsilon^{\frac{1}{2}}|y'|^{1-\frac{\beta+2}{4}}+|y'|^{\frac{7}{12}-\frac{5\beta}{24}}\right)r^{\beta-2}.
\end{aligned}
\end{equation*}

If $a>0$, then the components for which all horizontal derivatives fall on the cutoffs are bounded by $C\left(\Lambda_3\varepsilon^{\frac{1}{2}}|y'|^{1-\frac{\beta+2}{4}}+|y'|^{\frac{7}{12}-\frac{5\beta}{24}}\right)r^{\beta-2+a}$. The components for which at least one horizontal derivative falls on a potential are bounded by $C\varepsilon^{\frac{1}{2}}r^a$. Thus,
\begin{align*}
\left\|\sqrt{-1}\,\partial_{J_{\operatorname{prod}}}\bar\partial_{J_{\operatorname{prod}}}\left(\left(G_{y'}^{-1}\right)^*\phi-\widetilde\phi_{y'}\right)\right\|_{C^2_{0,0}}&\leq C \left[\left(\Lambda_3\varepsilon^{\frac{1}{2}}|y'|^{1-\frac{\beta+2}{4}}+|y'|^{\frac{7}{12}-\frac{5\beta}{24}}\right)|y'|^{\frac{\beta-2}{4}}\right]+C\varepsilon^{\frac{1}{2}}\\
&\leq C\left(\Lambda_3\varepsilon^{\frac{1}{2}}+|y'|^{\frac{23}{288}}\right)\\
&\leq C(\alpha)\epsilon_4^{\frac{1}{24}},
\end{align*}where we take $\beta=-\frac{1}{12}$ and choose $\varepsilon>0$ sufficiently small. 

Combining all the above estimates, we conclude that 
\begin{align*}
\left\|\left(G_{y'}^{-1}\right)^*\left(\omega_X+\sqrt{-1}\,\partial_J\bar\partial_J\phi\right)
-\omega_{y'}\right\|_{C^2_{0,0}}\leq C(\alpha)\epsilon_4^{\frac{1}{24}}.
\end{align*}

It remains to compare the base terms. A direct computation yields
\begin{align*}
\left|\frac{\sqrt{-1}A_y}{\varepsilon|y|^2}-\frac{\sqrt{-1}A_{y'}}{\varepsilon|y'|^2}\right|&=\left|\frac{A_y|y'|^2-A_{y'}|y|^2}{\varepsilon|y|^2|y'|^2}\right|\\
&\le \frac{A_y(|y'|+|y|)(||y|-|y'||)+\left|A_{y'}-A_y\right||y|^2}{\varepsilon|y|^2|y'|^2}\\
&\le \frac{C\Lambda_3}{\varepsilon^{\frac12}|y|^2}+\frac{C\Lambda_3}{\varepsilon^{\frac12}|y|}
\end{align*}
Hence
\begin{align*}
\left|\left(\frac{\sqrt{-1}A_y}{\varepsilon|y|^2}-\frac{\sqrt{-1}A_{y'}}{\varepsilon|y'|^2}\right)dy\wedge d\bar{y}\right|_{\omega_{\operatorname{prod},y'}}\le C\Lambda_3\varepsilon^{\frac12}.
\end{align*}After differentiating with respect to the rescaled horizontal variable, we obtain
\begin{equation*}
\left\|\left(\frac{\sqrt{-1}\,A_y}{\varepsilon |y|^2}-\frac{\sqrt{-1}\,A_{y'}}{\varepsilon |y'|^2}
\right)dy\wedge d\bar y\right\|_{C^2_{0,0}}\leq C\Lambda_3\varepsilon^{\frac{1}{2}}.
\end{equation*}
The term $\pi^*\beta$, when present, is $O(\varepsilon)$ in
this norm. Hence
\begin{equation}\label{eq:full-form-comparison}
\left\|\widehat\omega_\varepsilon-\omega_{\operatorname{prod},y'}\right\|_{C^2_{0,0}}\leq C(\alpha)\epsilon_4^{\frac{1}{24}}.
\end{equation}

Finally,
\begin{align*}
\widehat g_\varepsilon-g_{\operatorname{prod},y'}=\left(\widehat\omega_\varepsilon-\omega_{\operatorname{prod},y'}
\right)\left(\,\cdot\,,J_{\operatorname{prod}}\,\cdot\,\right)+\widehat\omega_\varepsilon\left(\,\cdot\,,
\left(\widehat J-J_{\operatorname{prod}}\right)\,\cdot\,\right).
\end{align*}
Using \eqref{eq:complex-structure-estimate} and \eqref{eq:full-form-comparison}, we obtain \begin{equation*}
\left\|\widehat g_\varepsilon-g_{\operatorname{prod},y'}\right\|_{C^2_{0,0}}\leq C(\alpha)\epsilon_4^{\frac{1}{24}}.
\end{equation*}
The mean-value estimate on the rescaled coordinate balls therefore
gives
\begin{equation*} 
\left\|\widehat g_\varepsilon-g_{\operatorname{prod},y'}\right\|_{C^{1,\alpha}_{0,0}}\leq C(\alpha)\epsilon_4^{\frac{1}{24}}.
\end{equation*}

\end{proof}

\begin{remark}\label{metric deviation higher order}

(a) and (b) cannot be upgraded to higher orders, solely due to the presence of the non-smooth term $\sqrt{-1}A_t\,dt\wedge d\bar{t}$ in $\omega_{\varepsilon}$. If this term were replaced by $\sqrt{-1}A_0\,dt\wedge d\bar{t}$, higher-order generalizations would likewise be possible in the corresponding regions.

\end{remark}

\section{Decomposition, patching and parametrix}\label{Decomposition, patching and parametrix}

The discussion in this section is carried out under Setting~(A). 

Let $\mathcal{A}^{0,\alpha}_{\delta-2,\tau-2,\lambda,\mu-2,\varepsilon}(M)$ denote the space of $\partial\bar{\partial}$-exact $(1,1)$-forms completed with respect to the $C^{0,\alpha}_{\delta-2,\tau-2,\lambda,\mu-2,\varepsilon}$ norm on $2$-forms, where we assume $\lambda - \frac12 + \frac{\mu-\alpha}{4} > 0$. The trace with respect to $\omega_\varepsilon$ gives a bounded linear map
\begin{equation*}
\operatorname{Tr}_{\omega_\varepsilon} : \mathcal{A}^{0,\alpha}_{\delta-2,\tau-2,\lambda,\mu-2,\varepsilon}(M) \longrightarrow C^{0,\alpha}_{\delta-2,\tau-2,\lambda,\mu-2,\varepsilon}(M).
\end{equation*}

To carry out harmonic analysis on $M$, it is convenient to work with the subspace of functions having vanishing integral:
\begin{equation*}
\left(C^{0,\alpha}_{\delta-2,\tau-2,\lambda,\mu-2,\varepsilon}(M)\right)_0
\coloneqq \left\{ f \in C^{0,\alpha}_{\delta-2,\tau-2,\lambda,\mu-2,\varepsilon}(M)
\;\Big|\; \int_M f\,\omega_\varepsilon^3 = 0 \right\},
\end{equation*}
where the integral is understood in the Lebesgue sense. Owing to the condition $\lambda - \frac12 + \frac{\mu-\alpha}{4} > 0$, every function in $C^{0,\alpha}_{\delta-2,\tau-2,\lambda,\mu-2,\varepsilon}(M)$ decays exponentially with respect to the distance function. Therefore, it is routine  to verify that the subspace $\left(C^{0,\alpha}_{\delta-2,\tau-2,\lambda,\mu-2,\varepsilon}(M)\right)_0$ remains a Banach space.

\begin{proposition}\label{inverse of Laplacian}

Let $\tau=-\frac23$, $-\frac{1}{40}<\delta < 0$, $\mu=-\frac{1}{12}$, $\lambda=\frac{13}{24}$, $0<\alpha<\frac{1}{12}$. Then there exists a right inverse $\mathcal{R}$ to $\mathrm{Tr}_{\omega_\varepsilon}$ on the subspace of functions with vanishing integral,
\begin{align*}
\mathcal{R} : \left(C^{0,\alpha}_{\delta-2,\tau-2,\lambda,\mu-2,\varepsilon}(M)\right)_0\to \mathcal{A}^{0,\alpha}_{\delta-2,\tau-2,\lambda,\mu-2,\varepsilon}(M),
\end{align*}
whose operator norm satisfies $\|\mathcal{R}\| \le C(\epsilon_1,\epsilon_3,\epsilon_4,\Lambda_3,\delta,\tau,\lambda,\mu,\alpha)$, with the constant $C$ independent of $\varepsilon$.

\end{proposition}

\begin{remark}

 The right inverse $\mathcal{R}$ can be formally understood as $\mathcal{R} = 2\sqrt{-1}\partial\bar{\partial}\Delta_{\omega_\varepsilon}^{-1}$, mapping a real-valued function to a real $(1,1)$-form.

\end{remark}

\begin{remark}

The above proposition holds for more general indices satisfying $-2+\alpha < \delta < 0$, $-2+\alpha < \tau < 0$, $-2+\alpha<\mu<0$, $\lambda-\frac12+\frac{\mu-\alpha}{4}>0$. For convenience, we just specify those concrete indices.

\end{remark}

From now on, we fix the parameters as follows:
\begin{equation*}
\tau = -\frac{2}{3}, \quad -\frac{1}{40} < \delta < 0, \quad \mu = -\frac{1}{12}, \quad \lambda = \frac{13}{24}, \quad 0 < \alpha < \frac{1}{12}.
\end{equation*}
With these choices, $\delta$ automatically avoids the discrete set of exceptional values appearing in Proposition~\ref{Laplace for C3}.

The parameters $\delta, \tau, \lambda, \mu, \alpha, \epsilon_1, \epsilon_2, \Lambda_1$ are now fixed once and for all, and all constants are allowed to depend on them. Additional parameters $\epsilon_3, \epsilon_4, \Lambda_2, \Lambda_3$ will be introduced and fixed as the argument proceeds. Whenever a uniform constant $C$ appears without explicit indication of its dependence, it is understood that $C$ may depend on all the parameters listed above. Moreover, we always assume that $\varepsilon$ is chosen sufficiently small relative to all other fixed parameter choices.

For the remainder of this section, let $f\in \left(C^{0,\alpha}_{\delta-2,\tau-2,\lambda,\mu-2,\varepsilon}(M)\right)_0$. Our goal is to construct an approximate solution to the Poisson equation $\Delta_{\omega_\varepsilon} u = f$. This construction forms the core of the proof of Proposition~\ref{inverse of Laplacian}.

We shall decompose $f$ both spatially and spectrally, so that the harmonic analysis becomes simpler on each piece. For the region near the singular fiber $X_0$, the decomposition closely parallels that of \cite{li2019gluing}; we therefore only state the corresponding results without spelling out the details. The essential difference lies in the treatment of the end and the overlap region.

Define a partition of unity $\{\chi_i'\}_{i=0}^{\infty}$ on the base $B=Y\backslash\{p_\infty\}\cong \mathbb{C}$ such that
\begin{itemize}

\item $\chi_0'=1$ for $|t| \le \frac12\epsilon_3$ and $\operatorname{supp}\chi_0'\subset \{|t| < \frac23\epsilon_3\}$, where $0 < \epsilon_3 \le \epsilon_1$ is some small number to be fixed independent of $\varepsilon$, satisfying the constraints listed in Subsection \ref{Metric deviation}.

\item For each $i = 1, 2, \dots$, the function $\chi_i'$ is supported on $\{|t| \ge \frac12\epsilon_3\}$, its support has characteristic length scale $\sim \varepsilon^{\frac12}$ measured with respect to the $\omega_B$ metric, and it possesses a marked point $y_i'$ situated at the center of that support. We may also arrange that $0 \le \chi_i' \le 1$, and that the whole family $\{\chi_i'\}$ enjoys uniform $C^k$ estimates relative to the rescaled metric $\frac{1}{\varepsilon}\omega_B$ for any prescribed $k \in \mathbb{N}$. Furthermore, at any given point of $Y$, the number of $\chi_i'$ that do not vanish is uniformly bounded, independently of $\varepsilon$.

\end{itemize}

Let $\{1,...,N'\}$ denote the indices for which $y_i'$ lies in the region $\left\{|t|>\frac{\epsilon_3}{2},|y|>\frac{\epsilon_4}{2}\right\}$; Let $\{N'+1,N'+2,...\}$ denote the indices for which $y_i'$ lies in the region $\left\{|y|<\frac{\epsilon_4}{2}\right\}$.

We also need to construct cutoff functions to decompose the function near the singular fiber into regions modeled by $\mathbb{C}^3$ and $\mathbb{C}\times X_0 $. Let $\Lambda_2 \gg 1$ be a large number to be chosen independent of $\varepsilon$. On the region $\{|t| < \epsilon_1\}$, we introduce the cutoff functions
\begin{align*}
\eta_1 = \gamma_1\left( \frac{r \Lambda_2}{\varepsilon^{\frac{1}{10}} + \varepsilon^{\frac{1}{12}} \rho'^{\frac16}} \right), \quad \eta_2 = 1 - \eta_1,
\end{align*}where $\gamma_1$ is the smooth cutoff function introduced in Subsection \ref{Rough description of the ansatz}.

Now decompose $f$ into pieces:
\begin{itemize}
    \item For $i = 1,2, \dots$, we use the trivialization
\begin{align*}
G_{y_i'}:\left\{x\in X:|y_i'-\pi(x)|<\epsilon_2|y_i'|\right\}\to \left\{y\in \mathbb{C}:|y-y_i'|<\epsilon_2|y_i'|\right\}\times X_{y_i'}\subset \mathbb{C}\times X_{y_i'}
\end{align*}
to identify $\chi_i' f$ as a compactly supported function $(G_{y_i'}^{-1})^*(\chi_i' f)$ on $\mathbb{C}\times X_{y_i'}$, and decompose it, via fiberwise integration, into the sum of a function $\tilde{f}_i$ with zero fiberwise average on $\mathbb{C}\times X_{y_i'}$, and a function $\tilde{f}_i'$ purely on the base:
\begin{align*}
   (G_{y_i'}^{-1})^*(\chi_i' f) = \tilde{f}_i + \tilde{f}_i'.
\end{align*}
Here
\begin{align*}
\tilde{f}_i'(y)=\left\{
	\begin{aligned}
&\int_{X_{y_i'}}(G_{y_i',y}^{-1})^*(\chi_i'f)\omega_{\operatorname{SRF}}^2|_{X_{y_i'}} \quad &&\text{for}\quad  i=1,...,N'\\
   &  \int_{X_{y_i'}}(G_{y_i',y}^{-1})^*(\chi_i'f)\omega_{y_i'}^2 \quad  &&\text{for}\quad  i=N'+1,...
    \end{aligned}
	\right.
\end{align*}and $\tilde{f}_i= (G_{y_i'}^{-1})^*(\chi_i' f)-\tilde{f}_i'$. 

Define 
\begin{align*}
f_i'(y)=G_{y_i',y}^*\tilde{f}_i'=\left\{
	\begin{aligned}
&\int_{X_{y}}\chi_i'f\left(G_{y_i',y}^*\omega_{\operatorname{SRF}}|_{X_{y_i'}}\right)^2 \quad &&\text{for}\quad  i=1,...,N'\\
   &  \int_{X_{y}}\chi_i'f\left(G_{y_i',y}^*\omega_{y_i'}\right)^2 \quad  &&\text{for}\quad  i=N'+1,...
    \end{aligned}
	\right.
\end{align*}and $f_i=\chi_i' f-f_i'$.

Clearly, both $f_i$ and $f_i'$ are supported within the support of $\chi_i'$. Their $C^{0,\alpha}_{\delta-2,\tau-2,\lambda,\mu-2,\varepsilon}$ norms are controlled by the corresponding norm of $f$, since fiberwise integration is a bounded operator.

 \item For $i=0$, we use the trivialization 
 \begin{align*}
G_0:\left\{x\in X:|t|<\epsilon_1\right\}\left\backslash\left\{r\le \Lambda_1 |t|^{\frac14}\right.\right\}\to U' \subset\{|t|<\epsilon_1\}\times X_0\subset \mathbb{C}\times X_0,
 \end{align*} 
 to identify $\chi_0'\eta_1 f$ as a compactly supported function $(G_0^{-1})^*(\chi_0'\eta_1 f)$ on $\mathbb{C}\times X_0$. By integrating over the fibers of $\mathbb{C} \times X_0$, we obtain a locally defined function on the base $\mathbb{C}$, given by
\begin{align*}
    \tilde{f}_0'(t) = \frac{\int_{\{t\}\times X_0} (G_0^{-1})^*(\eta_1 f) \omega_{\operatorname{SRF}}|_{X_0}^2}{\int_{\{t\}\times X_0} (G_0^{-1})^*(\eta_1) \omega_{\operatorname{SRF}}|_{X_0}^2}.
\end{align*}
We then set
\begin{align*}
f_0'(t)=G_0^*\tilde{f}_0'=\frac{\int_{X_t} \eta_1 f \left(G^*_{0,t} \omega_{\operatorname{SRF}}|_{X_0}\right)^2}{\int_{X_t} \eta_1\left(G^*_{0,t} \omega_{\operatorname{SRF}}|_{X_0}\right)^2}
\end{align*}
which satisfies the norm bound $\|\chi_0'f_0'\|_{C^{0,\alpha}_{\delta-2,\tau-2,\lambda,\mu-2,\varepsilon}}\le C\|f\|_{C^{0,\alpha}_{\delta-2,\tau-2,\lambda,\mu-2,\varepsilon}}$
    
\end{itemize}

Write
\begin{align*}
f_{0,1} = \chi_0'\eta_1 (f - f_0'(t)), \quad f_{0,2} = \chi_0'\eta_2 (f - f_0'(t)).
\end{align*}This gives a decomposition
\begin{align*}
\chi_0'f = f_{0,1} + f_{0,2} + \chi_0'f_0'.
\end{align*}

The construction above thus yields the following:
\begin{itemize}
    \item functions $f_{0,1}, f_1, f_2, \dots$, each enjoying a suitable fiberwise mean-zero property;
    \item a function $f_{0,2}$ supported near the nodal point $q_0$;
    \item a function on the base $f_B = \chi_0'f_0' + \sum_{i=1}^\infty f_i'$, which satisfies the weighted H\"older estimate
    \[
    \|f_B\|_{C^{0,\alpha}_{\delta-2,\tau-2,\lambda,\mu-2,\varepsilon}(M)} \le C \|f\|_{C^{0,\alpha}_{\delta-2,\tau-2,\lambda,\mu-2,\varepsilon}(M)},
    \]
    with a constant $C$ independent of $\varepsilon$. In fact, as will be noted later, with additional effort one can establish the stronger bound
    \[
    \|f_B\|_{C^{0,\alpha}_{\delta-2,\tau-2,\lambda,0,\varepsilon}(M)} \le C \|f\|_{C^{0,\alpha}_{\delta-2,\tau-2,\lambda,\mu-2,\varepsilon}(M)},
    \]
    again with $C$ independent of $\varepsilon$.
\end{itemize}

These functions together give a decomposition of $f$:
\begin{equation}\label{eq:f-decomposition}
f = \left(f_{0,1} + \sum_{i=1}^\infty f_i\right) + f_{0,2} + f_B,
\end{equation}where the infinite sum is locally finite and therefore well defined.

We first treat $f_{0,1}$ and $f_{0,2}$. 

The function $f_{0,1}$ is supported on the support of $\chi'_0\eta_1$ and may therefore be regarded as a function on $\mathbb{C} \times X_0$; moreover, it possesses a suitable fiberwise average zero property. To proceed, we introduce an auxiliary cutoff function $\widetilde{\chi}'_0$ on $B$ that is supported in $\{|t| < \epsilon_3\}$, identically equal to $1$ on $\{|t| \le \frac{3}{4}\epsilon_3\}$, and satisfies the estimate
\begin{align*}
    \left|\nabla_{\frac{1}{\varepsilon}\omega_{\operatorname{FS}}}^k \widetilde{\chi}'_0\right| \le C(k) (\varepsilon^{\frac12}\epsilon_3^{-1})^k,
\end{align*}where $\omega_{\operatorname{FS}}$ is the Fubini-Study metric on $Y=\mathbb{P}^1$.
We also need a logarithmic cutoff function
\begin{align*}
\tilde{\eta}_1 = \gamma_1\left( \frac{\log\left( \frac{r\Lambda_2^3}{\varepsilon^{\frac{1}{10}}+\varepsilon^{\frac{1}{12}}\rho'^{\frac16}} \right)}{\log \Lambda_2} \right),
\end{align*}
which equals 1 on the support of $f_{0,1}$, and whose gradient is supported in the range $\frac{1}{\Lambda_2^2} \le \frac{r}{\varepsilon^{\frac{1}{10}}+\varepsilon^{\frac{1}{12}}\rho'^{\frac16}} \le \frac{1}{\Lambda_2}$. Then we set
\begin{equation}
P_{0,1}f = \widetilde{\chi}'_0 \tilde{\eta}_1 P_{t=0} f_{0,1},
\end{equation}
where $P_{t=0}$ is the Green operator on $\mathbb{C}\times X_0 $ associated to the product metric, as constructed in Proposition~\ref{Laplace for C times K3 singular version}.

We treat $f_{0,2}$ in a similar fashion. This function is supported on the support of $\chi'_0\eta_2$ and can therefore be viewed as a function on $\mathbb{C}^3$. We introduce another logarithmic cutoff function
\begin{align*}
\tilde{\eta}_2 = \gamma_2 \left( \frac{2 \log \left( \frac{r \Lambda_2^{\frac32}}{2 \left(\varepsilon^{\frac{1}{10}}+\varepsilon^{\frac{1}{12}}\rho'^{\frac16}\right)} \right)}{\log \Lambda_2} \right),
\end{align*}
which equals $1$ on the support of $f_{0,2}$, and whose gradient is supported in the range $\frac{2}{\Lambda_2} \le \frac{r}{\varepsilon^{\frac{1}{10}}+\varepsilon^{\frac{1}{12}}\rho'^{\frac16}} \le \frac{2}{\Lambda_2^{\frac12}} \ll 1$. Then we set
\begin{equation}
P_{0,2}f = \left( \frac{\varepsilon}{2A_0} \right)^{\frac{1}{3}} \widetilde{\chi}'_0 \tilde{\eta}_2 P_{\mathbb{C}^3} f_{0,2},
\end{equation}
where $P_{\mathbb{C}^3}$ is the Green operator on $(\mathbb{C}^3, \omega_{\mathbb{C}^3})$ constructed in Proposition \ref{Laplace for C3}.

\begin{lemma}[\text{cf.} \cite{li2019gluing} Lemma 3.8]\label{approximate invert 2}

One can select $\Lambda_2 \gg 1$, $\epsilon_3 \ll 1$ subject to the constraints in Subsection \ref{Metric deviation}, so that for $\varepsilon$ sufficiently small with respect to all previous choices, the following estimates hold:
\begin{align*}
&\| \Delta_{\omega_\varepsilon} P_{0,1} f - f_{0,1} \|_{C^{0,\alpha}_{\delta-2,\tau-2,\lambda,\mu-2,\varepsilon}(M)} \le \frac{1}{100} \| f \|_{C^{0,\alpha}_{\delta-2,\tau-2,\lambda,\mu-2,\varepsilon}(M)};\\
&\| \Delta_{\omega_\varepsilon} P_{0,2} f - f_{0,2} \|_{C^{0,\alpha}_{\delta-2,\tau-2,\lambda,\mu-2,\varepsilon}(M)} \le \frac{1}{100} \| f \|_{C^{0,\alpha}_{\delta-2,\tau-2,\lambda,\mu-2,\varepsilon}(M)};\\
&\| P_{0,1} f \|_{C^{2,\alpha}_{\delta,\tau,\lambda,\mu,\varepsilon}(M)} \le C(\delta, \tau, \alpha) \| f \|_{C^{0,\alpha}_{\delta-2,\tau-2,\lambda,\mu-2,\varepsilon}(M)};\\
&\| P_{0,2} f \|_{C^{2,\alpha}_{\delta,\tau,\lambda,\mu,\varepsilon}(M)} \le C(\delta, \tau, \alpha) \| f \|_{C^{0,\alpha}_{\delta-2,\tau-2,\lambda,\mu-2,\varepsilon}(M)}.
\end{align*}

\end{lemma}

We then treat $f_i$ with $i\ge 1$. 

For those $y_i'$ in the region $\{|t|>\frac{\epsilon_3}{2},|y|>1\}$, define cutoff functions 
\begin{align*}
\widetilde{\chi}_i'=\gamma_2\left(\frac{|y-y_i'|}{\Lambda_3\varepsilon^{\frac12}}\right);
\end{align*}For those $y_i'$ in the region $\{|y|<1\}$, define cutoff functions 
\begin{align*}
\widetilde{\chi}_i'=\gamma_2\left(\frac{|y-y_i'|}{\Lambda_3|y_i'|\varepsilon^{\frac12}}\right).
\end{align*}
Define
\begin{equation}
P_1f\coloneqq \sum\limits_{i=1}^{N'}\widetilde{\chi}_i'P_{y_i'}f_i,\quad  P_2f\coloneqq \sum\limits_{i=N'+1}^{\infty}\widetilde{\chi}_i'P_{y_i'}f_i,
\end{equation}where $P_{y_i'}f_i=G^*_{y_i'}\left(P_{y_i'}\left((G_{y_i'}^{-1})^*f_i\right)\right)$. For $i=1,...,N'$, the Green operators $P_{y_i'}$ are those constructed in Proposition \ref{Laplace for C times K3}, while for $i=N'+1,...$, the Green operators $P_{y_i'}$ are those constructed in Proposition \ref{Laplace for C times K3 weighed version}.

\begin{lemma}\label{approximate invert 1}

For any prescribed $\epsilon_3$, one can simultaneously select $\Lambda_3$ sufficiently large and $\epsilon_4$ sufficiently small so that, for sufficiently small $\varepsilon > 0$ relative to the previous choices, the following estimates hold:
\begin{align*}
&\left\|\Delta_{\omega_\varepsilon}P_1f-\sum\limits_{i=1}^{N'} f_i\right\|_{C^{0,\alpha}_{\delta-2,\tau-2,\lambda,\mu-2,\varepsilon}(M)}\le \frac{1}{100}\|f\|_{C^{0,\alpha}_{\delta-2,\tau-2,\lambda,\mu-2,\varepsilon}(M)};\\
&\left\|\Delta_{\omega_\varepsilon}P_2f-\sum\limits_{i=N'+1}^\infty f_i\right\|_{C^{0,\alpha}_{\delta-2,\tau-2,\lambda,\mu-2,\varepsilon}(M)}\le \frac{1}{100}\|f\|_{C^{0,\alpha}_{\delta-2,\tau-2,\lambda,\mu-2,\varepsilon}(M)};\\
&\left\|P_1f\right\|_{C^{2,\alpha}_{\delta,\tau,\lambda,\mu,\varepsilon}(M)}\le C(\epsilon_3,\epsilon_4,\Lambda_3,\delta,\tau,\lambda,\mu,\alpha)\|f\|_{C^{0,\alpha}_{\delta-2,\tau-2,\lambda,\mu-2,\varepsilon}(M)};\\
&\left\|P_2f\right\|_{C^{2,\alpha}_{\delta,\tau,\lambda,\mu,\varepsilon}(M)}\le C(\epsilon_3,\epsilon_4,\Lambda_3,\delta,\tau,\lambda,\mu,\alpha)\|f\|_{C^{0,\alpha}_{\delta-2,\tau-2,\lambda,\mu-2,\varepsilon}(M)}.
\end{align*}

\end{lemma}

\begin{proof}

We first estimate$\left\|\Delta_{\omega_\varepsilon }P_1f-\sum\limits_{i=1}^{N'}f_i\right\|_{C^{0,\alpha}_{\delta-2,\tau-2,\lambda,\mu-2,\varepsilon}(M)}$. 

For $i=1,...,N'$, by Remark \ref{Everything holds for SRF,y}, we know that
\begin{align*}
\|P_{y_i'}f_i\|_{C^{2,\alpha}(\mathbb{C}\times X_{y_i'})}\le C(\epsilon_1,\epsilon_3,\alpha,\delta,\tau,\lambda,\mu)\epsilon_4^{-1}\|f_i\|_{C^{0,\alpha}(\mathbb{C}\times X_{y_i'})}.
\end{align*}
Note that
\begin{align*}
\Delta_{\omega_\varepsilon }P_1f-\sum\limits_{i=1}^{N'}f_i=\sum\limits_{i=1}^{N'}\widetilde{\chi}_i'\left(\Delta_{\omega_\varepsilon}(P_{y_i'}f_i)-f_i\right)+\sum\limits_{i=1}^{N'} \left\{2\left\langle \nabla\widetilde{\chi}_i',\nabla P_{y_i'}f_i \right\rangle_{\omega_\varepsilon}+\left(\Delta_{\omega_\varepsilon}\widetilde{\chi}_i'\right)P_{y_i'}f_i\right\}.
\end{align*}

The first sum accounts for the error arising from the deviation away from the product metric. For these functions $f_i$, the norm $\varepsilon^{\frac12(\delta-\tau)}\|\cdot\|_{C^{0,\alpha}(\mathbb{C}\times X_{y_i'})}$ is equivalent to $\|\cdot\|_{C^{0,\alpha}_{\delta-2,\tau-2,\lambda,\mu-2,\varepsilon}(M)}$ up to a factor $C(\epsilon_1,\epsilon_3,\alpha)\,\epsilon_4^{\frac{13}{24}}$. More precisely, we have
\begin{align*}
C_1(\epsilon_1,\epsilon_3,\alpha)\epsilon_4^{\frac{13}{24}}\|\cdot\|_{C^{0,\alpha}_{\delta-2,\tau-2,\lambda,\mu-2,\varepsilon}(M)}\leq \varepsilon^{\frac12(\delta-\tau)}\|\cdot\|_{C^{0,\alpha}(\mathbb{C}\times X_{y_i'})}\leq C_2(\epsilon_1,\epsilon_3,\alpha) \|\cdot\|_{C^{0,\alpha}_{\delta-2,\tau-2,\lambda,\mu-2,\varepsilon}(M)}.
\end{align*}Hence, by Proposition \ref{metric deviation} and Remark \ref{Everything holds for SRF,y}, we know that 
\begin{align*}
\left\|\Delta_{\omega_\varepsilon}(P_{y_i'}f_i)-f_i\right\|_{C^{0,\alpha}_{\delta-2,\tau-2,\lambda,\mu-2,\varepsilon}(M)}&\leq C(\epsilon_1,\epsilon_3,\alpha)\epsilon_4^{-\frac{13}{24}}\varepsilon^{\frac12(\delta-\tau)}\left\|\Delta_{\omega_\varepsilon}(P_{y_i'}f_i)-f_i\right\|_{C^{0,\alpha}(\mathbb{C}\times X_{y_i'})}\\
&\leq C(\epsilon_1,\epsilon_3,\alpha)\epsilon_4^{-\frac{13}{24}}\varepsilon^{\frac12(\delta-\tau)}\Lambda_3\varepsilon^{\frac12}\|P_{y_i'}f_i\|_{C^{2,\alpha}(\mathbb{C}\times X_{y_i'})}\\
&\leq C(\epsilon_1,\epsilon_3,\alpha)\epsilon_4^{-\frac{37}{24}}\varepsilon^{\frac12(\delta-\tau)}\Lambda_3\varepsilon^{\frac12}\|f_i\|_{C^{0,\alpha}(\mathbb{C}\times X_{y_i'})}\\
&\leq C(\epsilon_1,\epsilon_3,\alpha)\epsilon_4^{-\frac{37}{24}}\Lambda_3\varepsilon^{\frac12}\|f_i\|_{C^{0,\alpha}_{\delta-2,\tau-2,\lambda,\mu-2,\varepsilon}(M)}
\end{align*}
At each point only $O(\Lambda_3^2)$ terms are nonzero, so the errors add up to
\begin{align*}
\left\|\sum\limits_{i=1}^{N'}\widetilde{\chi}_i'\left(\Delta_{\omega_\varepsilon}(P_{y_i'}f_i)-f_i\right)\right\|_{C^{0,\alpha}_{\delta-2,\tau-2,\lambda,\mu-2,\varepsilon}(M)}\le C(\alpha,\epsilon_3)\epsilon_4^{-\frac{37}{24}}\Lambda_3^3\varepsilon^{\frac12} \| f\|_{C^{0,\alpha}_{\delta-2,\tau-2,\lambda,\mu-2,\varepsilon}(M)}.
\end{align*}
Here we have a collapsing factor $\varepsilon^{\frac12}$, so this term is actually negligible.

The second sum corresponds to the cutoff error. By the exponential decay property established in Proposition~\ref{Laplace for C times K3} and Remark~\ref{Everything holds for SRF,y}, together with the fact that at each point at most $O(\Lambda_3^2)$ terms are nonvanishing, this error is bounded by
\begin{align*}
    &\left\|\sum\limits_{i=1}^{N'}\left\{2\left\langle \nabla\widetilde{\chi}_i',\nabla P_{y_i'}f_i \right\rangle_{\omega_\varepsilon}+\left(\Delta_{\omega_\varepsilon}\widetilde{\chi}_i'\right)P_{y_i'}f_i\right\}\right\|_{C^{0,\alpha}_{\delta-2,\tau-2,\lambda,\mu-2,\varepsilon}(M)}\\
    \leq &C(\epsilon_1,\epsilon_3,\alpha)\epsilon_4^{-\frac{13}{24}}\varepsilon^{\frac12(\delta-\tau)}\left\|\sum\limits_{i=1}^{N'}\left\{2\left\langle \nabla\widetilde{\chi}_i',\nabla P_{y_i'}f_i \right\rangle_{\omega_\varepsilon}+\left(\Delta_{\omega_\varepsilon}\widetilde{\chi}_i'\right)P_{y_i'}f_i\right\}\right\|_{C^{0,\alpha}(\mathbb{C}\times X_{y_i'})}\\
    \le & C(\epsilon_1,\epsilon_3,\alpha)\epsilon_4^{-\frac{37}{24}}\Lambda_3^2e^{-m\Lambda_3} \| f\|_{C^{0,\alpha}_{\delta-2,\tau-2,\lambda,\mu-2,\varepsilon}(M)}.
\end{align*}

We then estimate $\left\|\Delta_{\omega_\varepsilon }P_2f-\sum\limits_{i=N'+1}^{\infty}f_i\right\|_{C^{0,\alpha}_{\delta-2,\tau-2,\lambda,\mu-2,\varepsilon}(M)}$. 

For $i=N'+1,N'+2,...$, by Proposition \ref{Laplace for C times K3 weighed version}, we know that
\begin{align*}
\|P_{y_i'}f_i\|_{C^{2,\alpha}_{\lambda,\mu}(\mathbb{C}\times X_{y_i})}\le C(\alpha)\|f_i\|_{C^{0,\alpha}_{\lambda,\mu-2}(\mathbb{C}\times X_{y_i})}.
\end{align*}
Note that
\begin{align*}
\Delta_{\omega_\varepsilon }P_2f-\sum\limits_{i=N'+1}^{\infty}f_i&=\sum\limits_{i=N'+1}^{\infty}\widetilde{\chi}_i'\left(\Delta_{\omega_\varepsilon}(P_{y_i'}f_i)-f_i\right)+\sum\limits_{i=N'+1}^{\infty}\left\{2\left\langle \nabla\widetilde{\chi}_i',\nabla P_{y_i'}f_i \right\rangle_{\omega_\varepsilon}+\left(\Delta_{\omega_\varepsilon}\widetilde{\chi}_i'\right)P_{y_i'}f_i\right\}.
\end{align*}

The first sum corresponds to the error caused by deviation from product metric. For these functions $f_i$, the norm $\varepsilon^{\frac12(\delta-\tau)}\|\cdot\|_{C^{0,\alpha}_{\lambda,\mu-2}(\mathbb{C}\times X_{y_i'})}$ is equivalent to $\|\cdot\|_{C^{0,\alpha}_{\delta-2,\tau-2,\lambda,\mu-2,\varepsilon}(M)}$. By Proposition \ref{metric deviation} and Proposition \ref{Laplace for C times K3 weighed version}, we know that 
\begin{align*}
\left\|\Delta_{\omega_\varepsilon}(P_{y_i'}f_i)-f_i\right\|_{C^{0,\alpha}_{\delta-2,\tau-2,\lambda,\mu-2,\varepsilon}(M)}&\leq C(\alpha)\varepsilon^{\frac12(\delta-\tau)}\left\|\Delta_{\omega_\varepsilon}(P_{y_i'}f_i)-f_i\right\|_{C^{0,\alpha}_{\lambda,\mu-2}(\mathbb{C}\times X_{y_i'})}\\
&\le C(\alpha)\varepsilon^{\frac12(\delta-\tau)}\epsilon_4^{\frac{1}{24}} \|P_{y_i'}f_i\|_{C^{2,\alpha}_{\lambda,\mu}(\mathbb{C}\times X_{y_i'})}\\
&\le C(\alpha)\varepsilon^{\frac12(\delta-\tau)}\epsilon_4^{\frac{1}{24}} \| f_i\|_{C^{0,\alpha}_{\lambda,\mu-2}(\mathbb{C}\times X_{y_i'})}\\
&\le C(\alpha)\epsilon_4^{\frac{1}{24}} \| f_i\|_{C^{0,\alpha}_{\delta-2,\tau-2,\lambda,\mu-2,\varepsilon}(M)}.
\end{align*}
At any given point, at most $O(\Lambda_3^2)$ terms are non-vanishing. Summing these errors gives
\begin{align*}
\left\|\sum\limits_{i=N'+1}^{\infty}\widetilde{\chi}_i'\left(\Delta_{\omega_\varepsilon}(P_{y_i'}f_i)-f_i\right)\right\|_{C^{0,\alpha}_{\delta-2,\tau-2,\lambda,\mu-2,\varepsilon}(M)}\le C(\alpha)\epsilon_4^{\frac{1}{24}}\Lambda_3^2\| f\|_{C^{0,\alpha}_{\delta-2,\tau-2,\lambda,\mu-2,\varepsilon}(M)}
\end{align*}

The second sum corresponds to the cutoff error. By the exponential decay property in Proposition \ref{Laplace for C times K3 weighed version} and the fact that at each point at most $O(\Lambda_3^2)$ terms are nonzero, this error is controlled by
\begin{align*}
    &\left\|\sum\limits_{i=N'+1}^{\infty}\left\{2\left\langle \nabla\widetilde{\chi}_i',\nabla P_{y_i'}f_i \right\rangle_{\omega_\varepsilon}+\left(\Delta_{\omega_\varepsilon}\widetilde{\chi}_i'\right)P_{y_i'}f_i\right\}\right\|_{C^{0,\alpha}_{\delta-2,\tau-2,\lambda,\mu-2,\varepsilon}(M)}\\
    \leq &C(\alpha)\varepsilon^{\frac12(\delta-\tau)}\left\|\sum\limits_{i=N'+1}^{\infty}\left\{2\left\langle \nabla\widetilde{\chi}_i',\nabla P_{y_i'}f_i \right\rangle_{\omega_\varepsilon}+\left(\Delta_{\omega_\varepsilon}\widetilde{\chi}_i'\right)P_{y_i'}f_i\right\}\right\|_{C^{0,\alpha}_{\lambda,\mu-2}(\mathbb{C}\times X_{y_i'})}\\
    \le & C(\alpha)\Lambda_3^2e^{-m\Lambda_3} \| f\|_{C^{0,\alpha}_{\delta-2,\tau-2,\lambda,\mu-2,\varepsilon}(M)}.
\end{align*}

We need to simultaneously choose $\epsilon_4$ and $\Lambda_3$ such that, for fixed $\epsilon_3$, the following inequalities hold
\begin{align*}
\left\{
	\begin{aligned}
&C(\epsilon_1,\epsilon_3,\alpha)\epsilon_4^{-\frac{37}{24}}\Lambda_3^2e^{-m\Lambda_3}<\frac{1}{200};\\
   & C(\alpha)\Lambda_3^2e^{-m\Lambda_3}<\frac{1}{200};\\
   &C(\alpha)\Lambda_3^2\epsilon_4^{\frac{1}{24}}<\frac{1}{200}.
    \end{aligned}
	\right.
\end{align*}
This is always possible since for sufficiently large $\Lambda_3$, it holds that
\begin{align*}
20C(\epsilon_1,\epsilon_3,\alpha)\Lambda_3e^{-\frac{m\Lambda_3}{2}}<\left(\frac{1}{200C(\alpha)\Lambda_3^2}\right)^{24}.
\end{align*}Then choose an $\epsilon_4$ in this interval. Note that the admissible $\epsilon_4$ tends to zero as $\Lambda_3\to+\infty$.

The boundedness of $P_1$ and $P_2$ follows directly from the construction.

\end{proof}

Finally, we discuss the function $f_B(y)$ on the base.

\begin{lemma}\label{decay of base function}

The function $f_B$ satisfies the bound
\begin{align*}
|f_B(x)|\le C\|f\|_{C^{0,\alpha}_{\delta-2,\tau-2,\lambda,\mu-2,\varepsilon}(M)}\left\{
	\begin{aligned}
&\left(1+\varepsilon^{-\frac12}|t|\right)^{\delta-\tau}\quad && \text{if}\quad x=t, |t|<\epsilon_1;\\
& \varepsilon^{\frac{\tau-\delta}{2}} \quad && \text{if}\quad x\in\left\{|t|> \frac12\epsilon_1,|y|> \frac12\epsilon_1\right\};\\
&\varepsilon^{\frac{\tau-\delta}{2}}|y|^{\frac{13}{24}} \quad && \text{if}\quad x=y,|y|<\epsilon_1,
\end{aligned}
	\right.
\end{align*}where $C=C(\epsilon_1,\delta,\tau,\lambda,\mu)$ is a uniform constant independent of $\varepsilon$.

For $d_{\omega_B}(x,x')<\varepsilon^{\frac12}$, the following H{\"o}lder estimates hold
\begin{align*}
\frac{|f_B(x)-f_B(x')|}{\left(d_{\frac{1}{\varepsilon}\omega_B}(x,x')\right)^\alpha}\leq C \|f\|_{C^{0,\alpha}_{\delta-2,\tau-2,\lambda,\mu-2,\varepsilon}(M)}  \left\{
	\begin{aligned}
&\left(1+\varepsilon^{-\frac12}|t|\right)^{\delta-\tau}\; && \text{if}\quad x=t,x'=t', |t|<\epsilon_1;\\
& \varepsilon^{\frac{\tau-\delta}{2}} \; && \text{if}\quad x\in\left\{|t|> \frac12\epsilon_1,|y|> \frac12\epsilon_1\right\};\\
&\varepsilon^{\frac{\tau-\delta}{2}}|y|^{\frac{13}{24}} \; && \text{if}\quad x=y,x'=y',|y|<\epsilon_1,
\end{aligned}
	\right.
\end{align*}
where $C=C(\epsilon_1,\delta,\tau,\lambda,\mu)$ is a uniform constant independent of $\varepsilon$.

Set
\begin{equation}
\overline{f_B}=\int_{B}f_B\omega_B.
\end{equation}Then
\begin{align*}
|\overline{f_B}|\le C\varepsilon^{\frac{\tau-\delta}{2}}\max\left\{\epsilon_3^{\delta-\tau+2},\epsilon_4^{\frac{23}{288}}\right\}\|f\|_{C^{0,\alpha}_{\delta-2,\tau-2,\lambda,\mu-2,\varepsilon}(M)},
\end{align*}where $C=C(\epsilon_1,\delta,\tau,\lambda,\mu)$ is a uniform constant independent of $\varepsilon$.

\end{lemma}

\begin{proof}

For the pointwise estimate on $f_B$, it suffices to treat the case $x = y$ with $|y| < \epsilon_1$, as the remaining two cases have already been dealt with in \cite{li2019gluing}.

\textbf{Claim:} For $\omega_y=\omega_X|_{X_y}+\sqrt{-1}\partial\bar{\partial}\widetilde{\phi}_y$, where $|y|<\epsilon_1$, it holds that
\begin{align*}
\left|\int_{X_y}f\omega_y^2\right|\le C\varepsilon^{\frac{\tau-\delta}{2}}|y|^{\frac{13}{24}}\|f\|_{C^{0,\alpha}_{\delta-2,\tau-2,\lambda,\mu-2,\varepsilon}(M)},
\end{align*}where $C$ is a uniform constant independent of $y$.

By the proof of Lemma \ref{fiber deviation 1}, we know that
\begin{align*}
\left|\nabla_{\operatorname{EH}_y}^2(\phi_1 - \phi_2)\right|_{\operatorname{EH}_y}
&\le C\left( \frac{|y|}{H}+H^{\frac12} + \frac{|y|^2}{H^2}\right)
\end{align*}for $H\ge |y|^{\frac23}$, where $\phi_1=\sqrt{H+|y|}$, $\phi_2=\Theta_\infty+G_{\infty,y}^*(\psi_0^{(\infty)}-c_\infty)$. 

Define a new metric 
\begin{align*}
\omega'_y=\omega_X|_{X_y}+\sqrt{-1}\partial\bar{\partial}\left(\gamma_1\left(\frac{r}{c}\right)G_{\infty,y}^*\left(\psi_0^{(\infty)}-c_\infty\right)+\gamma_2\left(\frac{r}{c}\right)\left(\sqrt{H+|y|}-\Theta_\infty\right)\right),
\end{align*}where $c$ is a small number such that $Cc^2<\frac{1}{100}$.

Then in the region $\{r<c\}$, the metric $\omega_y'$ and $\omega_y$ are uniformly equivalent. In the region $\{r\ge c\}$, the metric $\omega_y$ is uniformly equivalent to $\omega_X|_{X_y}$. So we have
\begin{align*}
\left|\int_{X_y}f\omega_y^2\right|&\le C\left(\int_{X_y\cap \{r<c\}}|f|\omega_y'^2+\int_{X_y\cap \{r\ge c\}}|f|\omega_X|_{X_y}^2\right)\\
&\le C\varepsilon^{\frac{\tau-\delta}{2}}\|f\|_{C^{0,\alpha}_{\delta-2,\tau-2,\lambda,\mu-2,\varepsilon}(M)}\left(\int_{X_y\cap \{r<c\}}\frac{|y|^{\frac{13}{24}}}{r^{2+\frac{1}{12}}}\omega_y'^2+\int_{X_y\cap \{r\ge c\}}|y|^{\frac{13}{24}}\omega_X|_{X_y}^2\right)\\
&\le C\varepsilon^{\frac{\tau-\delta}{2}}\|f\|_{C^{0,\alpha}_{\delta-2,\tau-2,\lambda,\mu-2,\varepsilon}(M)}\left(\int_{X_y\cap \{r<c\}}\frac{|y|^{\frac{13}{24}}}{r^{2+\frac{1}{12}}}\omega_y'^2+|y|^{\frac{13}{24}}\right).
\end{align*}
Note that
\begin{align*}
\int_{X_y\cap \{r<c\}}\frac{|y|^{\frac{13}{24}}}{r^{2+\frac{1}{12}}}\omega_y'^2&=\int_{X_y\cap \left\{|y|^{\frac14}\le r<c\right\}}\frac{|y|^{\frac{13}{24}}}{r^{2+\frac{1}{12}}}\omega_{\operatorname{EH}_y}^2\\
&=\int_{X_1\cap \left\{1\le r\le c|y|^{-\frac14}\right\}}\frac{|y|^{\frac{13}{24}}}{|y|^{\frac12+\frac{1}{48}}r^{2+\frac{1}{12}}}|y|\omega_{\operatorname{EH}_1}^2\\
&\le C|y|^{\frac{13}{24}+1-\frac{1}{2}-\frac{1}{48}}\int_1^{ c|y|^{-\frac14}}r^{-2-\frac{1}{12}}r^3dr\\
&\le C|y|^{\frac{13}{24}}.
\end{align*}
We carefully explain the above calculation. The K{\"a}hler metric on $X_y\cap \{r<c\}$ is $|y|^{\frac12}\Psi_y^*\omega_{\operatorname{EH}_1}$, where $\Psi_y:V_y\to V_1,\; x\mapsto \frac{x}{|y|^{\frac12}}$ and $\omega_{\operatorname{EH}_1}$ is the standard Stenzel metric on $V_1$. Via the diffeomorphism $\Psi_y$, the region $\{|y|\le H\le 1\}$ on $X_y$ corresponds to the region $\{1\le H\le |y|^{-1}\}$; the function $r$ turns into $|y|^{\frac14}r$. We also have the following relation between the volume forms $dV_{\omega_{\operatorname{EH}_y}}=|y|dV_{\omega_{\operatorname{EH}_1}}$. This explains the second equality. The Stenzel metric on $V_1$ is given by $\left(\sqrt{-1}\partial\bar{\partial}\sqrt{H+1}\right)\Big{|}_{V_1}$ which is uniformly equivalent to the cone metric $\left(\sqrt{-1}\partial\bar{\partial}\sqrt{H}\right)\Big{|}_{V_0}$ outside a fixed compact region. It therefore suffices to perform the integration on the K\"ahler cone $V_0$ instead. On $V_0$, the distance $r$ to the apex satisfies $r = H^{\frac14}$, and the volume form factorizes as $dV_{\omega_{\operatorname{EH}_0}} = r^3\,dr \wedge dL$, where $dL$ is the volume form of the link of the cone. The link has finite volume. This justifies the third inequality.

By Proposition \ref{fiber deviation 2} and Lemma \ref{fiber deviation 4}, we know that for $|y|$ uniformly small, and for $y'$ satisfying $|y-y'|<\varepsilon^{\frac12}|y'|$, the metrics $\omega_X|_{X_y}+\sqrt{-1}\partial\bar{\partial}\psi_y$, $G^*_{y',y}\left(\omega_X|_{X_{y'}}+\sqrt{-1}\partial\bar{\partial}\psi_{y'}\right)$ and $\omega_X|_{X_y}+\sqrt{-1}\partial\bar{\partial}G^*_{y',y}\psi_{y'}$ are all uniformly equivalent to $\omega_y$ on $X_y$. So they all share the same integration property in the above Claim. In particular, we can derive the pointwise bound on $f_B$.

Next we prove the weighted H{\"o}lder estimate.

We give a detailed argument in the region $\{|y|<\epsilon_1\}$, since this is the only case involving the degenerating fibers.

For each $i\geq 1$, let $y_i'$ be the marked point associated
with the cutoff function $\chi_i'$, and let $G_{y,y_i'}:X_{y_i'}\longrightarrow X_y$ be the fiber-preserving diffeomorphism used in the construction. Thus, for $z\in X_{y_i'}$,
\[
(G_{y,y_i'}^*f)(z)=f\left(y,G_{y,y_i'}(z)\right).
\]
Define
\[
dV_i\coloneqq
\begin{cases}
\omega_{\mathrm{SRF}}|_{X_{y_i'}}^2,
 & 1\leq i\leq N',\\[3pt]
\omega_{y_i'}^2,
 & i\geq N'+1.
\end{cases}
\]
Then
\[
f_i'(y)=\chi_i'(y)\int_{X_{y_i'}}f\left(y,G_{y,y_i'}(z)\right)\,dV_i(z).
\]

We first specify the distance function used on the reference fiber
$X_{y_i'}$. Recall that inside the Lefschetz coordinate neighborhood
$U_\infty$, we have $H(z)=|z_1|^2+|z_2|^2+|z_3|^2$ and $r(z)= H(z)^{\frac{1}{4}}$. Outside $U_\infty$, we use a fixed smooth extension of $r$ satisfying $r\sim 1$.

For $z\in X_{y_i'}$, put $p_z\coloneqq\left(y,G_{y,y_i'}(z)\right)\in X_y$, $q_z\coloneqq \left(y',G_{y',y_i'}(z)\right)\in X_{y'}$. 

The explicit construction of $G_{y,y_i'}$ gives, inside
$U_\infty$,
\begin{align*}
&r(p_z)^8=r(z)^8+|y|^2-|y_i'|^2,\\
&r(q_z)^8=r(z)^8+|y'|^2-|y_i'|^2.
\end{align*}
Since $z\in X_{y_i'}$, we have
\[
r(z)^8=H(z)^2\geq |y_i'|^2.
\]
Moreover, whenever $\chi_i'(y)\neq 0$, the construction of the
partition of unity gives
\[
|y-y_i'|
\leq
C\varepsilon^{\frac{1}{2}}|y_i'|.
\]
If, in addition, $D\coloneqq d_{\frac{1}{\varepsilon}\omega_B}(y,y')<1$, then $|y-y'| \leq C\varepsilon^{\frac{1}{2}}|y|$. Consequently, whenever either $\chi_i'(y)$ or $\chi_i'(y')$
is nonzero,
\[
|y|\sim|y'|\sim |y_i'|
\]
and
\[
r(p_z)\sim r(q_z)\sim r(z),
\]
uniformly in $i$ and $\varepsilon$. That is, the above quantities mutually control each other up to a uniform constant independent of $i$ and $\varepsilon$. Outside $U_\infty$, the latter comparison follows immediately from $r\sim 1$.

We shall also use the distance comparison
\[
d_{\omega_\varepsilon}(p_z,q_z)\leq CD.
\]
Indeed, near the end, $D\sim\frac{|y-y'|}{\varepsilon^{\frac{1}{2}}|y|}$. Thus the horizontal part of the displacement between $p_z$ and $q_z$ is bounded by $CD$. The explicit formula for
$G_{y,y_i'}$ shows that the vertical displacement is bounded by
\[
C\frac{|y-y'|}{r(z)^3} \leq C\varepsilon^{\frac{1}{2}}|y|r(z)^{-3}D.
\]
Since $r(z)^4\geq |y_i'|\sim |y|$, we have
\[
\varepsilon^{\frac{1}{2}}|y|r(z)^{-3}\leq C\varepsilon^{\frac{1}{2}}|y|^{\frac{1}{4}} \leq C.
\]
This proves the claimed distance comparison.

We now estimate $f_i'(y)-f_i'(y')$. After interchanging $y$ and $y'$, if necessary, we may suppose that $\chi_i'(y)\neq 0$. We write
\begin{equation*}
\begin{aligned}
f_i'(y)-f_i'(y')=&\chi_i'(y)\int_{X_{y_i'}}\left(f(p_z)-f(q_z)\right)\,dV_i(z)\\
&+\left(\chi_i'(y)-\chi_i'(y')\right)\int_{X_{y_i'}}f(q_z)\,dV_i(z).
\end{aligned}
\end{equation*}

We first estimate $\chi_i'(y)\int_{X_{y_i'}}\left(f(p_z)-f(q_z)\right)\,dV_i(z)$. Let $c_H>0$ be the fixed constant
implicit in the definition of the weighted H\"older seminorm, so that
this seminorm controls pairs $p,q$ satisfying $d_{\omega_\varepsilon}(p,q)\leq c_Hr(p)$. Choose a sufficiently large constant $K$, independent of $i$ and
$\varepsilon$, and decompose
\[
X_{y_i'}=E_{\mathrm{out}}\cup E_{\mathrm{in}},
\]
where
\[
E_{\mathrm{out}}\coloneqq\left\{z\in X_{y_i'}:r(z)\geq KD\right\}, \quad E_{\mathrm{in}} \coloneqq\left\{z\in X_{y_i'}:r(z)<KD\right\}.
\]

On $E_{\mathrm{out}}$, the preceding distance and weight comparisons give
\[
d_{\omega_\varepsilon}(p_z,q_z)\leq CD\leq\frac{C}{K}r(z)\leq c_Hr(p_z),
\]
provided that $K$ is sufficiently large. Therefore, the weighted H\"older seminorm of $f$ applies.

By the definition of the weighed H{\"o}lder norm and the relations $|y|\sim |y_i'|$, $r(p_z)\sim r(z)$, $d_{\omega_\varepsilon}(p_z,q_z)\leq CD$, we have
\begin{align*}
|f(p_z)-f(q_z)|&\leq \|f\|_{C^{0,\alpha}_{\lambda,\mu-2}(U_3)}\frac{|y|^\lambda}{r(p_z)^{2-\mu+\alpha}}d_{\omega_\varepsilon}(p_z,q_z)^\alpha\\
&\leq \|f\|_{C^{0,\alpha}_{\delta-2,\tau-2,\lambda,\mu-2,\varepsilon}}\varepsilon^{\frac{\tau-\delta}{2}}
|y_i'|^\lambda r(z)^{\mu-2-\alpha}D^\alpha.
\end{align*}

It follows that
\begin{equation*}
\begin{aligned}
\int_{E_{\mathrm{out}}}|f(p_z)-f(q_z)|\,dV_i(z)&\leq C\|f\|_{C^{0,\alpha}_{\delta-2,\tau-2,\lambda,\mu-2,\varepsilon}}\varepsilon^{\frac{\tau-\delta}{2}}|y_i'|^\lambda D^\alpha \int_{E_{\mathrm{out}}}r(z)^{\mu-2-\alpha}\,dV_i(z)\\
&\leq C\|f\|_{C^{0,\alpha}_{\delta-2,\tau-2,\lambda,\mu-2,\varepsilon}}\varepsilon^{\frac{\tau-\delta}{2}}|y_i'|^\lambda D^\alpha.
\end{aligned}
\end{equation*}
The last integral is uniformly bounded because $\mu-2-\alpha=-\frac{25}{12}-\alpha>-4$.

On $E_{\mathrm{in}}$, the points $p_z,q_z$ need not satisfy the
locality condition in the weighted H\"older seminorm. We therefore use
the weighted pointwise estimate. By definition of the weighted norm, we have
\begin{align*}
&|f(p_z)|\leq C \|f\|_{C^{0,\alpha}_{\delta-2,\tau-2,\lambda,\mu-2,\varepsilon}} \varepsilon^{\frac{\tau-\delta}{2}} |y_i'|^\lambda r(z)^{\mu-2},\\
&|f(q_z)|\leq C \|f\|_{C^{0,\alpha}_{\delta-2,\tau-2,\lambda,\mu-2,\varepsilon}} \varepsilon^{\frac{\tau-\delta}{2}} |y_i'|^\lambda r(z)^{\mu-2}.
\end{align*}
 Hence
\begin{equation*}
\begin{aligned}
\int_{E_{\mathrm{in}}}|f(p_z)-f(q_z)|\,dV_i(z)&\leq C\|f\|_{C^{0,\alpha}_{\delta-2,\tau-2,\lambda,\mu-2,\varepsilon}}\varepsilon^{\frac{\tau-\delta}{2}}|y_i'|^\lambda\int_{\{r<KD\}}r(z)^{\mu-2}\,dV_i(z)\\
&\leq C\|f\|_{C^{0,\alpha}_{\delta-2,\tau-2,\lambda,\mu-2,\varepsilon}}\varepsilon^{\frac{\tau-\delta}{2}}|y_i'|^\lambda D^{\mu+2}.
\end{aligned}
\end{equation*}
Since $D<1$ and $\mu+2=\frac{23}{12}>\alpha$, we have $D^{\mu+2}\leq D^\alpha$. Therefore,
\[
\int_{E_{\mathrm{in}}}|f(p_z)-f(q_z)|\,dV_i(z) \leq C \|f\|_{C^{0,\alpha}_{\delta-2,\tau-2,\lambda,\mu-2,\varepsilon}} \varepsilon^{\frac{\tau-\delta}{2}}|y_i'|^\lambda D^\alpha.
\]
If $KD$ is larger than the fixed coordinate radius $c$, the same conclusion follows from the full weighted integral and the fact that $D\geq c/K$.

Combining the inner and outer regions gives
\[
\left|\chi_i'(y)\int_{X_{y_i'}}\left(f(p_z)-f(q_z)\right)\,dV_i(z)\right| \leq C\|f\|_{C^{0,\alpha}_{\delta-2,\tau-2,\lambda,\mu-2,\varepsilon}} \varepsilon^{\frac{\tau-\delta}{2}}
|y_i'|^\lambda D^\alpha.
\]

It remains to estimate the cutoff term $\left(\chi_i'(y)-\chi_i'(y')\right)\int_{X_{y_i'}}f(q_z)\,dV_i(z)$. The cutoff functions have uniformly bounded first derivatives with respect to the rescaled base metric $\varepsilon^{-1}\omega_B$. Therefore,
\[
|\chi_i'(y)-\chi_i'(y')| \leq CD.
\]
The weighted pointwise estimate and the uniform integrability of
$r^{\mu-2}$ give
\[
\int_{X_{y_i'}}|f(q_z)|\,dV_i(z)\leq C\|f\|_{C^{0,\alpha}_{\delta-2,\tau-2,\lambda,\mu-2,\varepsilon}} \varepsilon^{\frac{\tau-\delta}{2}} |y_i'|^\lambda.
\]
Consequently,
\begin{align*}
\left|\left(\chi_i'(y)-\chi_i'(y')\right)\int_{X_{y_i'}}f(q_z)\,dV_i(z)\right|&\leq C\|f\|_{C^{0,\alpha}_{\delta-2,\tau-2,\lambda,\mu-2,\varepsilon}}\varepsilon^{\frac{\tau-\delta}{2}}|y_i'|^\lambda D\\
&\leq C\|f\|_{C^{0,\alpha}_{\delta-2,\tau-2,\lambda,\mu-2,\varepsilon}}\varepsilon^{\frac{\tau-\delta}{2}}|y_i'|^\lambda D^\alpha,
\end{align*}where the last inequality uses $D<1$.

We have proved that
\[
|f_i'(y)-f_i'(y')| \leq C\|f\|_{C^{0,\alpha}_{\delta-2,\tau-2,\lambda,\mu-2,\varepsilon}} \varepsilon^{\frac{\tau-\delta}{2}} |y_i'|^\lambda d_{\frac{1}{\varepsilon}\omega_B}(y,y')^\alpha.
\]
Whenever either $\chi_i'(y)$ or $\chi_i'(y')$ is nonzero, we have $|y_i'|\sim |y|$. Moreover, the family $\{\chi_i'\}$ is uniformly locally finite. Summing over $i$, we therefore obtain
\[
|f_B(y)-f_B(y')|\leq C\|f\|_{C^{0,\alpha}_{\delta-2,\tau-2,\lambda,\mu-2,\varepsilon}} 
\varepsilon^{\frac{\tau-\delta}{2}} |y|^\lambda d_{\frac{1}{\varepsilon}\omega_B}(y,y')^\alpha,
\quad |y|<\epsilon_1.
\]

We finally treat the other base regions. Near the finite singular fiber, introduce the rescaled coordinate $\zeta=\sqrt{2A_0}\varepsilon^{-\frac{1}{2}}t$. Then $d_{\frac{1}{\varepsilon}\omega_B}(t,t')\sim|\zeta-\zeta'|$ whenever $d_{\frac{1}{\varepsilon}\omega_B}(t,t')<1$. Lemma~\ref{holder on pushforward}, applied to the
fiber integral on $\mathbb C_\zeta\times X_0$, gives
\[
|f_B(t)-f_B(t')|\leq C\|f\|_{C^{0,\alpha}_{\delta-2,\tau-2,\lambda,\mu-2,\varepsilon}} (1+|\zeta|)^{\delta-\tau}|\zeta-\zeta'|^\alpha.
\]
The denominator occurring in the definition of $f_0'$ is uniformly bounded above and below and has uniformly bounded derivatives in the rescaled coordinate. Multiplication by $\chi_0'$ therefore preserves this estimate. Since $1+|\zeta|\sim 1+\varepsilon^{-\frac{1}{2}}|t|$, we obtain
\[
|f_B(t)-f_B(t')|   \leq C\|f\|_{C^{0,\alpha}_{\delta-2,\tau-2,\lambda,\mu-2,\varepsilon}} \left(1+\varepsilon^{-\frac{1}{2}}|t|\right)^{\delta-\tau}  d_{\frac{1}{\varepsilon}\omega_B} (t,t')^\alpha.
\]

Finally, on $\left\{|t|>\frac12\epsilon_1,\ |y|>\frac12\epsilon_1\right\}$, the fibers, the trivializations, and their volume forms vary in a uniformly controlled smooth family. Fiberwise integration is therefore a uniformly bounded map from $C^{0,\alpha}$ on the total space to
$C^{0,\alpha}$ on the base. Since the global norm contains the term $\varepsilon^{\frac{\delta-\tau}{2}}\|f\|_{C^{0,\alpha}(U_2)}$, we have
\[
\|f\|_{C^{0,\alpha}(U_2)} \leq \varepsilon^{\frac{\tau-\delta}{2}}\|f\|_{C^{0,\alpha}_{\delta-2,\tau-2,\lambda,\mu-2,\varepsilon}}.
\]
It follows that
\[
|f_B(x)-f_B(x')| \leq C\|f\|_{C^{0,\alpha}_{\delta-2,\tau-2,\lambda,\mu-2,\varepsilon}}\varepsilon^{\frac{\tau-\delta}{2}}d_{\frac{1}{\varepsilon}\omega_B}(x,x')^\alpha
\]
in the compact base region.

The weights in the three estimates above are uniformly equivalent on
the corresponding overlap regions. This proves the asserted weighted
H\"older estimate for $f_B$.

Define
\begin{align*}
\widetilde{f}_B=\pi_*\left(f(\omega_X+\sqrt{-1}\partial\bar{\partial}\phi)^2\right),
\end{align*}that is
\begin{align*}
\widetilde{f}_B(y)= \int_{X_y}f(\omega_X+\sqrt{-1}\partial\bar{\partial}\phi)\big{|}_{X_y}^2.
\end{align*}

Now by assumption\
\begin{align*}
0&=\varepsilon\int_Mf\omega_\varepsilon^3\\
&=\int_M3\pi^*\omega_B\wedge f(\omega_X+\sqrt{-1}\partial\bar{\partial}\phi)^2+3\varepsilon\pi^*\beta\wedge f(\omega_X+\sqrt{-1}\partial\bar{\partial}\phi)^2+\varepsilon f(\omega_X+\sqrt{-1}\partial\bar{\partial}\phi)^3.
\end{align*}
Note that the latter two terms are $O\left(\varepsilon^{1+\frac12(\tau-\delta)}\|f\|_{C^{0,\alpha}_{\delta-2,\tau-2,\lambda,\mu-2,\varepsilon}(M)}\right)$. Hence
\begin{align*}
\left| \int_M\pi^*\omega_B\wedge f(\omega_X+\sqrt{-1}\partial\bar{\partial}\phi)^2\right|&=\left|\int_{\{|t|<\epsilon_3\}}\widetilde{f}_B\omega_B+\int_{\{|t|>\epsilon_3\}}\widetilde{f}_B\omega_B\right|\\
&\leq C\varepsilon^{1+\frac12(\tau-\delta)}\|f\|_{C^{0,\alpha}_{\delta-2,\tau-2,\lambda,\mu-2,\varepsilon}(M)}.
\end{align*}

By the proof of Lemma 3.9 in \cite{li2019gluing}, we know that
\begin{align*}
\left|\int_{\{|t|<\epsilon_3\}}(f_B-\widetilde{f}_B)\omega_B\right|\le C\varepsilon^{\frac12(\tau-\delta)}\epsilon_3^{\delta-\tau+2}\|f\|_{C^{0,\alpha}_{\delta-2,\tau-2,\lambda,\mu-2,\varepsilon}(M)}
\end{align*}

By Proposition \ref{fiber deviation 2} and Lemmas \ref{fiber deviation 3}, \ref{fiber deviation 4}, we have
\begin{align*}
&\left|\int_{\{|t|>\epsilon_3\}}(f_B-\widetilde{f}_B)\omega_B\right|\\
\leq&\int_{\{|t|>\epsilon_3\}}\left|\int_{X_y}\sum\limits_{i=1}^{N'}\chi_i'f\left(G_{y_i',y}^*\omega_{\operatorname{SRF}}|_{X_{y_i'}}\right)^2+\sum\limits_{i=N'+1}^{\infty}\chi_i'f\left(G_{y_i',y}^*\omega_{y_i'}\right)^2-f(\omega_X|_{X_y}+\sqrt{-1}\partial\bar{\partial}\phi)^2\right|\omega_B\\
\le &C\varepsilon^{\frac12(\tau-\delta)}\epsilon_4^{\frac{23}{288}}\|f\|_{C^{0,\alpha}_{\delta-2,\tau-2,\lambda,\mu-2,\varepsilon}(M)},
\end{align*}where the factor $\epsilon_4^{\frac{23}{288}}$ comes from the deviation of metrics on the region $\left\{\varepsilon^{\frac{12}{2+\mu}}<|y|<\epsilon_4\right\}$. More precisely, taking $\beta=-\frac{1}{12}$ in Proposition \ref{fiber deviation 2}, we gain a factor $|y|^{\frac{7}{12}+\frac{5}{288}-\frac12-\frac{1}{48}}=|y|^{\frac{23}{288}}$.

\end{proof}

We now turn to the inversion of the Laplacian on $B \cong \mathbb{C}$, namely, to solving the Poisson equation
\begin{equation}\label{base eq 1}
\Delta_{\frac{1}{\varepsilon}\omega_B} u = f_B.
\end{equation}
Since $\omega_B \sim \frac{\sqrt{-1}}{2|y|^2}\,dy \wedge d\bar{y}$ for $|y| < \epsilon_1$, the manifold $\left(B, \frac{1}{\varepsilon}\omega_B\right)$ is asymptotically cylindrical in the sense of \cite{haskins2015asymptotically}. On a complete manifold of this type, a necessary condition for the Poisson equation to be solvable is that the forcing term has zero average. Consequently, equation~\eqref{base eq 1} must be adjusted, as it may otherwise admit no solution at all. One approach is to decompose $f_B$ into a function with zero average and a small error term.

Let $\chi_0''$ be a cutoff function on $Y$ supported in $\{|t| > \epsilon_1,\ |y| > \epsilon_1\}$, identically equal to $1$ on $\{|t| > 2\epsilon_1,\ |y| > 2\epsilon_1\}$, and having uniformly bounded higher derivatives with respect to the metric $\frac{1}{\varepsilon}\omega_B$. Let
\begin{equation}
f_B' = f_B - \overline{f_B} \frac{\chi_0''}{\int_B \chi_0'' \omega_B}, \quad f_B'' = \overline{f_B} \frac{\chi_0''}{\int_B \chi_0'' \omega_B}.
\end{equation}

By construction, we have
\begin{itemize}

\item $\int_{B}f'_B\omega_B=0$;

\item $|f'_B(x)|  \le C\|f\|_{C^{0,\alpha}_{\delta-2,\tau-2,\lambda,\mu-2,\varepsilon}(M)}\left\{
	\begin{aligned}
&\left(1+\varepsilon^{-\frac12}|t|\right)^{\delta-\tau}\quad && \text{if}\quad x=t, |t|<\epsilon_1;\\
& \varepsilon^{\frac{\tau-\delta}{2}} \quad && \text{if}\quad x\in\left\{|t|> \frac12\epsilon_1,|y|> \frac12\epsilon_1\right\};\\
&\varepsilon^{\frac{\tau-\delta}{2}}|y|^{\frac{13}{24}} \quad && \text{if}\quad x=y,|y|<\epsilon_1;
\end{aligned}
	\right.$

\item For $d_{\omega_B}(x,x')<\varepsilon^{\frac12}$, the following H{\"o}lder estimates hold
\begin{align*}
\frac{|f_B'(x)-f_B'(x')|}{\left(d_{\frac{1}{\varepsilon}\omega_B}(x,x')\right)^\alpha}\leq C \|f\|_{C^{0,\alpha}_{\delta-2,\tau-2,\lambda,\mu-2,\varepsilon}(M)}  \left\{
	\begin{aligned}
&\left(1+\varepsilon^{-\frac12}|t|\right)^{\delta-\tau}\; && \text{if}\quad x=t,x'=t', |t|<\epsilon_1;\\
& \varepsilon^{\frac{\tau-\delta}{2}} \; && \text{if}\quad x\in\left\{|t|> \frac12\epsilon_1,|y|> \frac12\epsilon_1\right\};\\
&\varepsilon^{\frac{\tau-\delta}{2}}|y|^{\frac{13}{24}} \; && \text{if}\quad x=y,x'=y',|y|<\epsilon_1;
\end{aligned}
	\right.
\end{align*}

\item $\|f''_B\|_{C^{0,\alpha}_{\delta-2,\tau-2,\lambda,\mu-2,\varepsilon}(M)}\le C|\overline{f_B}|\varepsilon^{\frac12(\delta-\tau)}\le C\max\left\{\epsilon_3^{\delta-\tau+2},\epsilon_4^{\frac{23}{288}}\right\}\|f\|_{C^{0,\alpha}_{\delta-2,\tau-2,\lambda,\mu-2,\varepsilon}(M)}$.

\end{itemize}
Here $C=C(\alpha,\epsilon_1,\delta,\tau,\lambda,\mu)$ is a uniform constant.

Define $P_B f$ to be the unique solution to the following Poisson equation on $\mathbb{C}=Y\backslash\{p_\infty\}$ that satisfies the exponential decay property:
\begin{equation}\label{base eq 2}
\Delta_{\frac{1}{\varepsilon}\omega_B} (P_B f) = f_B',\quad \left|\nabla_{\omega_B}^kP_Bf\right|_{\omega_B}\le C(k,\varepsilon)|y|^{\frac{13}{24}},\quad \forall k\in\{0,1,2\};
\end{equation}
or equivalently,
\begin{equation}\label{base eq 3}
\sqrt{-1} \partial \bar{\partial} (P_B f) = \frac{1}{2\varepsilon} f_B' \omega_B, \quad \left|\nabla_{\omega_B}^kP_Bf\right|_{\omega_B}\le C(k,\varepsilon)|y|^{\frac{13}{24}},\quad \forall k\in\{0,1,2\}.
\end{equation}
The existence and uniqueness of such solutions were established in \cite{lockhart1985elliptic}; see also Proposition~2.7 in \cite{haskins2015asymptotically}, which provides precisely the result needed here.

\begin{lemma}\label{approximate invert 3}

If $\epsilon_3$ and $\epsilon_4$ are chosen sufficiently small, then for sufficiently small $\varepsilon$ with respect to all previous choices, the following estimates hold:
\begin{align*}
\left\| \Delta_{\omega_\varepsilon} P_B f - f_B \right\|_{C^{0,\alpha}_{\delta-2,\tau-2,\lambda,\mu-2,\varepsilon}(M)} \le \frac{1}{100} \| f \|_{C^{0,\alpha}_{\delta-2,\tau-2,\lambda,\mu-2,\varepsilon}(M)},
\end{align*}
and
\begin{align*}
\left\| \partial \bar{\partial} (P_B f) \right\|_{C^{0,\alpha}_{\delta-2,\tau-2,\lambda,\mu-2,\varepsilon}(M)} \le C(\alpha,\delta, \tau,\lambda,\mu) \| f \|_{C^{0,\alpha}_{\delta-2,\tau-2,\lambda,\mu-2,\varepsilon}(M)}.
\end{align*}

\end{lemma}

\begin{proof}

By definition of $P_Bf$, we have
\begin{align*}
\left\| \partial \bar{\partial} (P_B f) \right\|_{C^{0,\alpha}_{\delta-2,\tau-2,\lambda,\mu-2,\varepsilon}(M)} &\le C \left\|f'_B\frac{1}{\varepsilon}\pi^*\omega_B \right\|_{C^{0,\alpha}_{\delta-2,\tau-2,\lambda,\mu-2,\varepsilon}(M)}\\
&\leq C \left\|f'_B \right\|_{C^{0,\alpha}_{\delta-2,\tau-2,\lambda,\mu-2,\varepsilon}(M)}\left\|\frac{1}{\varepsilon}\pi^*\omega_B \right\|_{C^{0,\alpha}_{0,0,0,0,\varepsilon}(M)}\\
&\le C\left\| f \right\|_{C^{0,\alpha}_{\delta-2,\tau-2,\lambda,\mu-2,\varepsilon}(M)},
\end{align*}where the second inequality follows from the numerical properties of the weights recorded in Remark~\ref{numerical properties of the weighted Holder spaces}, and the third inequality is implied by the weighted H{\"o}lder estimate on $f_B(x)$ and the uniform boundedness of $\left\| \frac{1}{\varepsilon} \pi^*\omega_B \right\|_{C^{0,\alpha}_{0,0,0,0,\varepsilon}(M)}$.

To estimate $\left\| \Delta_{\omega_\varepsilon} P_B f - f_B \right\|_{C^{0,\alpha}_{\delta-2,\tau-2,\lambda,\mu-2,\varepsilon}(M)}$, it suffices to bound $\left\| \Delta_{\omega_\varepsilon} P_B f - f_B' \right\|_{C^{0,\alpha}_{\delta-2,\tau-2,\lambda,\mu-2,\varepsilon}(M)}$, since the term $$\| f_B'' \|_{C^{0,\alpha}_{\delta-2,\tau-2,\lambda,\mu-2,\varepsilon}(M)} \le C\max\left\{\epsilon_3^{\delta-\tau+2},\epsilon_4^{\frac{23}{288}}\right\} \| f \|_{C^{0,\alpha}_{\delta-2,\tau-2,\lambda,\mu-2,\varepsilon}(M)}$$ can be made arbitrarily small by choosing sufficiently small $\epsilon_3$ and $\epsilon_4$. By $ \sqrt{-1} \partial \bar{\partial} (P_B f) \wedge \omega_\varepsilon^2 = \frac{1}{6} (\Delta_{\omega_\varepsilon} P_B f) \omega_\varepsilon^3$, we have 
\begin{align*}
\Delta_{\omega_\varepsilon} P_B f = f_B' \frac{3 \pi^*\omega_B \wedge \omega_\varepsilon^2}{\varepsilon \omega_\varepsilon^3}.
\end{align*}

In the region $\left\{ |t| \gtrsim \varepsilon^{\frac{6}{14+\tau}},|y|>\epsilon_1 \} \cup \{ |t| \le \varepsilon^{\frac{6}{14+\tau}}, r \gtrsim \varepsilon^{\frac{1}{10}} + \varepsilon^{\frac{1}{12}} \rho'^{\frac16} \right\}$ by Lemma 3.10 in \cite{li2019gluing}, one obtains the estimate in this region
\begin{align*}
\left\| \frac{3\pi^*\omega_B \wedge \omega_\varepsilon^2}{\varepsilon\omega_\varepsilon^3} - 1 \right\|_{C^{0,\alpha}_{0,0,0,0,\varepsilon}} \le C \varepsilon^{\frac{\tau+2}{2(14+\tau)}},
\end{align*}hence in this region
\begin{align*}
\left\| \Delta_{\omega_\varepsilon} P_B f - f_B' \right\|_{C^{0,\alpha}_{\delta-2,\tau-2,\lambda,\mu-2,\varepsilon}} \le C \varepsilon^{\frac{\tau+2}{2(14+\tau)}} \left\| f_B' \right\|_{C^{0,\alpha}_{\delta-2,\tau-2,\lambda,\mu-2,\varepsilon}} \le C \varepsilon^{\frac{\tau+2}{2(14+\tau)}} \left\| f \right\|_{C^{0,\alpha}_{\delta-2,\tau-2,\lambda,\mu-2,\varepsilon}}.
\end{align*}

In the region $\left\{|t| < \varepsilon^{\frac{6}{14+\tau}}, r < \varepsilon^{\frac{1}{10}} + \varepsilon^{\frac{1}{12}} \rho'^{\frac16}\right\}$, by Lemma 3.10 in \cite{li2019gluing}, we have a very coarse estimate
\begin{align*}
\left\| \frac{3\pi^*\omega_B \wedge \omega_\varepsilon^2}{\varepsilon\omega_\varepsilon^3} - 1 \right\|_{C^{0,\alpha}_{0,0,0,0,\varepsilon}} \le C,
\end{align*}
which implies
\begin{align*}
\left\| \Delta_{\omega_\varepsilon} P_B f - f_B' \right\|_{C^{0,\alpha}_{\delta-2,\tau-2,\lambda,\mu-2,\varepsilon}} \le C \left\| f \right\|_{C^{0,\alpha}_{\delta-2,\tau-2,\lambda,\mu-2,\varepsilon}(M)} \varepsilon^{\frac{2}{14+\tau}}
\end{align*}

In the region $\{|y|<\epsilon_1\}$, by exactly the same argument as in the proof of Proposition \ref{Ricci potential decay 1}, we have
\begin{align*}
\left\| \frac{3\pi^*\omega_B \wedge \omega_\varepsilon^2}{\varepsilon\omega_\varepsilon^3} - 1 \right\|_{C^{0,\alpha}_{0,0,0,0,\varepsilon}} \le C \varepsilon^{\frac23\delta+\frac{1}{60}},
\end{align*}hence in this region
\begin{align*}
\left\| \Delta_{\omega_\varepsilon} P_B f - f_B' \right\|_{C^{0,\alpha}_{\delta-2,\tau-2,\lambda,\mu-2,\varepsilon}} \le C \varepsilon^{\frac23\delta+\frac{1}{60}} \left\| f_B' \right\|_{C^{0,\alpha}_{\delta-2,\tau-2,\lambda,\mu-2,\varepsilon}} \le C \varepsilon^{\frac23\delta+\frac{1}{60}} \left\| f \right\|_{C^{0,\alpha}_{\delta-2,\tau-2,\lambda,\mu-2,\varepsilon}}.
\end{align*}

Now that all three terms are suppressed by a power of $\varepsilon$, they are negligible in the error estimate.

\end{proof}

We now sum over the pieces to define an approximate Green operator
\begin{equation}
Pf=P_{0,1}f+P_{0,2}f+P_1f+P_2f+P_Bf.
\end{equation}

Combining all the above arguments, we obtain the following result.

\begin{corollary}\label{approximate invert 4}

Let $\tau=-\frac23$, $-\frac{1}{40}<\delta < 0$, $\mu=-\frac{1}{12}$, $\lambda=\frac{13}{24}$, $0<\alpha<\frac{1}{12}$. Let $f$ be a function in $\left(C^{0,\alpha}_{\delta-2,\tau-2,\lambda,\mu-2,\varepsilon}(M)\right)_0$. Then one can choose $\epsilon_3,\epsilon_4, \Lambda_2, \Lambda_3$ such that
\begin{align*}
\left\|\Delta_{\omega_\varepsilon} P f - f\right\|_{C^{0,\alpha}_{\delta-2,\tau-2,\lambda,\mu-2,\varepsilon}(M)} \le \frac{1}{20} \|f\|_{C^{0,\alpha}_{\delta-2,\tau-2,\lambda,\mu-2,\varepsilon}(M)},
\end{align*}
and
\begin{align*}
\left\|\partial\bar{\partial} P f\right\|_{C^{0,\alpha}_{\delta-2,\tau-2,\lambda,\mu-2,\varepsilon}(M)} \le C \|f\|_{C^{0,\alpha}_{\delta-2,\tau-2,\lambda,\mu-2,\varepsilon}(M)},
\end{align*}where $C$ is a uniform constant.

\end{corollary}

\begin{proof}

This is essentially a consequence of Lemmas~\ref{approximate invert 2}, \ref{approximate invert 1}, and \ref{approximate invert 3}. The only subtlety lies in the order in which the parameters are chosen. We first select a sufficiently small $\epsilon_3$ and a sufficiently large $\Lambda_2$ to satisfy the requirements of Lemmas~\ref{approximate invert 2} and \ref{approximate invert 3}. Subsequently, one may simultaneously choose a sufficiently small $\epsilon_4$ and a sufficiently large $\Lambda_3$ to fulfill the conditions of Lemmas~\ref{approximate invert 1} and \ref{approximate invert 3}. Note that the smallness conditions imposed on $\epsilon_3$ and $\epsilon_4$ in Lemma~\ref{approximate invert 3} are independent of each other.

\end{proof}

\begin{remark}\label{exp decay prop of P}

By the construction of $P_2$ and $P_B$, we can obtain the following $\varepsilon$-dependent decay estimate. Let $f\in \left(C^{0,\alpha}_{\delta-2,\tau-2,\lambda,\mu-2,\varepsilon}(M)\right)_0$, then for points $x\in M$ with $|y(x)|\ll 1$, we have
\begin{align*}
&|Pf(x)|<\frac{C(\varepsilon)|y|^{\frac{13}{24}}}{r^{\frac{1}{12}}}<C(\varepsilon)|y|^{\frac{25}{48}};\\
&\left|\nabla_{\omega_\varepsilon}Pf(x)\right|_{\omega_\varepsilon}<\frac{C(\varepsilon)|y|^{\frac{13}{24}}}{r^{1+\frac{1}{12}}}<C(\varepsilon)|y|^{\frac{13}{48}};\\
&\left|\nabla^2_{\omega_\varepsilon}Pf(x)\right|_{\omega_\varepsilon}<\frac{C(\varepsilon)|y|^{\frac{13}{24}}}{r^{2+\frac{1}{12}}}<C(\varepsilon)|y|^{\frac{1}{48}}.
\end{align*}

\end{remark}

Finally, we prove Proposition \ref{inverse of Laplacian}.

\begin{proof}

Let $f$ be a function in $C^{0,\alpha}_{\delta-2,\tau-2,\lambda,\mu-2,\varepsilon}(M)$ satisfying $\int_M f\,\omega_\varepsilon^3 = 0$, and set
\begin{align*}
u \coloneqq P \sum_{j=0}^{\infty} \left(\operatorname{Id} - \Delta_{\omega_\varepsilon} P\right)^j f.
\end{align*}

We first verify that the operator $I - \Delta_{\omega_\varepsilon}P$ preserves the zero-integral subspace. Let $g \in \left( C^{0,\alpha}_{\delta-2,\tau-2,\lambda,\mu-2,\varepsilon}(M) \right)_0$, and set $v = Pg$. By the weighted estimate for $P$, together with Remark~\ref{exp decay prop of P}, near infinity we have
\begin{align*}
|v|\leq C(\varepsilon)|y|^{\frac{25}{48}},\quad \left|\nabla_{\omega_\varepsilon} v\right|_{\omega_\varepsilon}
\le C(\varepsilon)\,|y|^{\frac{13}{48}},\quad \left|\nabla_{\omega_\varepsilon}^2 v\right|_{\omega_\varepsilon}
\le C(\varepsilon)\,|y|^{\frac{1}{48}}.
\end{align*}
In particular, $|\Delta_{\omega_\varepsilon} v| \le C(\varepsilon)\,|y|^{\frac{1}{48}}$.

Since
\begin{align*}
\omega_\varepsilon^3
= \frac{3}{\varepsilon}\,\pi^*\omega_B \wedge \left(\omega_X + \sqrt{-1}\,\partial\bar\partial\phi\right)^2
  + \left(\omega_X + \sqrt{-1}\,\partial\bar\partial\phi\right)^3,
\end{align*}
the second term is integrable by Proposition~\ref{Well-definedness of the metric ansatz}, while
\begin{align*}
\int_{0<|y|<1} |y|^{\frac{1}{48}}\,\omega_B
\le C\int_0^1 \rho^{\frac{1}{48}}\,\frac{d\rho}{\rho} < \infty.
\end{align*}
Thus $\Delta_{\omega_\varepsilon} v \in L^1\left(M,\omega_\varepsilon^3\right)$. 

Let $M_R = \{x \in M : |y(x)| \ge R\}$, $\Sigma_R = \partial M_R = \{|y| = R\}$. Using $\sqrt{-1}\,\partial\bar\partial v \wedge \omega_\varepsilon^2
= \frac{1}{6}\,(\Delta_{\omega_\varepsilon} v)\,\omega_\varepsilon^3$ and Stokes' theorem, we obtain
\begin{align*}
\left| \int_{M_R} \Delta_{\omega_\varepsilon} v \;\omega_\varepsilon^3 \right|
= 6\,\left| \int_{\Sigma_R} \sqrt{-1}\,\bar\partial v \wedge \omega_\varepsilon^2 \right|.
\end{align*}
The hypersurfaces $\Sigma_R$ have uniformly bounded volume as $R \to 0$, because the base part of the metric $\omega_\varepsilon$ is asymptotically cylindrical and the fiber volume is fixed. Therefore,
\begin{align*}
\left| \int_{M_R} \Delta_{\omega_\varepsilon} v \;\omega_\varepsilon^3 \right|
\le C\,\sup_{\Sigma_R} |\nabla v|_{\omega_\varepsilon}\,
\operatorname{Vol}_{\omega_\varepsilon}(\Sigma_R)
\le C(\varepsilon)\,R^{\frac{13}{48}}.
\end{align*}
Letting $R \to 0$, and using the absolute integrability established above, gives
\begin{align*}
\int_M \Delta_{\omega_\varepsilon} v \;\omega_\varepsilon^3 = 0.
\end{align*}
This is exactly the boundary-flux calculation used later in Theorem~6.1.

Consequently,
\begin{align*}
\int_M \left(\operatorname{Id} - \Delta_{\omega_\varepsilon}P\right)g \;\omega_\varepsilon^3
= \int_M g \;\omega_\varepsilon^3 - \int_M \Delta_{\omega_\varepsilon} v \;\omega_\varepsilon^3 = 0.
\end{align*}
Hence
\begin{align*}
\left(\operatorname{Id} - \Delta_{\omega_\varepsilon}P\right)
\left( C^{0,\alpha}_{\delta-2,\tau-2,\lambda,\mu-2,\varepsilon}(M) \right)_0
\subset
\left( C^{0,\alpha}_{\delta-2,\tau-2,\lambda,\mu-2,\varepsilon}(M) \right)_0.
\end{align*}
It follows inductively that $\left(\operatorname{Id} - \Delta_{\omega_\varepsilon}P\right)^j g$ has zero integral for every $j \ge 0$, so every term in the Neumann iteration is well defined.

The series converges thanks to the completeness of $C^{0,\alpha}_{\delta-2,\tau-2,\lambda,\mu-2,\varepsilon}(M)$ and the estimate $\|\Delta_{\omega_\varepsilon} P - \operatorname{Id}\| \le \frac{1}{20}$.

Corollary~\ref{approximate invert 4} then gives
\[
\left\| \partial \bar{\partial} u \right\|_{C^{0,\alpha}_{\delta-2,\tau-2,\lambda,\mu-2,\varepsilon}}
\le C \,\left\| \partial \bar{\partial} P \right\|
\|f\|_{C^{0,\alpha}_{\delta-2,\tau-2,\lambda,\mu-2,\varepsilon}}
\le C \|f\|_{C^{0,\alpha}_{\delta-2,\tau-2,\lambda,\mu-2,\varepsilon}}.
\]

The $(1,1)$-form $\theta = 2\sqrt{-1}\,\partial \bar{\partial} u$ is $\partial \bar{\partial}$-exact and verifies $\operatorname{Tr}_{\omega_\varepsilon} \theta = \Delta_{\omega_\varepsilon} u = f$. Defining
\begin{equation}\label{eq:R}
\mathcal{R}(f) = \theta = 2\sqrt{-1}\,\partial \bar{\partial} P \sum_{j=0}^{\infty} \left(\operatorname{Id} - \Delta_{\omega_\varepsilon} P\right)^j f
\end{equation}therefore yields the desired right inverse $\mathcal{R}$, together with the required bounds.

\end{proof}

\begin{remark}\label{exp decay prop}

Note that the construction above also yields an $\varepsilon$-dependent bound on the solution $u=\Delta_{\omega_{\varepsilon}}^{-1}f$:
\begin{align*}
\left\|\Delta_{\omega_{\varepsilon}}^{-1}f\right\|_{C^{2,\alpha}_{\delta,\tau,\lambda,\mu,\varepsilon}}\leq C(\varepsilon) \left\|f\right\|_{C^{0,\alpha}_{\delta-2,\tau-2,\lambda,\mu-2,\varepsilon}}.
\end{align*}Hence for points $x\in M$ with $|y(x)|\ll 1$,  we have
\begin{align*}
&|u(x)|<\frac{C(\varepsilon)|y|^{\frac{13}{24}}}{r^{\frac{1}{12}}}<C(\varepsilon)|y|^{\frac{25}{48}};\\
&\left|\nabla_{\omega_\varepsilon}u(x)\right|_{\omega_\varepsilon}<\frac{C(\varepsilon)|y|^{\frac{13}{24}}}{r^{1+\frac{1}{12}}}<C(\varepsilon)|y|^{\frac{13}{48}};\\
&\left|\nabla^2_{\omega_\varepsilon}u(x)\right|_{\omega_\varepsilon}<\frac{C(\varepsilon)|y|^{\frac{13}{24}}}{r^{2+\frac{1}{12}}}<C(\varepsilon)|y|^{\frac{1}{48}}.
\end{align*}

\end{remark}

\section{Perturbation to a Calabi--Yau metric}\label{Perturbation to Calabi--Yau metric}

In this section, we apply the Banach fixed point theorem to perturb the ansatz $\omega_\varepsilon$ into a genuine Calabi--Yau metric, and we also analyze the geometric properties of the resulting Calabi--Yau metric. Theorem~\ref{pertubation argument} and Proposition~\ref{properties of final CY} below together constitute a complete proof of Theorem~\ref{main theorem}.

\begin{theorem}\label{pertubation argument}

We work under Setting~(A). Let $\tau = -\frac23$, $-\frac{1}{40} < \delta < 0$, $\mu = -\frac{1}{12}$, $\lambda = \frac{13}{24}$, and $0 < \alpha < \frac{1}{12}$. Then, for all $\varepsilon > 0$ sufficiently small, there exists a complete Calabi--Yau metric $\omega_{\operatorname{CY},\varepsilon}$ in the class $\left[ \omega_X|_M \right]$ satisfying $\omega_{\operatorname{CY},\varepsilon}^3 = \frac{3\sqrt{-1}}{\varepsilon}\,\Omega \wedge \overline{\Omega}$ and 
\begin{equation}
\|\omega_{\operatorname{CY},\varepsilon} - \omega_\varepsilon\|_{C^{0,\alpha}_{\delta-2,\tau-2,\lambda,\mu-2,\varepsilon}(M)} \le C \varepsilon^{\frac{1}{2}\delta + \frac{7}{20}}.
\end{equation}
In particular, the numerical property of the weights implies that
\begin{equation}
\|\omega_{\operatorname{CY,\varepsilon}} - \omega_\varepsilon\|_{C^{0,\alpha}_{0,0,0,0,\varepsilon}(M)} \le C\varepsilon^{\frac23\delta+\frac{1}{60}}.
\end{equation}
Here $C$ is a uniform constant independent of $\varepsilon$.

\end{theorem}

\begin{proof}

Consider the operator
\begin{align*}
\mathcal{N}_\varepsilon:\left(C^{0,\alpha}_{\delta-2,\tau-2,\lambda,\mu-2,\varepsilon}(M)\right)_0\to \left(C^{0,\alpha}_{\delta-2,\tau-2,\lambda,\mu-2,\varepsilon}(M)\right)_0,\;f\mapsto \frac{\left(\omega_\varepsilon+\mathcal{R}f\right)^3}{\omega_\varepsilon^3}-e^{-h_\varepsilon}.
\end{align*}

This map is well-defined. Indeed, for $f=0$, we have by Proposition \ref{Well-definedness of the metric ansatz}
\begin{align*}
\int_M\left(1-e^{-h_\varepsilon}\right)\omega_\varepsilon^3=\int_M\left(\omega_\varepsilon^3-\frac{3\sqrt{-1}\Omega\wedge \overline{\Omega}}{\varepsilon}\right)=0
\end{align*}By the exponential decay property with respect to the distance function in Remark \ref{exp decay prop}, for $f\in C^{0,\alpha}_{\delta-2,\tau-2,\lambda,\mu-2,\varepsilon}(M)$ we have 
\begin{align*}
\int_M\left(\frac{\left(\omega_\varepsilon+\mathcal{R}f\right)^3}{\omega_\varepsilon^3}-e^{-h_\varepsilon}\right)\omega_\varepsilon^3=\int_M\left(\omega_\varepsilon+\mathcal{R}f\right)^3-\omega_\varepsilon^3=0,
\end{align*}where the integral is understood in the Lebesgue sense. 

As an example, we verify that $\omega_\varepsilon^2 \wedge \mathcal{R}f$ is absolutely integrable on $M$ and that its integral vanishes; the remaining two terms can be treated similarly. Recall $\mathcal{R}f = 2\sqrt{-1}\partial \bar{\partial} u$. Then
\begin{align*}
\int_{\pi^{-1}(B(p_\infty,1)\setminus B(p_\infty,R))} \left|\omega_\varepsilon^2 \wedge \partial \bar{\partial} u\right|&
\le \int_{\pi^{-1}(B(p_\infty,1)\setminus B(p_\infty,R))} \left|\omega_\varepsilon^2 \wedge \partial \bar{\partial} u\right|_{\omega_\varepsilon} \,\omega_\varepsilon^3\\
&
\le \int_{\pi^{-1}(B(p_\infty,1)\setminus B(p_\infty,R))} \left|\partial \bar{\partial} u\right|_{\omega_\varepsilon} \,\omega_\varepsilon^3\\
&
\le C \int_{\pi^{-1}(B(p_\infty,1)\setminus B(p_\infty,R))} |y|^{\frac{1}{48}}\,\omega_\varepsilon^3.
\end{align*}
Recall that $\omega_\varepsilon^3 = \frac{3}{\varepsilon}\,\pi^*\omega_B \wedge (\omega_X + \sqrt{-1}\,\partial \bar{\partial} \phi)^2 + (\omega_X + \sqrt{-1}\,\partial \bar{\partial} \phi)^3$.
We have already shown in Proposition~\ref{Well-definedness of the metric ansatz} that $(\omega_X + \sqrt{-1}\,\partial \bar{\partial} \phi)^3$ is Lebesgue integrable on $M$, hence
\begin{align*}
\lim_{R \to 0} \int_{\pi^{-1}(B(p_\infty,1)\setminus B(p_\infty,R))} |y|^{\frac{1}{48}}\,|\omega_X + \sqrt{-1}\,\partial \bar{\partial} \phi|^3 < +\infty.
\end{align*}
On the other hand,
\begin{align*}
\int_{\pi^{-1}(B(p_\infty,1)\setminus B(p_\infty,R))} |y|^{\frac{1}{48}}\,\pi^*\omega_B \wedge (\omega_X + \sqrt{-1}\,\partial \bar{\partial} \phi)^2
\le C \int_{B(0,1)\setminus B(0,R)} \frac{\sqrt{-1}\,|y|^{\frac{1}{48}} \,dy \wedge d\bar{y}}{|y|^2},
\end{align*}
and consequently
\begin{align*}
\lim_{R \to 0} \int_{\pi^{-1}(B(p_\infty,1)\setminus B(p_\infty,R))} |y|^{\frac{1}{48}}\,\pi^*\omega_B \wedge (\omega_X + \sqrt{-1}\,\partial \bar{\partial} \phi)^2 < +\infty.
\end{align*}
Set $\Sigma_R = \{x \in M : |y(x)| = R\}$ and $M_R = \{x \in M : |y(x)| \ge R\}$. Integrating by parts gives
\begin{align*}
\left|\int_{M_R} \omega_\varepsilon^2 \wedge \partial \bar{\partial} u\right|
= \left|\int_{\Sigma_R} \omega_\varepsilon^2 \wedge \bar{\partial} u\right|
\le \int_{\Sigma_R} \left|\omega_\varepsilon^2 \wedge \bar{\partial} u\right|_{\omega_\varepsilon} \,dV_{\Sigma_R}
< C \int_{\Sigma_R} R^{\frac{13}{48}} \,dV_{\Sigma_R}
< C R^{\frac{13}{48}}.
\end{align*}
Letting $R \to 0$, we conclude $\int_M \omega_\varepsilon^2 \wedge \mathcal{R}f = 0$.

Our goal is to find a $f\in \left(C^{0,\alpha}_{\delta-2,\tau-2,\lambda,\mu-2,\varepsilon}(M)\right)_0$ such that $\mathcal{N}_\varepsilon(f)=0$. This is equivalent to finding a fixed point for the following map
\begin{align*}
\Phi_\varepsilon:\left(C^{0,\alpha}_{\delta-2,\tau-2,\lambda,\mu-2,\varepsilon}(M)\right)_0\to \left(C^{0,\alpha}_{\delta-2,\tau-2,\lambda,\mu-2,\varepsilon}(M)\right)_0,\;f\mapsto -\mathcal{N}_\varepsilon(0)-\mathcal{Q}_\varepsilon(f),
\end{align*}where $\mathcal{Q}_\varepsilon(f)$ is the nonlinear part of $\mathcal{N}_\varepsilon(f)$. More precisely, we have an expansion
\begin{align*}
\mathcal{N}_\varepsilon(f)=\mathcal{N}_\varepsilon(0)+d_0\mathcal{N}_\varepsilon(f)+\mathcal{Q}_\varepsilon(f)=\mathcal{N}_\varepsilon(0)+f+\mathcal{Q}_\varepsilon(f).
\end{align*}

Consider a closed ball of the Banach space $\left(C^{0,\alpha}_{\delta-2,\tau-2,\lambda,\mu-2,\varepsilon}(M)\right)_0$:
\begin{align*}
\mathcal{B}=\left\{f\in C^{0,\alpha}_{\delta-2,\tau-2,\lambda,\mu-2,\varepsilon}(M),\left\|f\right\|_{C^{0,\alpha}_{\delta-2,\tau-2,\lambda,\mu-2,\varepsilon}(M)}\leq\Lambda_4\varepsilon^{\frac12\delta+\frac{7}{20}}\left|\int_M f\omega_\varepsilon^3 = 0\right.\right\},
\end{align*}where $\Lambda_4$ is a large number to be chosen. Note that $\Phi_\varepsilon(0)=e^{-h_\varepsilon}-1$. Hence by Proposition \ref{Ricci potential decay 1}, there exists a uniform constant $C_0$ such that $$\left\|\Phi_\varepsilon(0)\right\|_{C^{0,\alpha}_{\delta-2,\tau-2,\lambda,\mu-2,\varepsilon}(M)}\leq C_0\varepsilon^{\frac12\delta+\frac{7}{20}}.$$ Choose $\Lambda_4>10C_0$ so that $\Phi_\varepsilon(0)\in \mathcal{B}$.

For any $f,g\in \mathcal{B}$, we have
\begin{align*}
\left\|\Phi_\varepsilon(f)-\Phi_\varepsilon(g)\right\|_{C^{0,\alpha}_{\delta-2,\tau-2,\lambda,\mu-2,\varepsilon}(M)}\leq&\left\|\frac{3\omega_\varepsilon\wedge \mathcal{R}(f+g)\wedge \mathcal{R}(f-g)}{\omega_\varepsilon^3}\right\|_{C^{0,\alpha}_{\delta-2,\tau-2,\lambda,\mu-2,\varepsilon}(M)}\\
&+\left\|\frac{\left\{\mathcal{R}f^2+\mathcal{R}f\wedge \mathcal{R}g+\mathcal{R}g^2\right\}\wedge \mathcal{R}(f-g)}{\omega_\varepsilon^3}\right\|_{C^{0,\alpha}_{\delta-2,\tau-2,\lambda,\mu-2,\varepsilon}(M)}\\
\leq &C\left\|\mathcal{R}\right\|\left\|f-g\right\|_{C^{0,\alpha}_{\delta-2,\tau-2,\lambda,\mu-2,\varepsilon}(M)}\Big\{\left\|\mathcal{R}(f+g)\right\|_{C^{0,\alpha}_{0,0,0,0,\varepsilon}(M)}\\
&+\left\|\mathcal{R}f\right\|^2_{C^{0,\alpha}_{0,0,0,0,\varepsilon}(M)}+\left\|\mathcal{R}g\right\|^2_{C^{0,\alpha}_{0,0,0,0,\varepsilon}(M)}\Big\}
\end{align*}
Note that for $h\in \mathcal{B}$, we have
\begin{align*}
\left\|\mathcal{R}h\right\|_{C^{0,\alpha}_{0,0,0,0,\varepsilon}(M)}\leq C\varepsilon^{\frac{\delta-2}{6}}\| h \|_{C^{0,\alpha}_{\delta-2,\tau-2,\lambda,\mu-2,\varepsilon}(M)}\leq C\varepsilon^{\frac23\delta+\frac{1}{60}},
\end{align*}thus
\begin{align*}
\left\|\Phi_\varepsilon(f)-\Phi_\varepsilon(g)\right\|_{C^{0,\alpha}_{\delta-2,\tau-2,\lambda,\mu-2,\varepsilon}(M)}\leq&C\left(\Lambda_4\varepsilon^{\frac23\delta+\frac{1}{60}}+\Lambda_4^2\varepsilon^{\frac43\delta+\frac{1}{30}}\right)\left\|f-g\right\|_{C^{0,\alpha}_{\delta-2,\tau-2,\lambda,\mu-2,\varepsilon}(M)}.
\end{align*}Choose $\varepsilon>0$ sufficiently small such that $C\left(\Lambda_4\varepsilon^{\frac23\delta+\frac{1}{60}}+\Lambda_4^2\varepsilon^{\frac43\delta+\frac{1}{30}}\right)<\frac{1}{10}$. Then $\Phi_\varepsilon$ is a contraction on the closed ball $\mathcal{B}$. By the Banach fixed point theorem, there exists a unique fixed point $f$ of the map $\Phi_\varepsilon$.

Since $\varepsilon$ is chosen to be sufficiently small, the perturbation solution $\omega_{\operatorname{CY},\varepsilon}=\omega_\varepsilon+\mathcal{R}f$ remains positive definite on $M$. Although $\omega_\varepsilon$ and $f$ are not smooth, the equation $\omega_{\operatorname{CY},\varepsilon}^3 = \frac{3\sqrt{-1}}{\varepsilon}\Omega \wedge \overline{\Omega}$, combined with elliptic regularity theory, forces $\omega_{\operatorname{CY},\varepsilon}=\omega_\varepsilon+\mathcal{R}f$ to be smooth. Thus $\omega_{\operatorname{CY},\varepsilon}$ is a genuine globally defined complete Calabi--Yau metric. Since $\mathcal{R}_f$ is $\partial\bar\partial$-exact and $[\omega_\varepsilon]=\left[\omega_X|_M\right]$, the resulting Calabi--Yau metric lies in the class $\left[\omega_X|_M\right]$.

\end{proof}

Next we analyze the geometric properties of the final Calabi--Yau metric $\omega_{\operatorname{CY},\varepsilon}$ and complete the proof of Theorem \ref{main theorem}.

Although Theorem \ref{pertubation argument} guarantees that $\omega_{\operatorname{CY},\varepsilon}$ is close to the metric ansatz $\omega_\varepsilon$, this closeness cannot be extended to higher-order derivatives because $\omega_\varepsilon$ itself is not smooth. Consequently, we cannot directly obtain information about the sectional curvature of $\omega_{\operatorname{CY},\varepsilon}$. Nevertheless, by employing a standard blow-up argument, we can at least prove that the sectional curvature remains unbounded.

\begin{proposition}\label{properties of final CY}

The complete Calabi--Yau metric $\omega_{\operatorname{CY},\varepsilon}$ given in Theorem {\ref{pertubation argument}} satisfies:

(a) The volume growth is linear;

(b) The tangent cone at infinity is $\left([0,\infty),dr^2\right)$;

(c) The sectional curvature is unbounded at infinity;  

(d) $(M,\omega_{\operatorname{CY},\varepsilon})$ is Gromov--Hausdorff asymptotic to $\mathbb{R}\times S^1\times X_\infty$ equipped with the product metric $\frac{\sqrt{-1}dy\wedge d\bar{y}}{2\varepsilon|y|^2}+\omega_{\operatorname{SRF}}|_{X_\infty}$.

\end{proposition}

\begin{proof}

(a), (b) By Theorem~\ref{pertubation argument}, we have
\begin{align*}
\left\|
\omega_{\operatorname{CY},\varepsilon} - \omega_\varepsilon
\right\|_{C^{0,\alpha}_{\delta-2,\tau-2,\lambda,\mu-2,\varepsilon}(M)}
\le C\varepsilon^{\frac12\delta + \frac7{20}}.
\end{align*}
By the definition of the weighted norm, this implies that, for some $c > 0$,
\begin{align*}
\left|
\omega_{\operatorname{CY},\varepsilon} - \omega_\varepsilon
\right|_{\omega_\varepsilon}(x)
\le C(\varepsilon)\,e^{-c\sqrt{\varepsilon}d_{\omega_\varepsilon}(x,x_0)}
\end{align*}
along the end of $M$. In particular, if $\eta(R) \coloneqq 
\sup\limits_{\{d_{\omega_\varepsilon}(x,x_0) \ge R\}}
\left|
\omega_{\operatorname{CY},\varepsilon} - \omega_\varepsilon
\right|_{\omega_\varepsilon}$, then $\eta(R) \to 0$ as $R \to \infty$, and hence $(1 - \eta(R))\,\omega_\varepsilon
\le \omega_{\operatorname{CY},\varepsilon}
\le (1 + \eta(R))\,\omega_\varepsilon$ outside $B_{\omega_\varepsilon}(x_0, R)$. Thus the two metrics are asymptotically bilipschitz with bilipschitz constant tending to $1$. Their distance functions differ by $o(r)$ on regions at distance $r$ from $x_0$, while their volume forms satisfy $\frac{\omega_{\operatorname{CY},\varepsilon}^{3}}{\omega_\varepsilon^{3}}
= 1 + o(1)$ uniformly at infinity. It follows that the corresponding metric balls have the same leading-order volume growth. Moreover, after rescaling by any sequence $r_i^{-2}$ with $r_i \to \infty$, the identity map between $\left(M, r_i^{-2}\omega_\varepsilon, x_0\right)$ and $\left(M, r_i^{-2}\omega_{\operatorname{CY},\varepsilon}, x_0\right)$ has distortion tending to zero on every bounded ball. Consequently, the two metrics have the same tangent cones at infinity. In particular, since $\omega_\varepsilon$ has linear volume growth and tangent cone $[0,\infty)$, the same conclusions hold for $\omega_{\operatorname{CY},\varepsilon}$.

(c) We finally prove that the sectional curvature of
$\omega_{\operatorname{CY},\varepsilon}$ is unbounded. 

Suppose, to the contrary, that $\sup\limits_M
\left|\operatorname{Rm}(\omega_{\operatorname{CY},\varepsilon})\right|_{\omega_{\operatorname{CY},\varepsilon}}
< \infty$.
Choose points $p_i \in U_\infty \setminus X_\infty$ lying on the vanishing
cycles $\{H = |y|\}$, with $y_i \coloneqq y(p_i) \to 0$, and consider the rescaled
metrics
\begin{align*}
\omega_i \coloneqq |y_i|^{-\frac{1}{2}}\,\omega_\varepsilon,\quad  \widehat\omega_i \coloneqq |y_i|^{-\frac{1}{2}}\,\omega_{\operatorname{CY},\varepsilon}.
\end{align*}
After pulling back by the natural holomorphic blow-up maps centered at
$p_i$, the construction of the ansatz gives $\omega_i \longrightarrow \omega_{\mathbb C} + \omega_{\operatorname{EH},1}$ smoothly on compact subsets of $\mathbb C \times V_1$.

The weighted estimate in Theorem~\ref{pertubation argument} implies that $\widehat\omega_i - \omega_i \to 0$ in $C^{0,\alpha}_{\mathrm{loc}}$. Moreover, in these rescaled coordinates the background metrics and their Monge--Amp\`ere densities converge smoothly to those of the model metric. The local higher-order estimates for the complex Monge--Amp\`ere equation
therefore upgrade this convergence to $\widehat\omega_i \longrightarrow \omega_{\mathbb C} + \omega_{\operatorname{EH},1}$ smoothly on compact subsets.

On the other hand, curvature scales according to $\left| \operatorname{Rm}(\widehat\omega_i)
\right|_{\widehat\omega_i} =|y_i|^{\frac{1}{2}}\,\left| \operatorname{Rm}(\omega_{\operatorname{CY},\varepsilon})
\right|_{\omega_{\operatorname{CY},\varepsilon}}$, and hence tends uniformly to zero on compact subsets. The smooth limit $\omega_{\mathbb C} + \omega_{\operatorname{EH},1}$ would therefore have to be flat. This is impossible, since the Eguchi--Hanson factor
$\omega_{\operatorname{EH},1}$ is nonflat. The contradiction proves that
the sectional curvature of $\omega_{\operatorname{CY},\varepsilon}$ is unbounded.

(d) Let $s = -\frac{1}{\sqrt{\varepsilon}}\log |y|$, $y = e^{-\sqrt{\varepsilon}s + \sqrt{-1}\theta}$, so that $s \to +\infty$ along the end. Let $g_\infty$ denote the Riemannian metric associated to the orbifold Ricci-flat K{\"a}hler form $\omega_{\mathrm{SRF}}|_{X_\infty}$, and equip $\mathbb{R} \times S^1 \times X_\infty$ 
with the product metric $g_{\operatorname{prod}} = ds^2 + \frac{1}{\varepsilon}\,d\theta^2 + g_\infty$. 

It suffices to show that, for every fixed $L > 0$, the translated slabs $\{T - L \le s \le T + L\}$ converge in the Gromov--Hausdorff sense to $[-L, L] \times S^1 \times X_\infty$ as $T \to \infty$.

We first prove this for the approximate metric $g_\varepsilon$. Fix $\vartheta > 0$, and restrict to the complement of the $\vartheta$-neighborhood of the nodal point of $X_\infty$. For $T$ sufficiently large, the trivialization $G_\infty$ is defined on this region, and we obtain a comparison map
\begin{align*}
    F_T(\sigma, \theta, x)
= G_\infty^{-1}\left( e^{-\sqrt{\varepsilon}(T+\sigma) + \sqrt{-1}\theta},\, x \right),
\quad
|\sigma| \le L,\quad r(x) \ge \vartheta.
\end{align*}
On every such fixed region, the metric converges uniformly:
\begin{align*}
F_T^* g_\varepsilon \longrightarrow d\sigma^2 + \frac{1}{\varepsilon}\,d\theta^2 + g_\infty.
\end{align*}
Indeed, the base part satisfies
\begin{align*}
\frac{\sqrt{-1}A_y}{\varepsilon|y|^2}\,dy \wedge d\bar{y}
\longrightarrow d\sigma^2 + \frac{1}{\varepsilon}\,d\theta^2
\end{align*}
because $A_y \to A_\infty = \frac12$. The vertical part converges smoothly to $g_\infty$ under $G_\infty$, while the mixed terms tend to zero, since $dy = y\left( -\sqrt{\varepsilon}\,d\sigma + \sqrt{-1}\,d\theta \right)$. Consequently,
\begin{align*}
\left| F_T^* g_\varepsilon - g_{\operatorname{prod}} \right|_{g_{\operatorname{prod}}}
= o_T(1)
\end{align*}
uniformly on $[-L, L] \times S^1 \times \{r \ge \vartheta\}$.

The regions omitted by $G_\infty$ have uniformly small diameter. More precisely, the cone structure of $g_\infty$ near the node and the Eguchi--Hanson scaling on the nearby smooth fibers imply
\begin{align*}
\operatorname{diam}_{g_\infty}\{r \le 2\vartheta\}
+ \sup_{|y| \ll 1}
\operatorname{diam}_{g_\varepsilon|_{X_y}}\{r \le 2\vartheta\}
\le C\vartheta.
\end{align*}

We now define a correspondence between the two slabs by taking the graph of $F_T$ on $\{r \ge \vartheta\}$, and by relating the two regions $\{r \le 2\vartheta\}$ over the same base point $(\sigma, \theta)$. The preceding metric convergence on the regular region, together with the diameter estimate for the omitted regions, gives
\begin{align*}
\operatorname{dis}(\mathcal R_{T,\vartheta})
\le o_T(1) + C_L\vartheta.
\end{align*}
Hence
\begin{align*}
\limsup_{T\to\infty}\,
d_{\mathrm{GH}}\left( \{T - L \le s \le T + L\},\,
[-L, L] \times S^1 \times X_\infty \right)
\le C_L\vartheta.
\end{align*}
Letting $\vartheta \to 0$, we conclude that $g_\varepsilon$ is Gromov--Hausdorff asymptotic to $g_{\operatorname{prod}}$.

Finally, Theorem~\ref{pertubation argument} and the pointwise decay in Remark~\ref{exp decay prop} imply
\begin{align*}
|g_{\operatorname{CY},\varepsilon} - g_\varepsilon|_{g_\varepsilon}
\le C(\varepsilon)\,|y|^{\frac{1}{48}}.
\end{align*}
Therefore, on every translated slab,
\begin{align*}
\sup_{\{T - L \le s \le T + L\}}
|g_{\operatorname{CY},\varepsilon} - g_\varepsilon|_{g_\varepsilon}
\longrightarrow 0
\end{align*}
as $T \to \infty$. The identity correspondence between the same slab equipped with $g_{\operatorname{CY},\varepsilon}$ and $g_\varepsilon$ consequently has distortion tending to zero. The triangle inequality for the Gromov--Hausdorff distance now proves that $(M, g_{\operatorname{CY},\varepsilon})$ is Gromov--Hausdorff asymptotic to $\left( \mathbb{R} \times S^1 \times X_\infty,\,
ds^2 + \tfrac{1}{\varepsilon}\,d\theta^2 + g_\infty \right)$.

\end{proof}

\section{Further discussion}\label{Further discussion}

\subsection{Existence without collapsing}Theorem~\ref{main theorem} restricts attention to constructing Calabi--Yau metrics $\omega_{\operatorname{CY},\varepsilon}$ in the class $\left[\omega_X|_M\right]$ that satisfy $\omega_{\operatorname{CY},\varepsilon}^3 = \frac{3\sqrt{-1}}{\varepsilon}\,\Omega \wedge \overline{\Omega}$ with $0 < \varepsilon \ll 1$. This restriction stems solely from the lack of a satisfactory nonlinear theory for the complex Monge--Amp\`ere equation on complete manifolds with unbounded sectional curvature. There is, however, no apparent geometric obstruction to the existence of such Calabi--Yau metrics for more general values of $\varepsilon > 0$. Indeed, if one merely wishes to invert the Laplacian without tracking the dependence of the inverse's norm on $\varepsilon$, the decomposition, patching, and parametrix method developed in this paper is unnecessary. By adapting the iteration argument of \cite{haskins2015asymptotically} to the Poisson equation, one can, for any $\varepsilon > 0$, invert the Laplacian of $\omega_\varepsilon$ in a weighted H\"older space analogous to the one defined on the region $U_3$---that is, a space depending only on the distance function on $\mathbb{C}$ and on the distance to the vanishing cycle. The essential difficulty lies in establishing a priori $C^2$ estimates, for which a uniform pointwise bound on the sectional curvature is typically required. It should be pointed out that in the example considered here, although the sectional curvature is indeed unbounded, the region where this unboundedness occurs occupies only a very small portion of the manifold. The author therefore believes that, if a more robust nonlinear theory could be developed, it should be possible to show that for any $\varepsilon > 0$, the class $\left[\omega_X|_M\right]$ admits a unique complete noncompact Calabi--Yau metric $\omega_{\operatorname{CY},\varepsilon}$ satisfying $\omega_{\operatorname{CY},\varepsilon}^3 = \frac{3\sqrt{-1}}{\varepsilon}\,\Omega \wedge \overline{\Omega}$, such that near infinity $\omega_{\operatorname{CY},\varepsilon}$ and $\omega_\varepsilon$ are close in a suitable weighted H\"older space.

\subsection{More examples}The Lefschetz K3 fibrations considered in this paper are generic from the algebro-geometric viewpoint; see \cite[Proposition~2.14]{kovalev2005coassociative} and Proposition~\ref{abundance of examples}. It is nevertheless natural to expect analogous constructions for fibrations whose singular fibers have more general isolated singularities. For simplicity, suppose that every singular fiber contains at most one isolated normal singular point and that, near such a point, the fibration is locally modeled by
\begin{align*}
\left\{(z,t)\in \mathbb{C}^{3}\times\mathbb{C}:f(z)=t\right\}  \longrightarrow \mathbb{C}, \;
 (z,t)\longmapsto t,
\end{align*}
where $f$ is weighted homogeneous: 
\begin{align*}
 f\left(\lambda^{a_1}z_1,\lambda^{a_2}z_2,\lambda^{a_3}z_3\right)
 =
 \lambda^{d}f(z_1,z_2,z_3).
\end{align*}
Writing $w_i=\frac{a_i}{d}$, we restrict attention to the case $ \sum\limits_{i=1}^{3} w_i>1$, which corresponds to klt singularities.

Let $\Omega=dt\wedge\Omega_t$ near a singular fiber in the finite region, and set $ A_t\coloneqq\int_{X_t}\Omega_t\wedge\overline{\Omega_t}$. Standard Hodge-theoretic arguments then show that
\begin{equation}\label{generalized KE}
A_t = A_0 + C|t|^{2(\sum w_i-1)}+\text{higher order terms},
\end{equation}
where $A_0>0$ and $C\neq 0$.

Near each singular fiber in the finite part, one may glue the exotic model metrics constructed in \cite{Conlon2021Newexamples}   or \cite{szekelyhidi2019degenerations} to the semi-Ricci-flat metric.  Near the removed fiber, one can still work with the semi-Ricci-flat approximation. The base term continues to approximate $\frac{\sqrt{-1}}{2|y|^2}\,dy \wedge d\bar{y}$, while in the fiber directions the vanishing cycles keep shrinking as $|y|\to 0$. Consequently, the overall geometry remains close to the one analyzed here. The difference in the asymptotic geometry becomes visible only after taking a sequence of points on the vanishing cycle and passing to the pointed Gromov--Hausdorff limit.

In this setting, one no longer has an explicit model analogous to the Stenzel metric, so the geometric analysis of the metric may become more difficult. The harmonic analysis near the model metric may also require minor adjustments. Nevertheless, the overall strategy should remain consistent with the case treated in this paper. The author therefore believes that the example studied here is not an isolated phenomenon, but rather one instance of a more general picture. With additional effort, the construction may also extend to higher-dimensional fibrations.

\appendix

\section{Abundance of examples}\label{Abundance of examples}

We give a concrete construction of a threefold satisfying the geometric hypotheses of Theorem~\ref{main theorem}. For simplicity, we consider a pencil of quartic surfaces in $\mathbb{P}^3$.

\begin{proposition}\label{abundance of examples}
Let $V\coloneqq H^0\left(\mathbb{P}^3,\mathcal{O}_{\mathbb{P}^3}(4)\right)$, so that $\mathbb{P}(V)\cong\mathbb{P}^{34}$ parametrizes quartic surfaces in
$\mathbb{P}^3$. Choose two sufficiently general quartic forms
$Q_0,Q_1\in V$, and let $L\coloneqq\left\{[sQ_0+tQ_1]:[s:t]\in\mathbb{P}^1\right\}\subset \mathbb{P}(V)$ be the corresponding pencil. Let $C\coloneqq V(Q_0)\cap V(Q_1)\subset\mathbb{P}^3$ be its base locus, and let $\sigma:X\coloneqq \operatorname{Bl}_C(\mathbb{P}^3)\longrightarrow\mathbb{P}^3$ be the blow-up along $C$. Then the pencil induces a morphism $\pi:X\longrightarrow\mathbb{P}^1$ with the following properties:
\begin{enumerate}
    \item $X$ is a smooth projective threefold and $\chi(X)=-60$;
 
    \item a general fiber of $\pi$ is a smooth K3 surface;
    
    \item $K_X=-[F]$, where $F$ denotes a fiber;
    
    \item $\pi$ is a Lefschetz fibration, and every singular fiber contains exactly one ordinary double point;
    
    \item $\pi$ has precisely $108$ singular fibers.
    
\end{enumerate}

\end{proposition}

\begin{proof}

Since $Q_0$ and $Q_1$ are sufficiently general, their zero loci intersect
transversely. Thus $C$ is a smooth complete-intersection curve of type
$(4,4)$. In particular, $\deg C=16$, $2g(C)-2=\deg K_C
    =\deg\left(\mathcal{O}_C(4)\right)=64$, and hence $g(C)=33$ and $\chi(C)=-64$. Since the center of the blow-up is
smooth, $X$ is smooth and projective. If $E=\mathbb{P}\left(N_{C/\mathbb{P}^3}\right)$ denotes the exceptional divisor, then $E$ is a $\mathbb{P}^1$-bundle over
$C$, so
\begin{align*}
\chi(X)=\chi(\mathbb{P}^3)-\chi(C)+\chi(E)=4-\chi(C)+2\chi(C)=-60.    
\end{align*}

The rational map $ [Q_1:-Q_0]:\mathbb{P}^3\dashrightarrow\mathbb{P}^1$ has indeterminacy locus precisely $C$ and is resolved by $\sigma$, giving the morphism $\pi$. Its fiber over $[s:t]$ is the strict transform of $S_{[s:t]}\coloneqq   V(sQ_0+tQ_1)\subset\mathbb{P}^3$. Since $C$ is a Cartier divisor on every member of the pencil, blowing up $S_{[s:t]}$ along $C$ does not change it. Hence $\pi^{-1}([s:t])\cong S_{[s:t]}$. For general $[s:t]$, the surface $S_{[s:t]}$ is a smooth quartic surface, and therefore a K3 surface by adjunction.

Let $\widetilde H=\sigma^*H$, where $H$ is the hyperplane class on
$\mathbb{P}^3$. The blow-up formula gives
\begin{align*}
    K_X=\sigma^*K_{\mathbb{P}^3}+E=-4\widetilde H+E,
\end{align*}
whereas the class of a fiber is
\begin{align*}
    [F]=4\widetilde H-E.
\end{align*}
Thus $K_X=-[F]$.

It remains to describe the singular fibers. Let $\Delta\subset\mathbb{P}(V)$ be the discriminant hypersurface parametrizing singular quartic surfaces. Consider the incidence variety $$\mathcal{I}\coloneqq \left\{([Q],p)\in\mathbb{P}(V)\times\mathbb{P}^3:     dQ(p)=0\right\}.$$ For a fixed $p$, the condition $dQ(p)=0$ imposes four independent linear conditions on $Q$; Euler's identity then implies $Q(p)=0$. Thus $\mathcal{I}$ is a projective bundle over $\mathbb{P}^3$, and is therefore irreducible of dimension $33$. Its image under the first projection is precisely $\Delta$, so $\Delta$ is an irreducible hypersurface.

The discriminant of degree-$d$ hypersurfaces in $\mathbb{P}^n$ has degree $(n+1)(d-1)^n$. Hence, for quartic surfaces in $\mathbb{P}^3$, $\deg\Delta=4(4-1)^3=108$. A point of $\Delta$ is smooth precisely when the corresponding quartic has a unique singular point and the Hessian at that point is nondegenerate, namely, when the singularity is an ordinary double point.
Quartics with at least two singular points, or with degenerate Hessian at
a singular point, form subvarieties of codimension at least two in $\mathbb{P}(V)$.

 Consequently, a general line $L\subset\mathbb{P}(V)$ avoids
$\operatorname{Sing}(\Delta)$ and meets $\Delta$ transversely. It follows
that every singular member of the pencil has exactly one node and that
\begin{align*}
    \#(L\cap\Delta)=L\cdot\Delta=108.
\end{align*}
Notice also that no singular point of a pencil member lies on $C$: indeed,
$dQ_0$ and $dQ_1$ are linearly independent along $C$, since $C$ is smooth.
Thus the blow-up does not modify the nodal points. Transversality of
$L$ to $\Delta$ implies that the corresponding critical points of $\pi$
are nondegenerate; equivalently, by the holomorphic Morse lemma, $\pi$ is
locally of the form
\begin{align*}
    (z_1,z_2,z_3)\longmapsto z_1^2+z_2^2+z_3^2.
\end{align*}
Hence $\pi$ is a Lefschetz fibration with exactly $108$ singular fibers.

\end{proof}

\bibliographystyle{plain}
\bibliography{reference}

@article{tian1990complete,
  title={Complete {K}{\"a}hler manifolds with zero {R}icci curvature {I}},
  author={Tian, G. and Yau, S.T.},
  journal={Journal of the American Mathematical Society},
  volume={3},
  number={3},
  pages={579--609},
  year={1990},
  publisher={JSTOR}
}

@article{yau1978ricci,
  title={On the {R}icci curvature of a compact {K}{\"a}hler manifold and the complex {M}onge-{A}mp{\`e}re equation, {I}},
  author={Yau, S.T.},
  journal={Communications on Pure and Applied Mathematics},
  volume={31},
  number={3},
  pages={339--411},
  year={1978},
  publisher={Wiley Online Library}
}

@book{hein2010gravitational,
  title={On gravitational instantons},
  author={Hein, H.-J.},
  year={2010},
  publisher={Princeton University}
}

@article{wang2022ricci,
  title={Ricci-flat manifolds of generalized {ALG} asymptotics},
  author={Wang, Y.},
  journal={arXiv:2212.11267},
  year={2022}
}

@article{haskins2015asymptotically,
  title={Asymptotically cylindrical {C}alabi--{Y}au manifolds},
  author={Haskins, M. and Hein, H-J. and Nordstr{\"o}m, J.},
  journal={Journal of Differential Geometry},
  volume={101},
  number={2},
  pages={213--265},
  year={2015},
  publisher={Lehigh University}
}

@article{li2019gluing,
  title={A gluing construction of collapsing {C}alabi--{Y}au metrics on {K}3 fibred 3-folds},
  author={Li, Y.},
  journal={Geometric and Functional Analysis},
  volume={29},
  number={4},
  pages={1002--1047},
  year={2019},
  publisher={Springer}
}

@article{spotti2014deformations,
  title={Deformations of nodal {K}{\"a}hler--{E}instein {D}el {P}ezzo surfaces with discrete automorphism groups},
  author={Spotti, C.},
  journal={Journal of the London Mathematical Society},
  volume={89},
  number={2},
  pages={539--558},
  year={2014},
  publisher={Oxford University Press}
}

@article{li2019new,
  title={A new complete {C}alabi--{Y}au metric on {$\mathbb{C}^3$}},
  author={Li, Y.},
  journal={Inventiones Mathematicae},
  volume={217},
  number={1},
  pages={1--34},
  year={2019},
  publisher={Springer}
}

@book{han2011elliptic,
  title={Elliptic partial differential equations},
  author={Han, Q. and Lin, F.},
  volume={1},
  year={2011},
  publisher={American Mathematical Soc.}
}

@article{szekelyhidi2019degenerations,
  title={Degenerations of {$\mathbb{C}^n$} and {C}alabi--{Y}au metrics},
  author={Sz{\'e}kelyhidi, G.},
  journal={Duke Mathematical Journal},
  volume={168},
  number={14},
  pages={2651--2700},
  year={2019},
  publisher={Duke University Press}
}

@article{lockhart1985elliptic,
  title={Elliptic differential operators on noncompact manifolds},
  author={Lockhart, R.B. and McOwen, R.C.},
  journal={Annali della Scuola Normale Superiore di Pisa-Classe di Scienze},
  volume={12},
  number={3},
  pages={409--447},
  year={1985}
}

@article{gibbons1978gravitational,
  title={Gravitational multi-instantons},
  author={Gibbons, G.W. and Hawking, S.W.},
  journal={Physics Letters B},
  volume={78},
  number={4},
  pages={430--432},
  year={1978},
  publisher={Elsevier}
}

@article{liang2025complete,
  title={Complete {C}alabi-{Y}au metrics on noncompact abelian fibered threefolds},
  author={Liang, R. and Zhang, Y.},
  journal={arXiv:2501.15205},
  year={2025}
}

@article{anderson1989complete,
  title={Complete {R}icci-flat {K}{\"a}hler manifolds of infinite topological type},
  author={Anderson, M.T. and Kronheimer, P.B. and LeBrun, C.},
  journal={Communications in Mathematical Physics},
  volume={125},
  number={4},
  pages={637--642},
  year={1989},
  publisher={Springer}
}

@article{gross2000large,
  title={Large complex structure limits of {K}3 surfaces},
  author={Gross, M. and Wilson, P.M.H.},
  journal={Journal of Differential Geometry},
  volume={55},
  number={3},
  pages={475--546},
  year={2000},
  publisher={Lehigh University}
}

@article{kovalev2005coassociative,
  title={Coassociative {K}3 fibrations of compact {$G_2$}-manifolds},
  author={Kovalev, A.},
  journal={arXiv:math/0511150},
  year={2005}
}

@article{yan2025uniqueness,
  title={Uniqueness of the asymptotic limits for {R}icci-flat manifolds with linear volume growth},
  author={Yan, Z. and Zhu, X.},
  journal={arXiv:2510.00420},
  year={2025}
}

@article{eyssidieux2009singular,
  title={Singular {K}{\"a}hler-{E}instein metrics},
  author={Eyssidieux, P. and Guedj, V. and Zeriahi, A.},
  journal={Journal of the American Mathematical Society},
  volume={22},
  number={3},
  pages={607--639},
  year={2009}
}

@incollection {donaldson2010calabi,
    AUTHOR = {Donaldson, S. K.},
     TITLE = {Calabi-{Y}au metrics on {K}ummer surfaces as a model gluing
              problem},
 BOOKTITLE = {Advances in geometric analysis},
    SERIES = {Adv. Lect. Math. (ALM)},
    VOLUME = {21},
     PAGES = {109--118},
 PUBLISHER = {Int. Press, Somerville, MA},
      YEAR = {2012},
      ISBN = {978-1-57146-248-0},
   MRCLASS = {32Q25 (14J28 14J32 32J15 53C55 58J05)},
  MRNUMBER = {3077251},
MRREVIEWER = {Sergiy\ Koshkin},
       DOI = {10.14219/jada.archive.2012.0111},
       URL = {https://doi.org/10.14219/jada.archive.2012.0111},
}

@article {Biquard2011Kummer,
    AUTHOR = {Biquard, O. and Minerbe, V.},
     TITLE = {A {K}ummer construction for gravitational instantons},
   JOURNAL = {Communications in Mathematical Physics},
  FJOURNAL = {Communications in Mathematical Physics},
    VOLUME = {308},
      YEAR = {2011},
    NUMBER = {3},
     PAGES = {773--794},
      ISSN = {0010-3616,1432-0916},
   MRCLASS = {53C26 (53C80)},
  MRNUMBER = {2855540},
MRREVIEWER = {Derek\ G.\ Harland},
       DOI = {10.1007/s00220-011-1366-y},
       URL = {https://doi.org/10.1007/s00220-011-1366-y},
}

@article {Conlon2021Newexamples,
    AUTHOR = {Conlon, R. J. and Rochon, F.},
     TITLE = {New examples of complete {C}alabi-{Y}au metrics on {$\mathbb
              C^n$} for {$n \geq 3$}},
   JOURNAL = {Annales Scientifiques de l'\'Ecole Normale Sup\'erieure.
              Quatri\`eme S\'erie (4)},
  FJOURNAL = {Annales Scientifiques de l'\'Ecole Normale Sup\'erieure.
              Quatri\`eme S\'erie},
    VOLUME = {54},
      YEAR = {2021},
    NUMBER = {2},
     PAGES = {259--303},
      ISSN = {0012-9593,1873-2151},
   MRCLASS = {32Q25 (14J32 53C55)},
  MRNUMBER = {4258163},
       DOI = {10.24033/asens.2459},
       URL = {https://doi.org/10.24033/asens.2459},
}

@article {Li2021Oncollapsing,
    AUTHOR = {Li, Y.},
     TITLE = {On collapsing {C}alabi-{Y}au fibrations},
   JOURNAL = {Journal of Differential Geometry},
  FJOURNAL = {Journal of Differential Geometry},
    VOLUME = {117},
      YEAR = {2021},
    NUMBER = {3},
     PAGES = {451--483},
      ISSN = {0022-040X,1945-743X},
   MRCLASS = {53C55 (32Q25 53C25)},
  MRNUMBER = {4255068},
MRREVIEWER = {Yanir\ A.\ Rubinstein},
       DOI = {10.4310/jdg/1615487004},
       URL = {https://doi.org/10.4310/jdg/1615487004},
}

@incollection {Li1980Estimates,
    AUTHOR = {Li, P. and Yau, S. T.},
     TITLE = {Estimates of eigenvalues of a compact {R}iemannian manifold},
 BOOKTITLE = {Proceedings of the Symposium in Pure Mathematics},
      VOLUME = {XXXVI},
     PAGES = {205--239},
 PUBLISHER = {American Mathematical Society, Providence, RI},
      YEAR = {1980},
      ISBN = {0-8218-1439-7},
   MRCLASS = {58G25 (53C20)},
  MRNUMBER = {573435},
MRREVIEWER = {P.\ G\"unther},
}


\end{document}